\documentclass[12pt]{article}

\usepackage{amssymb}
\usepackage{amsmath}
\usepackage{amsbsy}
\usepackage{amscd}
\usepackage{amsfonts}
\usepackage{amsthm}
\usepackage{mathrsfs}
\usepackage{verbatim}
\usepackage[colorlinks]{hyperref}
\usepackage{fullpage}
\usepackage{mathdots}
\usepackage{graphicx,subfigure}
\usepackage[english]{babel}
\usepackage[utf8]{inputenc}
\usepackage{tikz}

\usetikzlibrary{backgrounds,fit, matrix}
\usetikzlibrary{positioning}
\usetikzlibrary{calc,through,chains}
\usetikzlibrary{arrows,shapes,snakes,automata, petri}

\usepackage{authblk}

\newtheorem{Theorem}{Theorem}[section]
\newtheorem{Corollary}{Corollary}[section]
\newtheorem{Lemma}{Lemma}[section]
\newtheorem{Proposition}{Proposition}[section]
\newtheorem{Example}{Example}[section]

\newtheorem{Remark}{Remark}[section]
\newtheorem{Definition}{Definition}[section]

\newtheorem{Conjecture}{Conjecture}[section]

\newcommand{\C}{{\mathbb C}}

\newcommand{\Q}{{\mathbb Q}}
\newcommand{\R}{{\mathbb R}}

\newcommand{\N}{{\mathbb N}}

\newcommand{\mb}[1]{\mathbb{ #1}}
\newcommand{\mr}[1]{\mathbf {#1}}
\newcommand{\mc}[1]{\mathcal{#1}}

\newcommand{\mt}[1]{\text{#1}}

\newcommand*\numcircledtikz[1]{\tikz[baseline=(char.base)]
{\node[shape=circle,draw,inner sep=1.2pt,thick] (char) {#1};}}

\begin{document}

\title{Loop-augmented forests and a variant of the Foulkes' conjecture}

\author[1]{Mahir Bilen Can}
\author[2]{Jeff Remmel}

\affil[1]{{\small Tulane University, New Orleans LA; mahirbilencan@gmail.com}}
\affil[2]{{\small University of California, San Diego, La Jolla CA, jremmel@ucsd.edu}}

\normalsize

\date{\today}
\maketitle

\begin{abstract}

A loop-augmented forest is a labeled rooted forest with loops on some of its roots. By exploiting an 
interplay between nilpotent partial functions and labeled rooted forests, we investigate the permutation 
action of the symmetric group on loop-augmented forests. Furthermore, we describe an extension of 
the Foulkes' conjecture and prove a special case. Among other important outcomes of our analysis 
are a complete description of the stabilizer subgroup of an idempotent in the semigroup of partial 
transformations and a generalization of the (Knuth-Sagan) hook length formula.

\vspace{.5cm}
\noindent 
\textbf{Keywords:} Labeled rooted forests, symmetric group, plethysm.\\ 
\noindent 
\textbf{MSC:} 20C30, 05E10, 16W22
\end{abstract}

\section{Introduction}\label{S:Introduction}

This paper is a natural continuation of our earlier work~\cite{Can17}, where we considered a permutation 
action of the symmetric group on the set of labeled rooted forests. The origin of this representation goes 
back to the semigroup of partial transformations on a finite set. 
We shall show in this paper that the circle of ideas here naturally leads to a new variant of a long standing conjecture in representation 
theory, namely the Foulkes' conjecture. Let us recall this famous conjecture. If $m\leq n$, then the multiplicity 
of an irreducible representation $V$ of $\mc{S}_{mn}$ in the induced trivial representation $\mt{Ind}_{\mc{S
}_m \wr \mc{S}_n}^{\mc{S}_{mn}} 1$ is less than or equal to the multiplicity of $V$ in $\mt{Ind}_{\mc{S}_n\wr 
\mc{S}_m}^{\mc{S}_{mn}} 1$. We state this conjecture in a combinatorial way, which can be viewed as a 
precursor of our work.

Consider a set of vertices with $mn$ elements which are placed along a line and grouped into blocks of $m$ 
elements so that there are $n$ blocks in total. We will distribute the labels $1,2,\dots, mn$ on these vertices 
and call two resulting distributions equivalent if one configuration can be obtained from the other by permuting 
the labels within the same group and/or by permuting the blocks. For example, the three configurations in the
first line of Figure~\ref{F:ab}, where $n=3$ and $m=2$, are in the same equivalence class. (For convenience 
we separated individual blocks in each configuration by putting a bar between them.) In the second line of 
Figure~\ref{F:ab}, we have two inequivalent configurations since the set of elements of the first block of the 
first configuration is not equal to any set of elements of any blocks in the second configuration.
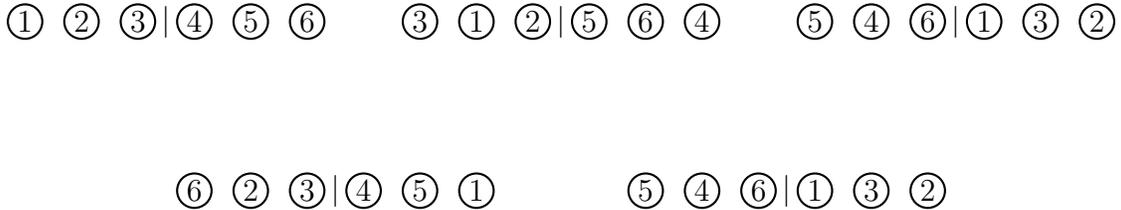
\begin{figure}[htp]
\centering
\begin{tikzpicture}[scale=.75]
\begin{scope}[xshift=-7cm]
\node at (.5,0) {$\vert$};
\node at (-2,0) {$\numcircledtikz{1}$};
\node at (-1,0) {$\numcircledtikz{2}$};
\node at (0,0) {$\numcircledtikz{3}$};
\node at (1,0){$\numcircledtikz{4}$};
\node at (2,0) {$\numcircledtikz{5}$};
\node at (3,0) {$\numcircledtikz{6}$};
\end{scope}
\begin{scope}
\node at (-2,0) {$\numcircledtikz{3}$};
\node at (-1,0) {$\numcircledtikz{1}$};
\node at (0,0) {$\numcircledtikz{2}$};
\node at (1,0){$\numcircledtikz{5}$};
\node at (2,0) {$\numcircledtikz{6}$};
\node at (3,0) {$\numcircledtikz{4}$};
\node at (.5,0) {$\vert$};
\end{scope}
\begin{scope}[xshift=7cm]
\node at (.5,0) {$\vert$};
\node at (-2,0) {$\numcircledtikz{5}$};
\node at (-1,0) {$\numcircledtikz{4}$};
\node at (0,0) {$\numcircledtikz{6}$};
\node at (1,0){$\numcircledtikz{1}$};
\node at (2,0) {$\numcircledtikz{3}$};
\node at (3,0) {$\numcircledtikz{2}$};
\end{scope}

\begin{scope}[yshift=-3cm]
\begin{scope}[xshift=-4cm]
\node at (.5,0) {$\vert$};
\node at (-2,0) {$\numcircledtikz{6}$};
\node at (-1,0) {$\numcircledtikz{2}$};
\node at (0,0) {$\numcircledtikz{3}$};
\node at (1,0){$\numcircledtikz{4}$};
\node at (2,0) {$\numcircledtikz{5}$};
\node at (3,0) {$\numcircledtikz{1}$};
\end{scope}

\begin{scope}[xshift=4cm]
\node at (.5,0) {$\vert$};
\node at (-2,0) {$\numcircledtikz{5}$};
\node at (-1,0) {$\numcircledtikz{4}$};
\node at (0,0) {$\numcircledtikz{6}$};
\node at (1,0){$\numcircledtikz{1}$};
\node at (2,0) {$\numcircledtikz{3}$};
\node at (3,0) {$\numcircledtikz{2}$};
\end{scope}
\end{scope}

\end{tikzpicture}
\caption{Three equivalent (in the first line), two inequivalent (in the second line) configurations.}
\label{F:ab}
\end{figure}

Let us denote by $o_1'$ the vector space spanned by all such equivalence classes. Next, we define another 
vector space, denoted by $o_1$, by changing the roles of $m$ and $n$. Clearly both of $o_1$ and $o_1'$ 
are $\mc{S}_{mn}$-modules. In this terminology, the Foulkes' conjecture states that the multiplicity of an 
irreducible $\mc{S}_{mn}$-module $V$ in $o_1'$ is less than or equal to the multiplicity of $V$ in $o_1$.

Our conjecture is the following natural extension of Foulkes' conjecture. 
Let $m$ and $n$ be positive integers such that $n> m \geq 2$. Consider $m$ copies of the labeled rooted 
tree on $n$ vertices where the root has exactly $n-1$ children. See Figure~\ref{F:ab2}, where $n=4$ and 
$m=2$. 
\begin{figure}[htp]
\centering
\begin{tikzpicture}[scale=.75]
\begin{scope}
\node at (-2,-1.65) {$\numcircledtikz{1}$};
\node at (-3,0) {$\numcircledtikz{2}$};
\node at (-2,0) {$\numcircledtikz{3}$};
\node at (-1,0) {$\numcircledtikz{4}$};
\node at (2,-1.65) {$\numcircledtikz{5}$};
\node at (1,0){$\numcircledtikz{6}$};
\node at (2,0) {$\numcircledtikz{7}$};
\node at (3,0) {$\numcircledtikz{8}$};
\draw[-, thick] (-2,-1.3) to  (-2,-.35);
\draw[-, thick] (-2,-1.3) to  (-2.85,-.3);
\draw[-, thick] (-2,-1.3) to  (-1.2,-.3);

\draw[-, thick] (2,-1.3) to  (2,-.35);
\draw[-, thick] (2,-1.3) to  (2.85,-.3);
\draw[-, thick] (2,-1.3) to  (1.2,-.3);
\end{scope}
\end{tikzpicture}
\caption{An extension of Foulkes' conjecture, combinatorially.}
\label{F:ab2}
\end{figure}
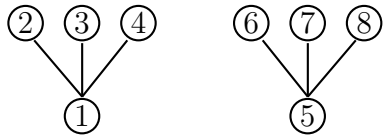

The symmetric group $\mc{S}_{mn}$ acts on vertices without breaking the relations between labels imposed 
by the edges of the underlying trees. Hence we get a representation denoted by $o_2'$. In a similar manner, 
next, we consider $n$ copies of labeled rooted tree on $m$ vertices where the root has exactly $m-1$ 
children. This time the $\mc{S}_{mn}$ representation is denoted by $o_2$. Now, let $V=V(\lambda)$ be an 
irreducible representation of $\mc{S}_{mn}$ indexed by a partition $\lambda$ which has more than two parts.  
Then the multiplicity of $V$ in $o_2$ is greater than or equal to the multiplicity of $V$ in $o_2'$. One of the 
main results of our paper is the proof of this conjecture when $m=2$.

As we mentioned before, we arrived at this conjecture by considering nilpotent partial transformations and 
rook theory. Our goal here is to further develop this topic, so we start with recalling the relevant concepts 
and results from~\cite{Can17}. This way, we hope that the reader will see a coherent picture of our work 
and the origins of our conjecture.

Let $n$ denote a positive integer, $[n]$ and $\overline{[n]}$ denote the sets $\{1,\dots, n\}$ and $[n]\cup 
\{0\}$, respectively. We set
\begin{center}
\begin{tabular}{c c l}
$\mc{F}ull_n$ &:& the full transformation semigroup on $\overline{[n]}$; \\
$\mc{P}_n$ &:& the semigroup of partial transformations on $[n]$; \\
$\mc{R}_n$ &:& the rook monoid on $[n]$; \\
$\mc{C}_n$ &:& the set of nilpotent partial transformations on $[n]$.
\end{tabular}
\end{center}

For us, a partial transformation on $[n]$ is a function $f: A \rightarrow [n] $, where $A$ is a nonempty 
subset of $[n]$. A {\em full transformation} on $\overline{[n]}$ is a function $g:\overline{[n]} \rightarrow 
\overline{[n]}$. We note that there is an ``extension by 0'' morphism from partial transformations on $[n]$ into full 
transformations on $\overline{[n]}$,
\begin{align*}
\varphi_0 \ : \ & \mc{P}_n  \rightarrow \mc{F}ull_n\\ 
& f  \longmapsto  \varphi_0(f)
\end{align*} 
which is defined by
$$
\varphi_0(f)(i) = 
\begin{cases}
f(i) & \text{ if $i$ is in the domain of $f$}; \\
0 & \text{ otherwise.}
\end{cases}
$$
The map $ \varphi$ is an injective semigroup homomorphism.

By the {\em rook monoid} $\mc{R}_n$, we mean the monoid of injective partial transformations. We will often refer to the 
elements of $\mc{R}_n$ as rook placements (on an $n\times n$ grid). Since $\mc{P}_n$ is embedded into 
$ \mc{F}ull_n $ and since the latter monoid contains a zero transformation it makes sense to talk about 
nilpotent partial transformations in $\mc{P}_n$. In particular, $\mc{C}_n$ and $\mc{C}_n\cap \mc{R}_n$ 
are well defined although their definitions require an embedding of $\mc{P}_n$ into $\mc{F}ull_n$. (In 
literature, many authors assume that $0\in \mc{R}_n$.)

The unit groups of $\mc{P}_n$ and $\mc{R}_n$ are the same and 
are equal to $\mc{S}_n$. Moreover, the 
conjugation action of $\mc{S}_n$ on itself extends to these semigroups as well as to the subset $\mc
{C}_n$ of nilpotent partial transformations. Therefore we have $ \mc{S}_n$-representations on the 
vector spaces $ \C\,\mc{P}_n,\  \C\,\mc{R}_n$ and on $ \C\,\mc{C}_n$.

The combinatorial significance of $\mc{C}_n$, which is crucial to our purposes, stems from the fact that there 
is a bijection between $\mc{C}_n$ and the set of labeled rooted forests on $ n$ vertices. Furthermore, the 
conjugation action of $\mc{S}_n$ on functions in $\mc{C}_n$ translates to the permutation action of $ \mc{S}
_n$ on the labels. See~\cite{Can17}. 

Let $\tau$ be a rooted forest on $n$ vertices. $ \tau$ is called labeled if there exists a bijective map $ \phi $ 
from $ [n]$ onto the vertex set of $ \tau $. Throughout 
this paper, when we talk about labeled rooted forests, 
we omit writing the corresponding labeling function despite the fact that the action of $\mc{S}_n $ 
does not change the underlying forest but the labeling function only. In accordance with this convention, 
when we write $ \mc{S}_n\cdot \tau$ we actually mean the orbit 
\begin{align}\label{A:an orbit}
\mc{S}_n \cdot (\tau,\phi) = \{ (\tau,\phi'):\ \phi' = \sigma \cdot \phi,\ \sigma \in \mc{S}_n \}.
\end{align}

In~\cite{Can17}, the $ \mc{S}_n$-representation defined by the orbit (\ref{A:an orbit}) is called the {\em odun} of 
$ \tau$. It only depends on $\tau$, so we write $ \mt{Stab}_{\mc{S}_n} (\tau)$ to denote the stabilizer 
subgroup of the pair $(\tau,\phi) $. As an $\mc{S}_n$-module, the vector space of functions on the right 
cosets, that is $\C[ \mc{S}_n/\text{Stab}_{ \mc{S}_n }(\tau)] $ is isomorphic to the odun of $ \tau$, which is 
equivalent to the induced representation $ \mt{Ind}_{\mt{Stab}_{\mc{S}_n} (\tau)}^{\mc{S}_n}\mathbf{1}$. 
It turns out that the Frobenius character of this module is recursive in nature and can be computed by certain 
combinatorial rules. To avoid introducing too excessive notation, we illustrate these rules on an example. 
Let $ \tau$ be the rooted forest depicted in Figure~\ref{F:Example1} and let $ \tau_i $, $ i=1, 2, 3 $ denote 
its connected components (from left to right in the figure). 
\begin{figure}[htp]
\centering
\begin{tikzpicture}[scale=.75]
\begin{scope}[xshift=3cm]
\node at (0,0) {$\bullet$};
\node at (1,1) {$\bullet$};
\node at (-1,1) {$\bullet$};
\node at (1,2) {$\bullet$};
\node at (2,3) {$\bullet$};
\node at (0,3) {$\bullet$};
\node at (-1,2) {$\bullet$};
\draw[-, thick] (0,0) to  (1,1);
\draw[-, thick] (0,0) to  (-1,1);
\draw[-, thick] (-1,1) to  (-1,2);
\draw[-, thick] (1,1) to  (1,2);
\draw[-, thick] (1,2) to  (2,3);
\draw[-, thick] (1,2) to  (0,3);
\end{scope}
\begin{scope}[xshift=0cm]
\node at (0,0) {$\bullet$};
\node at (1,1) {$\bullet$};
\node at (-1,1) {$\bullet$};
\node at (1,2) {$\bullet$};
\node at (2,3) {$\bullet$};
\node at (0,3) {$\bullet$};
\node at (-1,2) {$\bullet$};
\draw[-, thick] (0,0) to  (1,1);
\draw[-, thick] (0,0) to  (-1,1);
\draw[-, thick] (-1,1) to  (-1,2);
\draw[-, thick] (1,1) to  (1,2);
\draw[-, thick] (1,2) to  (2,3);
\draw[-, thick] (1,2) to  (0,3);
\end{scope}
\begin{scope}[xshift=-3.5cm]
\node at (0,0) {$\bullet$};
\node at (0,1) {$\bullet$};
\node at (-.5,2) {$\bullet$};
\node at (-1.5,2) {$\bullet$};
\node at (.5,2) {$\bullet$};
\node at (1.5,2) {$\bullet$};
\draw[-, thick] (0,0) to  (0,1);
\draw[-, thick] (0,1) to  (-1.5,2);
\draw[-, thick] (0,1) to  (1.5,2);
\draw[-, thick] (0,1) to  (-.5,2);
\draw[-, thick] (0,1) to  (.5,2);
\end{scope}
\end{tikzpicture}
\caption{An example.}
\label{F:Example1}
\end{figure}
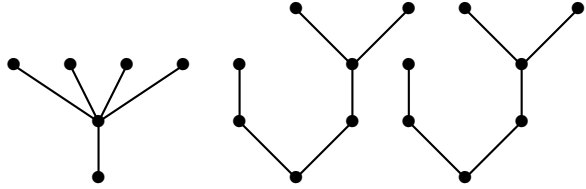
Let $F_{\tau}$ and $F_{\tau_i}$, $i=1,2,3$ denote the corresponding Frobenius characters. Since
$\tau_1 \neq \tau_2 = \tau_3$, we have 
\begin{align*}
F_{\tau} &= F_{\tau_1} \cdot s_2[ F_{\tau_2}].
\end{align*}
Here, $ s_2 $ is the Schur function $ s_{(2)} $, the bracket stands for the plethysm of symmetric functions and 
the dot stands for ordinary multiplication. (We will give a brief 
introduction to plethysm in Section 2.2.) More generally, if 
a connected component $\tau'$, which of course is a rooted tree appears $k$-times in a forest $ \tau $, then 
$F_{\tau}$ has $s_k [ F_{\tau'}]$ as a factor. Now we proceed to explain the computation of the 
Frobenius character of a rooted tree. As an example we use $ F_{\tau_1} $ of Figure~\ref{F:Example1}. 
The combinatorial rule that we obtained in~\cite{Can17} is simple; it is the removal of the root from tree. The 
effect on the Frobenius character of this simple rule is as follows: Let $ \tau_1'$ denote the rooted forest that 
we obtain from $ \tau_1$ by removing the root. Then $ F_{\tau_1}  = s_1 \cdot F_{\tau_1'} $. Thus, 
by the repeated application of this rule and the previous factorization rule, we obtain 
\begin{align*}
F_{\tau_1} &= s_1\cdot F_{\tau_1}= s_1\cdot s_1 \cdot s_4[ s_1].
\end{align*}
We compute $ F_{\tau_2} $ by the same method and we arrive at the following satisfactory expression 
for the Frobenius character of $\tau$:
\begin{align*}
F_{\tau}&= s_1^2\cdot s_4[s_1] \cdot F_{\tau_2}= s_1^2\cdot s_4\cdot s_2[s_1^5 \cdot s_2].
\end{align*}

Let $ \tau$ be a rooted forest and $a$ be a vertex in $\tau$. 
We let $ \tau_a^0$ denote the rooted subforest whose 
set of vertices is $ \{ b \in \tau :\ b < a \} $ and we let 
$ \tau_a $ denote the rooted subtree whose set of 
vertices is $\{b\in \tau :\ b \leq a \}$. Then we let $ \{a_1, \dots, a_r
\}$ of $\tau_a$ such that the corresponding subtrees $\tau_{a_i}$ ($ i=1,\dots, r$) are mutually distinct. 
Finally, let $ m(a;a_i)= m_\tau(a;a_i)$ ($i=1,\dots, r$) denote the number of copies of $\tau_{a_i}$ that 
appear in $\tau_a^0$. In this notation, the dimension of the odun of $\tau$ turns out to be 
\begin{align}\label{A:ms}
\dim_{\C} o(\tau) = \frac{ n! } { \prod_{a\in \tau} \prod_{b \in \tau_a} m(a;b)!},
\end{align}
where $n$ is the number of vertices of $\tau$. See~\cite[Theorem 9.3]{Can17}. 
This is a generalization of the Knuth-Sagan hook length formula for the rooted trees.

Another important advancement of~\cite{Can17} is the determination of the rooted trees whose odun afford
the sign representation. We will give a detailed explanation of the arguments involved with this computation 
later in this paper.  Such arguments were used in \cite{Can17} to 
prove that the generating function for the number of rooted forests which 
afford the sign representation is $ \sum_{n \geq 2} 2^{n-2} x^n $. Here, $ n$ is interpreted as the number of 
vertices.

We are now ready to introduce our new combinatorial object by means of which we will generalize the results of \cite{Can17}. 
The algebraic significance of these objects will be developed later in this paper. A {\em 
loop-augmented forest} is a rooted forest with loops on some of its roots. See, for 
example, Figure~\ref{F:pruned}, where we depict a loop-augmented forest on 22 vertices and four loops.
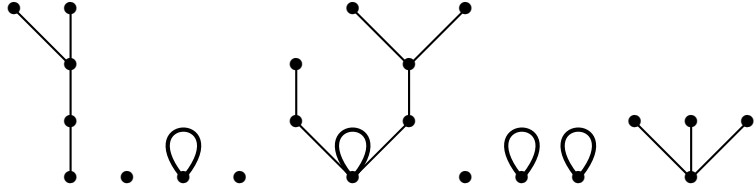
\begin{figure}[htp]
\centering
\begin{tikzpicture}[scale=.75]
\begin{scope}
\node at (0,0) {$\bullet$};
\node at (1,1) {$\bullet$};
\node at (-1,1) {$\bullet$};
\node at (1,2) {$\bullet$};
\node at (2,3) {$\bullet$};
\node at (0,3) {$\bullet$};
\node at (-1,2) {$\bullet$};
\draw[-, thick] (0,0) to  (1,1);
\draw[-, thick] (0,0) to  (-1,1);
\draw[-, thick] (-1,1) to  (-1,2);
\draw[-, thick] (1,1) to  (1,2);
\draw[-, thick] (1,2) to  (2,3);
\draw[-, thick] (1,2) to  (0,3);
\end{scope}
\begin{scope}[scale=3]
\draw[ultra thick] (0,0)  to[in=50,out=130, loop] (0,0);
\end{scope}
\begin{scope}[xshift=-5cm]
\node at (0,0) {$\bullet$};
\node at (0,1) {$\bullet$};
\node at (0,2) {$\bullet$};
\node at (0,3) {$\bullet$};
\node at (-1,3) {$\bullet$};
\draw[-, thick] (0,0) to  (0,1);
\draw[-, thick] (0,1) to  (0,2);
\draw[-, thick] (0,2) to  (0,3);
\draw[-, thick] (0,2) to  (-1,3);
\end{scope}
\begin{scope}[xshift=-4cm]
\node at (0,0) {$\bullet$};
\end{scope}
\begin{scope}[xshift=-3cm,scale=3]
\node at (0,0) {$\bullet$};
\draw[ultra thick] (0,0)  to[in=50,out=130, loop] (0,0);
\end{scope}
\begin{scope}[xshift=-2cm]
\node at (0,0) {$\bullet$};
\end{scope}
\begin{scope}[xshift=2cm]
\node at (0,0) {$\bullet$};
\end{scope}
\begin{scope}[xshift=3cm,scale=3]
\node at (0,0) {$\bullet$};
\draw[ultra thick] (0,0)  to[in=50,out=130, loop] (0,0);
\end{scope}
\begin{scope}[xshift=4cm,scale=3]
\node at (0,0) {$\bullet$};
\draw[ultra thick] (0,0)  to[in=50,out=130, loop] (0,0);
\end{scope}
\begin{scope}[xshift=6cm]
\node at (0,0) {$\bullet$};
\node at (1,1) {$\bullet$};
\node at (-1,1) {$\bullet$};
\node at (0,1) {$\bullet$};
\draw[-, thick] (0,0) to  (1,1);
\draw[-, thick] (0,0) to  (-1,1);
\draw[-, thick] (0,0) to  (0,1);
\end{scope}
\end{tikzpicture}
\caption{A loop-augmented forest.}
\label{F:pruned}
\end{figure}
It is a well known variation of the Cayley's theorem that the number of labeled forests on $ n$ vertices with 
$k$ roots is equal to ${n-1\choose k-1} n^{n-k}$. See~\cite{MR0460128}, Theorem D, pg 70. It follows that
the number of loop-augmented forests on $ n $ vertices with $ k $ roots is $2^k{n-1\choose k-1} n^{n-k} $. 
It follows from generating function manipulations that the number of loop-augmented forests on $n$ vertices 
for $ n\geq 2$ is $2n^{n-3}$. 

Let $ \tau $ be a rooted tree on $n$ vertices. We group the maximal branches of $ \tau$ according to the 
number of vertices on each maximal branch. As a convention, let us collect the maximal branches with more 
vertices on the left hand side of the graph and repeat this procedure inductively for every maximal branch 
and for their maximal branches, and so on. We do not worry about the ordering of the maximal branches in 
the same group. As an example, see Figure~\ref{F:Example2}. Once the grouping is done, we label the 
vertices by $[n]$ from left to right and from bottom to top, decreasing the label by 1 each time. Here, $n$ is 
the number of vertices and the root of $\tau$ is labeled by n.
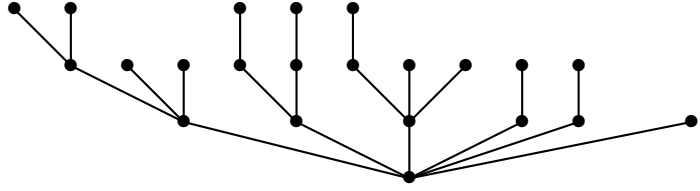
\begin{figure}[htp]
\centering
\begin{tikzpicture}[scale=.75]

\begin{scope}[xshift=0cm]
\node at (0,0) {$\bullet$};
\node at (-4,1) {$\bullet$};
\draw[-, thick] (0,0) to  (-4,1);
\node at (-6,2) {$\bullet$};
\node at (-5,2) {$\bullet$};
\node at (-4,2) {$\bullet$};
\draw[-, thick] (-4,1) to  (-6,2);
\draw[-, thick] (-4,1) to  (-5,2);
\draw[-, thick] (-4,1) to  (-4,2);
\node at (-7,3) {$\bullet$};
\node at (-6,3) {$\bullet$};
\draw[-, thick] (-6,2) to  (-7,3);
\draw[-, thick] (-6,2) to  (-6,3);
\node at (-2,1) {$\bullet$};
\draw[-, thick] (0,0) to  (-2,1);
\node at (-3,2) {$\bullet$};
\node at (-2,2) {$\bullet$};
\node at (-2,3) {$\bullet$};
\node at (-3,3) {$\bullet$};
\draw[-, thick] (-2,1) to  (-3,2);
\draw[-, thick] (-2,1) to  (-2,2);
\draw[-, thick] (-2,2) to  (-2,3);
\draw[-, thick] (-3,2) to  (-3,3);
\node at (-1,2) {$\bullet$};
\node at (0,2) {$\bullet$};
\node at (1,2) {$\bullet$};
\node at (-1,3) {$\bullet$};
\draw[-, thick] (0,0) to  (0,1);
\draw[-, thick] (0,1) to  (-1,2);
\draw[-, thick] (0,1) to  (0,2);
\draw[-, thick] (0,1) to  (1,2);
\draw[-, thick] (-1,2) to  (-1,3);
\node at (0,1) {$\bullet$};
\node at (2,1) {$\bullet$};
\node at (2,2) {$\bullet$};
\node at (3,1) {$\bullet$};
\node at (3,2) {$\bullet$};
\node at (5,1) {$\bullet$};
\draw[-, thick] (0,0) to  (2,1);
\draw[-, thick] (0,0) to  (3,1);
\draw[-, thick] (0,0) to  (5,1);
\draw[-, thick] (2,1) to  (2,2);
\draw[-, thick] (3,1) to  (3,2);
\end{scope}
\end{tikzpicture}
\caption{A grouping of the branches of a rooted tree.}
\label{F:Example2}
\end{figure}
Notwithstanding the ambiguity created by not ordering the maximal branches of the same size, the resulting
incidence matrix from our labeling is a strictly upper triangular (nilpotent) matrix. We adopt this procedure to
the rooted forests by first grouping (and depicting left to right) the rooted trees with respect to the number of
vertices they have. The resulting incidence matrix of the forest is a direct sum of nilpotent matrices, hence it
is nilpotent also. Notice that this procedure also explains why there is a bijection between labeled rooted 
forests and nilpotent partial functions. That is, given $\tau \in \mc{C}_n$ with respect to any labeling that we just 
described, we simply  let $f$ be the partial function defined by $f(i) = j$ whenever $i$ is a child of $j$ in $\tau$ and there 
is an edge between $i$ and $j$. Then $f$ is nilpotent. In the light of this fact, from now on,  we will will use the same notation 
$\tau$ for a nilpotent partial function and its corresponding labeled rooted forest.

We now proceed to explain how to interpret loop-augmented forests in terms of partial functions. Let $\sigma$ be 
a loop-augmented forest. Then some of the roots of $\sigma$ have loops. There is still a partial function for 
$\sigma$, as defined in the previous paragraph for a rooted forest. The loops in this case correspond to the 
``fixed points'' of the associated function. Indeed, with respect to labeling conventions of the previous 
paragraph, the incidence matrix of $\sigma$ is upper triangular. However 
it can have  both of 1's and 0's 
on its main diagonal. The resulting matrix can be viewed as the matrix representation of the partial function
of $\sigma$. ({\em Warning:} Unlike our notation in this paragraph, in the sequel we will often use lower case 
Latin letters for denoting both of loop-augmented forests and partial functions.) We remark here that the 
correspondence between loop-augmented forests and (certain) partial functions fails if we consider adding loops 
on non-roots. In such a situation,
the resulting objects do not always give partial functions, but, instead, give 
``partial relations.''

The subset of $\mc{P}_n$ which corresponds to the loop-augmented forests will be formally defined in the next section. Nevertheless, 
we can briefly explain the underlying algebraic theme of our work. 
Recall that the permutation action of $ \mc{S}_n $ on the labels translates to the conjugation action on the 
incidence matrices. Clearly, the conjugates of a nilpotent (respectively unipotent) matrix are still nilpotent 
(respectively unipotent). Let $U$ be an arbitrary upper triangular matrix. Then $U$ is equal to a sum of the 
form $ D+N $, where $ D$ is a diagonal matrix and $N$ is a nilpotent matrix. If $ \sigma $ is a permutation matrix 
of the same size as $ U$, then we conclude from the equalities $\sigma \cdot U= \sigma U\sigma^{-1}= \sigma 
D \sigma^{-1}  + \sigma N \sigma^{-1}$ that the conjugation action on loop-augmented forests is equivalent to 
the simultaneous conjugation action on labeled rooted forests and the representation of $\mc{S}_n$ on 
diagonal matrices.

The outline of this paper is as follows. In 
Section~\ref{S:Preliminaries}, we review the most basic facts about the idempotents of the semigroups 
$\mc{F}ull_n,\mc{P}_n$, and $\mc{R}_n$. In addition we review the basic theory of plethysm for symmetric 
functions. In Section~\ref{S:Nilps}, we prove the first of our 
main results, Theorem~\ref{T:Stabilizer}, which states 
that the stabilizer subgroup in $\mc{S}_n$ of a loop-augmented forest (viewed as a partial function on $[n]$) 
has a direct product decomposition of the form $Z(\sigma)\times \mt{Stab}_{\mc{S}_{n-k}}(\tau)$, where 
$Z(\sigma)$ is the centralizer of a permutation $\sigma$ in $\mc{S}_k$ and $\mt{Stab}_{\mc{S}_{n-k}}(\tau)$ 
is the stabilizer subgroup in $\mc{S}_{n-k}$ of a labeled rooted forest $\tau$. It follows from this result that the 
Frobenius character of the $\mc{S}_n$ orbit of a loop-augmented forest $f$ is of the form $s_\lambda \prod_
{i=1}^r h_{\lambda_i}[F_{\tau_i}]$, where $\tau_1,\dots,\tau_r$ are the distinct connected components of 
$\tau$ where $\tau_i$ occurs $\lambda_i$ times. We record this theorem as Theorem~\ref{T:Master 1} in 
Section~\ref{S:Characters}. Another consequence of the factorization of the stabilizer subgroup is a dimension 
formula for the $\mc{S}_n$ representation defined by $f$ (Corollary~\ref{C:hooklength}). If $(1^{m_1} 2^{m_2} 
\dots k^{m_k})$ denotes the conjugacy type of $\sigma$ in $\mc{S}_k$, then the dimension of $\C [\mc{S}_n
\cdot f]$ is given by  
\begin{align}\label{A:ms pruned}
 \frac{n!}{k!}\frac{ \prod_j j^{m_j} (m_j!)}{ ( \prod_{a\in \tau} \prod_{b \in \gamma_a} m(a;b)!)},
\end{align}
where $m(a;b)=m_\tau(a;b)$ is as defined in (\ref{A:ms}).
The goal of Section~\ref{S:Sign} is to determine the loop-augmented forests that afford a sign representation. 
We prove in Theorem~\ref{T:number of sign} that the number of loop-augmented forests on $n$ vertices with 
$k$ loops whose $\mc{S}_n$ orbit admits a sign representation is $2^{n-k-2}$. In particular, the total number 
of loop-augmented forests on $n$ vertices admitting a sign representation is $2^{n-1}-1$. In  
Section~\ref{S:Nondiagonal}, we prove the main algebraic result of our paper. 
Namely, if $e\in \mc{P}_n$ is an arbitrary idempotent, then there exists a multiset of integers $\{\beta_1^{m_1},\dots, 
\beta_s^{m_s}\}$ and $m\in \N$ such that 
\begin{enumerate}
\item $m+ \sum_{i=1}^s m_i \beta_i =n$;
\item 
$\mt{Stab}_{\mc{S}_n}(e) \simeq (\mc{S}_{m_1} \wr \mc{S}_{\beta_1}) \times (\mc{S}_{m_2} 
\wr \mc{S}_{\beta_2}) \times \cdots 
\times (\mc{S}_{m_s}\wr \mc{S}_{\beta_s}) \times \mc{S}_m$.
\end{enumerate}
We record this result in Corollary~\ref{C:Algebraic}. Finally, in Section~\ref{S:Conjecture}, we prove our variant of the Foulkes' conjecture 
in the special case where $m=2$. (See Conjecture~\ref{C:Conjecture} for 
an equivalent statement in terms of symmetric functions and plethysm.)

\section{Preliminaries}\label{S:Preliminaries}

Structures of the semigroups $\mc{P}_n$ and $\mc{R}_n$ become more transparent by virtue of matrices; 
each element $g$ of $\mc{P}_n$ has a unique $n\times n$ matrix representative $g=(g_{ij})_{i,j=1}^{n}$ 
defined by 
$$
g_{ij} = 
\begin{cases}
1 &  \ \text{ if } g(i)=j;\\ 
0 &  \ \text{ otherwise}
\end{cases}
\qquad (i,j\in [n]).
$$ 
What is important for us is that the matrix product translates to the semigroup multiplication (composition).
In particular, a partial transformation $f\in \mc{P}_n$ is nilpotent if some power of the corresponding matrix 
$f=(f_{i,j})_{i,j=1}^n$ is the zero matrix $\mr{0}_n$.

\subsection{Idempotents.}


An idempotent in a semigroup $S$ is an element $ e \in S$ satisfying $e^2=e$. The identity transformation 
$id_n:[n]\rightarrow [n ] $ of the unit group is an idempotent, and conversely for any idempotent $ e \in S:= 
\mc{P}_n $, there is a corresponding  ``local group'' $S_e$ determined by $ e$; it is defined as 
$ S_e:= e\mc{P}_n e $. The identity of the local group $S_e$ is $e$. We denote the idempotent set of a 
semigroup $S$ by $E(S)$.

By a {\em diagonal idempotent} in $\mc{P}_n $, we mean an idempotent of $ \mc{R}_n$. The nomenclature 
follows from the fact that the matrix representation of $ e\in E (\mc{ R}_n) $ is a diagonal matrix. $ E(\mc{R
}_n) \subset \mc{R}_n$ has exactly $ 2^n $ elements (including the $0$-matrix) and it forms a commutative 
submonoid of $\mc{R}_n$. It is not difficult to see that $ E(\mc{R}_n) $ is isomorphic to the ``face lattice'' of 
the regular ``$ n $-simplex'' in $\R^n $. For example, 
\begin{align*}
E(\mc{R}_2) &= 
\left\{
\begin{pmatrix} 1 & 0 \\ 0 & 1 \end{pmatrix},
\begin{pmatrix} 1 & 0 \\ 0 & 0 \end{pmatrix},
\begin{pmatrix} 0 & 0 \\ 0 & 1 \end{pmatrix},
\begin{pmatrix} 0 & 0 \\ 0 & 0 \end{pmatrix}
\right\}
\end{align*}
Let $ U $ denote the unit group of a semigroup $ S $. For our purposes it is important to note that $E(S)$ is 
closed under the conjugation action 
$$
g\cdot e = g e g^{-1},\qquad g\in U, e\in E(S).
$$ 
In particular we see that $ E(\mc{R}_n) $ is an $ \mc{S}_n $-set that is contained in $E(\mc{P}_n)$. For example, 
$$
E(\mc{P}_2) = 
\left\{
\begin{pmatrix} 1 & 1 \\ 0 & 0 \end{pmatrix},
\begin{pmatrix} 0 & 0 \\ 1 & 1 \end{pmatrix} \right\} \bigcup E(\mc{R}_2).
$$
We list some easy facts about the set of idempotents of $\mc{P}_n$. For the proofs of some of these facts 
(and more) we recommend~\cite{GanyushkinMazorchuk}.
\begin{enumerate}
\item The number of idempotents in $\mc{P}_n$ is given by $\sum_{k=0}^n {n \choose k} (k+1)^{n-k}$. 
\item The number of idempotents in $\mc{F}ull_n$ is given by $\sum_{k=1}^n {n \choose k} k^{n-k}$.
\item The idempotents in $\mc{P}_n$ do {\em not} form a subsemigroup of $\mc{P}_n$ for $n\geq 2$.
\item The idempotents in $\mc{F}ull_n$ do {\em not} form a subsemigroup of $\mc{F}ull_n$ for $n\geq 3$.
\item The idempotents in $\mc{R}_n$ form a submonoid of $\mc{R}_n$.
\end{enumerate}

Let $e$ be a diagonal idempotent in $\in \mc{P}_n$. Since all idempotents of rank $k$ in $\mc{R}_n$ are 
related to each other via conjugation, the $\mc{S}_n$ orbit of $e$ is uniquely determined by the number 
$k=\mt{rank}(e)$.  We define 
\begin{align}\label{A:oi}
\mc{N}(n, k) = \{ f\in \mc{P}_n :\ f^m \ \text{is a rank $k$ diagonal idempotent for some } m\geq 1\},
\end{align}
which is closed under conjugation action of $\mc{S}_n$. Clearly, if $k=0$, then $\mc{N}(n, k)$ is equal to 
$\mc{C}_n$, the set of nilpotent elements in $ \mc{P}_n$.

\begin{Remark}\label{R:not to be confused with}
By relaxing the requirement on diagonal property, we could have considered  
\begin{align}\label{A:could have}
\{ f\in \mc{P}_n :\ f^m \ \text{is a rank $k$ idempotent for some } m\geq 1 \}.
\end{align}
This set properly contains $ \mc{N} (n, k) $. However, the matrix rank 
of idempotents is not a refined invariant with respect to conjugation action. To justify this claim, we consider
$$
e_1:= 
\begin{pmatrix} 1 & 0 \\ 0 & 0 
\end{pmatrix} 
\ \text{ and } \ 
e_2:= 
\begin{pmatrix} 1 & 1 \\ 0 & 0 
\end{pmatrix}
$$
These idempotents are of the same rank but they are not conjugate. To capture differences between the
conjugacy classes of $e_1$ and $e_2$, we will introduce below the concept of ``standard idempotent.''
\end{Remark}

Let $ e \in E(\mc{P}_n) $ be an idempotent with the property that the only nonzero row of $ e$ has 1's as its 
entries. For example, $e=\begin{pmatrix} 0 & 0 & 0 \\  0 & 0 & 0\\1& 1 & 1\end{pmatrix}$. Such idempotents,
viewed as functions, are constant. Conversely, any nonzero constant function $f: [n]\rightarrow [n]$ defines 
an idempotent $e\in \mc{P}_n$ with a single row of 1's.

\begin{Definition}
An idempotent is called constant, and denoted by $c^{(n)}$, if the corresponding function is 1. In other words
$c^{(n)}(i)=1$ for all $i\in \{1,\dots, n\}$. An idempotent in $\mc{P}_n$ is called standard if it is a direct sum of 
constant idempotents of various sizes plus a $0$-matrix. 
\end{Definition}

Let $x\in \mc{P}_n$. The following set (not to be confused with $\mc{N}(n, k)$, where $k$ is a number) 
will play an important role in the sequel:
\begin{align}\label{A:general oi}
\mc{N}(n, x) := \{ f\in \mc{P}_n :\ f^m  \text{ is conjugate to $x$ for some } m\geq 1 \}.
\end{align}
\begin{Lemma}
If $x=e$ be an idempotent, then $ \mc{N}(n,e)$ is closed under conjugation. Furthermore, if $e$ is a 
diagonal idempotent of rank $k$, then $\mc{N}(n, k) = \mc{N}(n, e)$. 
\end{Lemma}
\begin{proof}
Straightforward.
\end{proof}

\begin{Remark}\label{R:The reason}
The reason for we introduce the set (\ref{A:oi}) is because $\mc{N}(n, k)$ is precisely the subset of 
$\mc{P}_n$ that corresponds to the loop-augmented forests. Thus, our main goal is to understand the 
$\mc{S}_n$-module structure on $\mc{N}(n, k)$. 
\end{Remark}

\subsection{Symmetric functions and plethysm.}

In representation theory, plethysm refers to a general idea on a functorial operation, which is defined from a 
representation to another one.   More concretely, it is a description of the decompositions, with multiplicities, 
of representations derived from a given representation $V$ such as $V\otimes V,\bigwedge^k(V), \bigwedge
^k(\bigwedge^l(V)),\mt{Sym}^k(V),\mt{Sym}^l(\mt{Sym}^k(V)),V^*$.. In practice one often considers functors 
attached to the weight of an irreducible representation such as Schur functors. Let $\mb{S}_\mu$ resp. $\mb 
{S}_\lambda$ denote the Schur functors indexed by the partitions $\mu$ and $\lambda$, respectively.. Then 
$ \mb{S}_\lambda \circ \mb{S}_\mu := \mb{S}_\lambda (\mb{S}_\mu (\C^n))$ is called the plethysm of $  \mb
{S}_\lambda$ with $ \mb{S}_\mu $. There is a corresponding operation on the ring of symmetric functions.

The {\em   $ k $-th power-sum symmetric function}, denoted by $ p_k $, is the sum of $ k $-th powers of the 
variables. For a partition $\lambda = (\lambda_1 \leq \cdots \leq \lambda_l)$, we define $p_\lambda$ to be 
equal to $\prod_i p_{\lambda_i} $. (We will maintain the convention of writing parts of a partition in the 
increasing order throughout our paper.) The {\em $k$-th complete symmetric function}, $h_k$ is defined to be the 
sum of all monomials $x_{i_1}^{a_1}\dots x_{i_r}^{a_r}$ with $ \sum a_i =k $ and $ h_\lambda$ is defined to 
be equal to $\prod_{i=1}^l h_{\lambda_i}$. The {\em Schur function associated with $ \lambda$}, denoted by 
$s_\lambda$, is the symmetric function defined by the determinant $\det (h_{\lambda_i+j - i})_{ i,j=1}^l$. The 
power sum and Schur symmetric functions will play special roles in our computations via the {\em Frobenius 
character} map 
\begin{align*}
\textbf{F} : \text{ class functions on $\mc{S}_n$ } & \rightarrow  \text{ symmetric functions} \\
\delta_\sigma & \mapsto \frac{1}{n!}\ p_\lambda,
\end{align*}
where $ \sigma \subset \mc{S}_n $ is a conjugacy class of type $ \lambda $ and $ \delta_\sigma$ is the indicator 
function
\begin{align*}
\delta_\sigma( x) =
\begin{cases}
1 & \text{ if } x\in \sigma;\\
0 & \text{ otherwise.}
\end{cases}
\end{align*}
It turns out that if $ \chi^\lambda$ is the irreducible character of $ \mc{S}_n $ indexed by the partition $ \lambda $, 
then $\textbf{F} (\chi^\lambda) = s_\lambda $. In the sequel, we will not distinguish between representations 
of $ \mc{S}_n $ the corresponding characters. In particular, we will often write the {\em Frobenius character of an 
orbit} to mean the image under $\textbf{F} $ of the character of the representation of $ \mc{S}_n $ that is defined 
by the action on the orbit.

For two irreducible characters $ \chi^\mu$ and $ \chi^\lambda$ indexed by partitions $ \lambda $ and $\mu$,  
the Frobenius characteristic image of the plethysm $\chi^\lambda [ \chi^\mu]$ is the plethystic substitution 
$s_\lambda [ s_\mu]$ of the corresponding Schur functions. Roughly speaking, the plethysm of the Schur 
function $s_\lambda$ with $s_\mu$ is the symmetric function obtained from $s_\lambda$ by substituting the 
monomials of $s_\mu$ for the variables of $s_\lambda$. In the notation of~\cite{LoehrRemmel}; the plethysm 
of symmetric functions is the unique map $[\cdot ]:\ \Lambda\times \Lambda \rightarrow \Lambda$ satisfying the 
following three axioms: 
\begin{enumerate}
\item[P1.] For all $m,n\geq 1$, $p_m [p_n] = p_{mn}$.
\item[P2.] For all $m\geq 1$, the map $g \mapsto p_m [ g]$, $g\in \Lambda$ defines a $\Q$-algebra 
homomorphism on $\Lambda$. 
\item[P3.] For all $g\in \Lambda$, the map $h \mapsto h [ g]$, $h\in \Lambda$ defines a $\Q$-algebra 
homomorphism on $\Lambda$. 
\end{enumerate}

Although the problem of computing the plethysm of two (arbitrary) symmetric functions is very difficult, there 
are some useful formulas for Schur functions:
\begin{align}\label{A:plethysm formula 1}
s_\lambda [ g+h] &= \sum_{\mu,\nu} c_{\mu,\nu}^\lambda (s_\mu [ g]) (s_\nu [ h]),
\end{align}
and 
\begin{align}\label{A:plethysm formula 2}
s_\lambda [gh] &= \sum_{\mu,\nu} \gamma_{\mu,\nu}^\lambda (s_\mu [ g]) (s_\nu [h]).
\end{align}
Here, $ g$ and $ h$ are arbitrary symmetric functions, $ c_{\mu,\nu}^\lambda $ is a scalar, and $\gamma_{
\mu,\nu}^\lambda$ is $\frac{1}{n!} \langle \chi^\lambda, \chi^\mu \chi^\nu \rangle$, where the pairing stands
for the standard Hall inner product on characters.

In~(\ref{A:plethysm formula 1}) the summation is over all pairs of partitions $\mu,\nu \subset \lambda$, and 
the summation in~(\ref{A:plethysm formula 2}) is over all pairs of partitions $\mu,\nu$ such that $  |\mu| =  |
\nu | = |\lambda|$. In the special case when $\lambda = (n)$, or $(1^n)$ we have
\begin{align}
s_{(n)} [gh] &= \sum_{\lambda \vdash n} (s_\lambda [g]) (s_\lambda [h]), \label{A:complete}\\
s_{(1^n)} [gh]&= \sum_{\lambda \vdash n}(s_\lambda[g])(s_{\lambda'}[h]), \label{A:elementary}
\end{align}
where $\lambda'$ denotes the conjugate of $\lambda$.

Finally, let us mention the conjugation property of plethysm:
\begin{align}\label{A:conjugation prop}
\langle s_\lambda [s_\mu] , s_\gamma \rangle =
\begin{cases}
\langle s_\lambda [s_{\mu'}], s_{\gamma'} \rangle & \text{ if $|\mu|$ is even;}\\
\langle s_{\lambda'} [ s_{\mu'}], s_{\gamma'} \rangle & \text{ if $|\mu|$ is odd.}
\end{cases}
\end{align}

\section{Nilpotents and diagonal idempotents}\label{S:Nilps}

Let $f$ be an element of $\mc{N}(n, k)$, $m$ be a positive integer such that $f^m$ is a diagonal idempotent
in $\mc{P}_n$. Without loss of generality we assume that $f^m$ is of the form 
\begin{align}\label{A:idk}
f^m = 
\begin{pmatrix}
id_k & 0 \\
0 & 0 
\end{pmatrix}.
\end{align}
Our first result is about the actual ``shape'' of this matrix representation of $f$.

\begin{Proposition}\label{P:First result} 
If $f$ is a partial transformation as in the previous paragraph, then there exist a permutation $\sigma \in 
\mc{S}_k$ and a nilpotent partial transformation $\tau$ from $\mc{C}_{n-k}$ such that 
\begin{align}\label{A:block diagonal}
f = 
\begin{pmatrix}
\sigma & 0 \\
0 & \tau
\end{pmatrix}.
\end{align}
\end{Proposition}

\begin{proof}
First, we are going to show that $f$ has the block-diagonal form as in (\ref{A:block diagonal}). We will prove 
this by contradiction, so we assume that there exists $j\in \{k,\dots, n\}$ such that $f(j) = i < k$. On one hand, 
since $f^m$ is as in (\ref{A:idk}), we have $f^m(i)=i$. On the other hand, by our assumption $0 = f^m(j) = 
f^{m-1}(f(j))= f^{m-1}(i)$. Therefore, $f^m(i)=0$, but this is a contradiction. Therefore, $f$ is as in 
(\ref{A:block diagonal}) for some $\sigma \in \mc{P}_k$ and $\tau \in \mc{C}_{n-k}$. 

Since the $m$-th power of $\begin{pmatrix} \sigma & 0 \\ 0 & \tau \end{pmatrix}$ is as in (\ref{A:idk}), the 
$m$-th power of $\sigma$ is the $k\times k$ identity matrix and the $m$-th power of $\tau$ is $\mathbf{0}_
{n-k}$. In other words, $\sigma$ is a permutation matrix, $\tau$ is a nilpotent matrix. The proof is complete.
\end{proof}

\begin{Theorem}\label{T:Stabilizer}
Let $f$ be a partial transformation of the form $f = \begin{pmatrix} \sigma & 0 \\ 0 & \tau \end{pmatrix}$, where 
$\sigma \in \mc{S}_k$ is a permutation and $\tau\in \mc{C}_{n-k}$ is a nilpotent partial transformation. In this
case, the stabilizer subgroup in $\mc{S}_n$ of $f$ has the following decomposition:
$$
\mt{Stab}_{\mc{S}_n} (f) = Z(\sigma) \times \mt{Stab}_{\mc{S}_{n-k}}(\tau),
$$
where $Z(\sigma)$ is the centralizer of $\sigma$ in $\mc{S}_k$ and $\mt{Stab}_{\mc{S}_{n-k}}(\tau)$ is the stabilizer 
subgroup of $\tau$ in $\mc{S}_{n-k}$.
\end{Theorem}

\begin{proof}
Let $ \pi = \begin{pmatrix} A & B \\ C & D  \end{pmatrix} \in \mc{S}_n $ be a permutation matrix given in block form 
so that $A$ is a $k\times k$-matrix, $D$ is $n-k \times n-k$-matrix, $B$ is $k\times n-k$-matrix, and $C$ is $n-k\times k$-matrix. 
Assume that $\pi f \pi^{-1} = f$. We claim that both $B$ and $C$ are 0 matrices. To see this, we use $\pi f = f 
\pi$ to get four matrix equations:
\begin{align}
A \sigma &= \sigma A \label{A:4-I}\\
B N &= \sigma B \label{A:4-II} \\ 
C \sigma &= N C  \label{A:4-III} \\
N D & = DN. \label{A:4-IV}
\end{align}
We are going to show only that $B=0$ since for $C$ the idea of the proof is the same. 

First of all, if $ N=0 $, then it immediately follows from the above equations that $ B=0 $, so we assume that 
$N$ is a nonzero nilpotent matrix. Let $r$ be the smallest integer such that $BN^r = 0$. If $ r=1$, then since 
$ \sigma $ is invertible, it follows from equation (\ref{A:4-II}) that $B=0$, so we continue with the assumption 
that $r>1$. In this case, by multiplying equation (\ref{A:4-II}) with $ N^{r-1}$ on the right we obtain $\sigma B 
N^{r-1} = BN^r = 0 $. By the same token, since $ \sigma $ is invertible, we see that $ BN^{r-1} = 0 $. But this 
contradicts with our assumption on the minimality of $ r$. Therefore, $r=1$, hence $B=0$. Since both of $B$ 
and $C$ are $0$, the proof follows from equations (\ref{A:4-I}) and (\ref{A:4-IV}).
\end{proof}

By Proposition~\ref{P:First result} and Remark~\ref{R:The reason}, we see that a partial transformation $f\in 
\mc{P}_n$ represents a loop-augmented forest, then its matrix form is as in (\ref{A:block diagonal}). By this
notation, as an application of Theorem~\ref{T:Stabilizer}, we will now record a representation theoretic 
generalization of the famous Knuth-Sagan hook length formula. 
\begin{Corollary}\label{C:hooklength}
Let $f$ be partial transformation representing a loop-augmented forest as in (\ref{A:block diagonal}). If 
$(1^{m_1} 2^{m_2} \dots k^{m_k})$ is the conjugacy type of the permutation $\sigma$, then the dimension of 
$o(f)$, the $\mc{S}_n$ representation defined on the orbit of $f$ is equal to  
\begin{align*}
 \frac{n!}{k!}\frac{ \prod_j j^{m_j} (m_j!)}{ ( \prod_{a\in \tau} \prod_{b \in \gamma_a} m(a;b)!)},
\end{align*}
where $m(a;b)=m_\tau(a;b)$ is as defined in (\ref{A:ms}).
\end{Corollary}
\begin{proof}
By Theorem~\ref{T:Stabilizer} we know that the $\mc{S}_n$ representation defined by $f$ is isomorphic to 
$\C [ \mc{S}_n / Z(\sigma) \times \mt{Stab}_{\mc{S}_{n-k}}(\tau)]$, hence its dimension is given by 
$n! /( |Z(\sigma)||\mt{Stab}_{\mc{S}_{n-k}}(\tau)]|)$. 
We express this number in a different way: 
\begin{align}
\dim_\C \mc{S}_n \cdot f &= \frac{n!}{ |Z(\sigma)||\mt{Stab}_{\mc{S}_{n-k}}(\tau)|} \notag  \\
&= {n \choose k} \frac{k!}{ |Z(\sigma)|} \frac{(n-k)!}{|\mt{Stab}_{\mc{S}_{n-k}}(\tau)|}. \label{AC:useful}
\end{align}
Note that $|Z(\sigma)|$ is equal to $k!/ |C(\sigma)|$, where $C(\sigma)$ is the conjugacy class of $\sigma$ in 
$\mc{S}_k$. The cardinality of $C(\sigma)$ is equal to $k!/ \prod_j j^{m_j} (m_j!)$. Thus, the rest of the proof 
follows from (\ref{A:ms}) and (\ref{AC:useful}).
\end{proof}

\section{Characters}\label{S:Characters}

Let $ f$ be a partial transformation in $ \mc{N}(n,k)$ and let $H(f)$ denote the stabilizer subgroup 
$\mt{Stab}_{\mc{S}_{n}}(f)$. As a consequence of Theorem~\ref{T:Stabilizer} we see that the vector space $o(f)$ is 
$\mc{S}_n$-equivariantly isomorphic to the vector space $\C [ \mc{S}_n / H(f)]$ and that 
$H(f)\simeq Z(\sigma)\times\mt{Stab}_{\mc{S}_{n-k}}(\tau)$ for some permutation $\sigma\in \mc{S}_k$ 
and a nilpotent partial transformation $\tau \in \mc{C}_{n-k}$. The $ \mc{S}_n $ module structure on $\C
[\mc{S}_n/H(f)]$ is given by the left translation action of $ \mc{S}_n $ 
on the cosets of $H(f)$. In other words, $o(f)$ is isomorphic to the induced trivial representation from $H(f)$ to 
$ \mc{S}_n$. Since $ H(f) $ decomposes as a direct product, the character of $ o(f)$ is equal to the product of 
characters of the induced trivial representations on the corresponding ambient symmetric subgroups. 
We paraphrase the (isomorphism of) representations in induction notation:
\begin{align}\label{A:summary}
\mt{Ind}_{ Z(\sigma) \times \mt{Stab}_{\mc{S}_{n-k}}(\tau)}^{\mc{S}_n} \mathbf{1} \simeq 
\mt{Ind}_{Z(\sigma)}^{\mc{S}_k} \mathbf{1} \otimes \mt{Ind}_{ \mt{Stab}_{\mc{S}_{n-k}}(\tau)}^{\mc{S}_{n-k}} 
\mathbf{1}.
\end{align}

{\em Notation :} Let $H \subset G$ be a subgroup. We denote by $\chi_{H,G}$ the character of the induced 
trivial representation $ \mt{Ind}_H^G \mathbf{1}_H $. If $ \tau $ is a loop-augmented forest on $m$ vertices, 
then we denote the character of the induced trivial representation from $ Stab_{\mc{S}_m}(\tau)$ to $ \mc{S
}_m $ by $\chi_{o(\tau)}$. 
\begin{Lemma}\label{L:character factorization}
We preserve our notation from the previous paragraph and let $f = \begin{pmatrix} \sigma & 0 \\ 0 & \tau 
\end{pmatrix}$ denote a partial function representing a loop-augmented forest. In this case the character
of $o(f)$ has the following factorization:
$$
\chi_{o(f)}=\chi_{Z(\sigma), \mc{S}_k}  \cdot \chi_{o(\tau)}.
$$ 
\end{Lemma}
\begin{proof}
The proof is a straightforward consequence of (\ref{A:summary}).
\end{proof}

The set of cosets of the subgroup $ Z(\sigma)$ in $\mc{S}_k$ is in natural bijection with the set of conjugates 
of $\sigma\in \mc{S}_k $. In other words, viewed as an $ \mc{S}_n $-set, $ \mc{S}_k/Z(\sigma) $ is identified 
with the conjugacy class of $\sigma$ in $ \mc{S}_k$. It is well known that the conjugacy classes (of $\mc{S}_
k $) are parametrized by the partitions (of $ k $), and a representation on a conjugacy class is irreducible. By 
the Frobenius characteristic map the irreducible representation on a conjugacy class that is indexed by a 
partition $\nu$ of $k$ is mapped to the Schur function $s_\nu$. Clearly, this argument takes care of the 
character of the induced representation $\mt{Ind}_{Z(\sigma)}^{\mc{S}_k} \mathbf{1}$ in (\ref{A:summary}). 
For the other tensor factor, we refer to our earlier work:
\cite[Corollary 6.5]{Can17} states that if $\lambda=(\lambda_1,\dots, \lambda_r)$ is a partition of $n-k$ and 
$\tau$ is a rooted forest consisting of $\lambda_1$ copies of a rooted tree $\tau_1$, $ \lambda_2$ copies of 
a rooted tree $\tau_2$ and so on, then the character of $o(\tau)$ is given by 
\begin{align}\label{A:decomposition of chi}
\chi_{o(\tau)} = (\chi^{(\lambda_1)} [ \chi_{o(\tau_1)}]) \cdot (\chi^{(\lambda_2)} [ \chi_{o(\tau_2)}])\cdot 
\cdots \cdot (\chi^{(\lambda_r)} [ \chi_{o(\tau_r)}]).
\end{align}

We summarize these observations as follows:

\begin{Theorem}\label{T:Master 1}
Let $f$ be a partial function representing a loop-augmented forest on $n$ vertices. 
Then $f$ is similar to a block diagonal matrix of the form $\begin{pmatrix}  \sigma & 0 \\ 0 & \tau 
\end{pmatrix} $ for some $ \sigma \in \mc{S}_k $ and $ \tau \in \mc{N} (n-k, 0) $. Furthermore, 
if $ \nu$, which is a partition of $k$, is the conjugacy type of $\sigma$ in $\mc{S}_k$ and if the underlying rooted forest of 
the nilpotent partial transformation $\tau$ has $\lambda_1$ copies of the rooted tree $\tau_1$, 
$\lambda_2$ copies of the rooted forest $\tau_2$ and so on, then the character of $ o(f)$ is given by 
\begin{align*}
\chi_{o(f)}  =  \chi^\nu \cdot  \chi_{o(\tau)} = \chi^\nu  \cdot  (\chi^{(\lambda_1)} [ \chi_{o(\tau_1)}]) \cdot (\chi
^{(\lambda_2)} [ \chi_{o(\tau_2)}])\cdot \cdots \cdot (\chi^{(\lambda_r)} [\chi_{o(\tau_r)}]).
\end{align*}
\end{Theorem}
\begin{proof}
The proof follows from the discussion in the previous paragraph and Lemma~\ref{L:character factorization}.
\end{proof}


\section{Sign representation}\label{S:Sign}
In~\cite[Section 7]{Can17}, the first author found a characterization of labeled rooted forests that afford the sign 
representation.

\begin{Definition}\label{D:blossoming}
A subtree $ a $ of a rooted tree is called a terminal branch (TB for short) if any vertex of $ a $ has at most 1 
successor. A maximal terminal branch (or, MTB for short) is a terminal branch that is not a subtree of any 
terminal branch other than itself. The length (or height) of a TB is the number of vertices it has. We call a 
rooted tree blossoming if all of its MTB's are of even length, or no two odd length MTB's of the same length 
are connected to the same parent. We call a non-blossoming tree {\em dry}. Finally, a rooted forest is called 
blossoming if the rooted tree that is obtained by adding a new root to the forest is a blossoming tree; 
otherwise, the forest is called dry. 
\end{Definition}

Note that if a forest is blossoming, then all of its connected components are blossoming. Also, if a single 
connected component is dry, then the whole forest is dry. In Figure~\ref{F:blossoming} we have the list of all 
blossoming forests up to $5$ vertices. 
\begin{figure}[htp]
\centering
\begin{tikzpicture}[scale=.5]

\begin{scope}[yshift=3.5cm]
\node at (0,0) {$\bullet$};
\end{scope}

\begin{scope}[yshift=0cm]
\node at (0,0) {$\bullet$};
\node at (0,1) {$\bullet$};
\draw[-, thick] (0,0) to  (0,1);
\end{scope}

\begin{scope}[yshift=-3.5cm]
\node at (0,0) {$\bullet$};
\node at (0,1) {$\bullet$};
\node at (0,2) {$\bullet$};
\node at (3,0) {$\bullet$};
\node at (3,1) {$\bullet$};
\node at (3.75,0) {$\bullet$};
\draw[-, thick] (0,0) to  (0,2);
\draw[-, thick] (3,0) to  (3,1);
\end{scope}

\begin{scope}[yshift=-8cm]
\node at (0,0) {$\bullet$};
\node at (0,1) {$\bullet$};
\node at (0,2) {$\bullet$};
\node at (0,3) {$\bullet$};
\node at (3,0) {$\bullet$};
\node at (3,1) {$\bullet$};
\node at (3,2) {$\bullet$};
\node at (3.75,0) {$\bullet$};
\node at (6,0) {$\bullet$};
\node at (6,1) {$\bullet$};
\node at (6,2) {$\bullet$};
\node at (6.75,1) {$\bullet$};
\node at (9,0) {$\bullet$};
\node at (9,1) {$\bullet$};
\node at (9.75,0) {$\bullet$};
\node at (9.75,1) {$\bullet$};
\draw[-, thick] (0,0) to  (0,3);
\draw[-, thick] (3,0) to  (3,2);
\draw[-, thick] (6,0) to  (6,2);
\draw[-, thick] (6,0) to  (6.75,1);
\draw[-, thick] (9,0) to  (9,1);
\draw[-, thick] (9.75,0) to  (9.75,1);
\end{scope}

\begin{scope}[yshift=-13.75cm]
\node at (0,0) {$\bullet$};
\node at (0,1) {$\bullet$};
\node at (0,2) {$\bullet$};
\node at (0,3) {$\bullet$};
\node at (0,4) {$\bullet$};

\node at (3,0) {$\bullet$};
\node at (3,1) {$\bullet$};
\node at (3,2) {$\bullet$};
\node at (3,3) {$\bullet$};
\node at (3.75,0) {$\bullet$};

\node at (6,0) {$\bullet$};
\node at (6,1) {$\bullet$};
\node at (6,2) {$\bullet$};
\node at (6,3) {$\bullet$};
\node at (6.75,2) {$\bullet$};

\node at (9,1) {$\bullet$};
\node at (9,2) {$\bullet$};
\node at (9.75,1) {$\bullet$};
\node at (9.75,2) {$\bullet$};
\node at (9.35,0) {$\bullet$};

\node at (12,0) {$\bullet$};
\node at (12,1) {$\bullet$};
\node at (12,2) {$\bullet$};
\node at (13,0) {$\bullet$};
\node at (12.75,1) {$\bullet$};

\node at (15,0) {$\bullet$};
\node at (15,1) {$\bullet$};
\node at (15,2) {$\bullet$};
\node at (15,3) {$\bullet$};
\node at (15.75,1) {$\bullet$};

\node at (18,0) {$\bullet$};
\node at (18,1) {$\bullet$};
\node at (18,2) {$\bullet$};
\node at (18.75,0) {$\bullet$};
\node at (18.75,1) {$\bullet$};

\node at (21,0) {$\bullet$};
\node at (21,1) {$\bullet$};
\node at (21.75,0) {$\bullet$};
\node at (21.75,1) {$\bullet$};
\node at (22.5,0) {$\bullet$};

\draw[-, thick] (0,0) to  (0,4);
\draw[-, thick] (3,0) to  (3,3);
\draw[-, thick] (6,0) to  (6,3);
\draw[-, thick] (6,1) to  (6.75,2);
\draw[-, thick] (9.35,0) to  (9,1);
\draw[-, thick] (9.35,0) to  (9.75,1);
\draw[-, thick] (9.75,1) to  (9.75,2);
\draw[-, thick] (9,1) to  (9,2);
\draw[-, thick] (12,0) to  (12,2);
\draw[-, thick] (12,0) to  (12.75,1);
\draw[-, thick] (15,0) to  (15,3);
\draw[-, thick] (15,0) to  (15.75,1);
\draw[-, thick] (18,0) to (18,2);
\draw[-, thick] (18.75,0) to (18.75,1);
\draw[-, thick] (21,0) to (21,1);
\draw[-, thick] (21.75,0) to (21.75,1);

\end{scope}

\end{tikzpicture}
\caption{Blossoming forests on $1\leq n \leq 5$ vertices.}
\label{F:blossoming}
\end{figure}
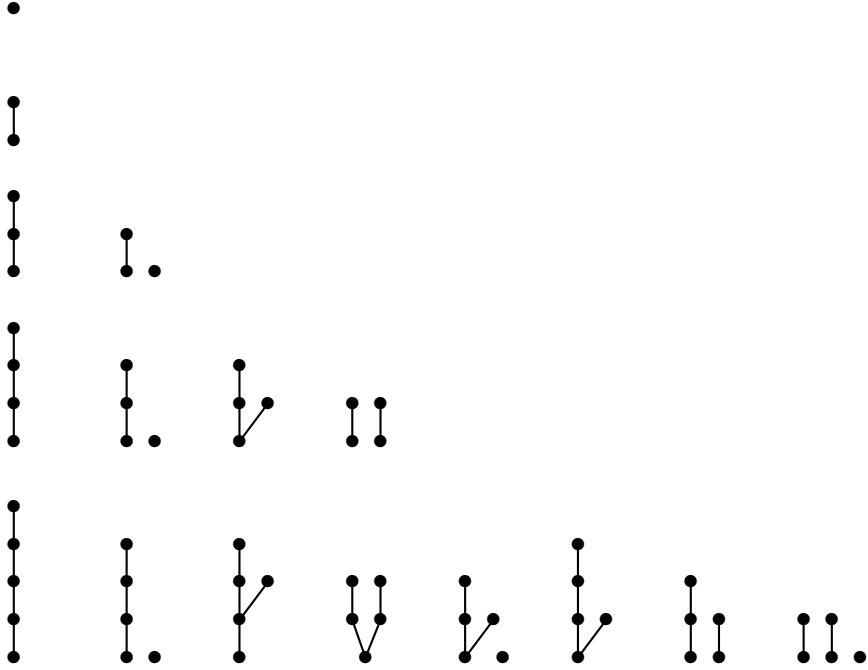

\begin{Lemma}[\cite{Can17}, Proposition 7.9]\label{L:Sign rep from Can17}
Let $\lambda$ be a partition with $|\lambda|=l > 1$ and let $\tau$ be a rooted forest on $m$ vertices.
\begin{enumerate}
\item 
If $\tau$ is a blossoming forest, then 
$
\langle s_{\lambda}[F_\tau], s_{(1^{ml})} \rangle = 
\begin{cases}
1 & \mt{ if $\lambda = (1^l)$};\\
0 & \mt{ otherwise.} 
\end{cases}
$
\item If $\tau$ is a dry forest, then $\langle s_{\lambda}[F_\tau], s_{(1^{ml})} \rangle =0$. 
\end{enumerate}
\end{Lemma}

The following consequence of the above lemma is not very difficult to see.

\begin{Corollary}[\cite{Can17}, Theorem 7.11]\label{C:Can17}
The number of blossoming forests on $n$ vertices is $2^{n-2}$. 
\end{Corollary}

We will now combine these results with our master character formula from Theorem~\ref{T:Master 1}. 
Let $f$ be a loop-augmented forest on $n$ vertices as in Theorem~\ref{T:Master 1}, that is, the partial 
transformation corresponding to $f$ is $\mc{S}_n$-conjugate to a block diagonal matrix of the form 
$\begin{pmatrix} \sigma & 0 \\ 0 & \tau \end{pmatrix}$ for some $\sigma \in \mc{S}_k$ and $\tau \in \mc{N}
(n-k, 0)$. Let $\nu$ denote the conjugacy type of $\sigma$ in $\mc{S}_k$. Then it follows from 
Theorem~\ref{T:Master 1} and Lemma~\ref{L:Sign rep from Can17} that 
\begin{align}\label{A:number of f}
\langle \chi^{(1^n)}, \chi_{o(f)} \rangle =
\begin{cases}
1 & \mt{ if $\nu = (1^k)$ and $o(\tau)$ is the odun of a blossoming forest};\\
0 & \mt{ otherwise.} 
\end{cases}
\end{align}

\begin{Theorem}\label{T:number of sign}
The number of loop-augmented forests on $n$ vertices with $k$ loops whose odun affords the sign 
representation is $2^{n-k-2}$. The total number of loop-augmented forests on $n$ vertices admitting 
the sign representation is $2^{n-1}-1$.  
\end{Theorem}
\begin{proof}
The proof of the first claim follows from Corollary~\ref{C:Can17} and the formula (\ref{A:number of f}). 
To prove the second claim, we first note that $k$ can be any number from 2 to $n$.  Since $1+2+\cdots 
+ 2^{n-2}=2^{n-1}-1$, the proof is finished. 
\end{proof}

\section{Non-diagonal idempotents}\label{S:Nondiagonal}

Let $e$ denote the $i$-th constant idempotent $c^{(i)}$. Observe that $e$ stays unchanged by multiplication 
by permutations on the right and therefore the conjugation action on $e$ is given by the left multiplication 
action; this is equivalent to the ``standard representation'' of $\mc{S}_n$ on $\C^n$. In particular, the stabilizer 
subgroup of $e$ is isomorphic to $\mc{S}_1\times \mc{S}_{n-1} \simeq \mc{S}_{n-1}$.

\begin{Definition}\label{D:standard}
Let $e$ be an idempotent with $k$ 1's, for $0<k\leq n$, on its $i$-th row and $0$'s elsewhere.  
Then there exists $\sigma \in \mc{S}_n$ such that the nonzero entries of 
$\sigma e \sigma^{-1}$ are in the first row and in the first $k$ columns. In this case, we call 
$\sigma e \sigma^{-1}$ the standard form of $e$. 
\end{Definition}
Note that a constant idempotent $c^{(r)}$ is already in its standard form. Note also that the standard form 
of an idempotent is unique. We will show in an algorithmic way that it is always possible to determine the 
permutation $\sigma$ that leads to the standard form of $e$. 
\begin{Example}
Let $e$ denote the idempotent $\begin{pmatrix} 0 & 0 & 0  \\  0 & 0 & 0 \\1& 0 & 1 \end{pmatrix}.$ 
Then the standard form of $e$ is given by $\sigma e \sigma^{-1}=\begin{pmatrix} 1& 1 & 0  \\  0 & 0 & 0 \\0&0&0 \end{pmatrix}$,
where $\sigma$ is the permutation $(2,3)(1,3)$. 
\end{Example}

We proceed to explain our algorithmic construction of the permutation $\sigma$ for the idempotent $e$ as in Definition~\ref{D:standard}.
If $i=1$, then skip the next three sentences. As $e$ is an idempotent, the $i$-th entry in the $i$-th row of 
$e$ has to be 1. Apply $\sigma_1:= (1,i)$ to bring the nonzero entry at the $(i,i)$-th position to $(1,1)$-th position. 
In fact, the application of $\sigma_1$ to $e$ interchanges the $i$-th row with the first row and interchanges 
the $i$-th column with the first column. We denote the resulting idempotent by $e_1$. Assuming $i=1$, 
we look for the first occurrence of 1 in the first row (other than the one in the first position). If this 1 is at the 
$j$-th position with $j>2$, then apply the transposition $\sigma_2:=(2,j)$ to $e_1$. $\sigma_2$ interchanges the 
second and the $j$-th rows and it interchanges the second and the $j$-th columns in $e_1$. Thus, the $2$-nd 
entry of the first row of $e_2:=\sigma_2\cdot e_1$ is now 1. We repeat this process, by choosing transpositions,
$\sigma_3,\sigma_4,\dots$ until we move all of the 1's in the first row to the left of the row. In summary, by 
conjugating $e$ by a product $\sigma_s \cdots \sigma_1$ of not all disjoint transpositions we bring $e$ to its 
standard form, the idempotent having 1's in the first $k$ entries of its first row.

\begin{Lemma}

Let $e$ be an idempotent from $\mc{P}_n$ and let $j_1 \leq \dots \leq j_r$ denote the indices of those rows 
of $e$ which have nonzero entries in them. If $\alpha_i$ ($i=1,\dots, r$) denotes the number of 1's in the 
$j_i$-th row, then there exists a permutation $\sigma\in \mc{S}_n$ such that $\sigma e \sigma^{-1}$ is a block 
diagonal matrix of the form 
$$
e = c^{(\alpha_1)} \oplus \cdots \oplus c^{(\alpha_r)} \oplus \mathbf{0}_m,
$$
where $\mathbf{0}_m$ is the square 0 matrix of dimension $m:=n-\sum \alpha_i$.
\end{Lemma}

\begin{proof}

The proof uses induction by first applying the standardization procedure outlined in the above paragraphs to 
the row $j_1$. Instead of repeating what is written above, we demonstrate this procedure by an example.  
\begin{eqnarray*}
&&
\begin{pmatrix}
0&0&0&0&0&0&0\\ 
0&1&0&1&0&1&0\\ 
0&0&0&0&0&0&0\\ 
0&0&0&0&0&0&0\\
0&0&0&0&1&0&0\\ 
0&0&0&0&0&0&0\\ 
1&0&0&0&0&0&1
\end{pmatrix}
\xrightarrow{(1,2)\cdot }
\begin{pmatrix}
1&0&0&1&0&1&0\\ 
0&0&0&0&0&0&0\\ 
0&0&0&0&0&0&0\\ 
0&0&0&0&0&0&0\\ 
0&0&0&0&1&0&0\\ 
0&0&0&0&0&0&0\\ 
0&1&0&0&0&0&1
\end{pmatrix}
\xrightarrow{(2,4)\cdot }
\begin{pmatrix}
1&1&0&0&0&1&0\\ 
0&0&0&0&0&0&0\\ 
0&0&0&0&0&0&0\\ 
0&0&0&0&0&0&0\\ 
0&0&0&0&1&0&0\\ 
0&0&0&0&0&0&0\\ 
0&0&0&1&0&0&1 
\end{pmatrix} 
\xrightarrow{(3,6)\cdot } \\
&&
\begin{pmatrix}
1&1&1&0&0&0&0\\ 
0&0&0&0&0&0&0\\ 
0&0&0&0&0&0&0\\ 
0&0&0&0&0&0&0\\ 
0&0&0&0&1&0&0\\ 
0&0&0&0&0&0&0\\ 
0&0&0&1&0&0&1 
\end{pmatrix} 
\xrightarrow{(4,5)\cdot } 
\begin{pmatrix}
1&1&1&0&0&0&0\\ 
0&0&0&0&0&0&0\\ 
0&0&0&0&0&0&0\\ 
0&0&0&1&0&0&0\\ 
0&0&0&0&0&0&0\\ 
0&0&0&0&0&0&0\\ 
0&0&0&0&1&0&1 
\end{pmatrix} 
\xrightarrow{(5,7)\cdot } 
\begin{pmatrix}
1&1&1&0&0&0&0\\ 
0&0&0&0&0&0&0\\ 
0&0&0&0&0&0&0\\ 
0&0&0&1&0&0&0\\ 
0&0&0&0&1&0&1\\ 
0&0&0&0&0&0&0\\ 
0&0&0&0&0&0&0 
\end{pmatrix} 
\xrightarrow{(6,7)\cdot } \\
&&
\begin{pmatrix}
1&1&1&0&0&0&0\\ 
0&0&0&0&0&0&0\\ 
0&0&0&0&0&0&0\\ 
0&0&0&1&0&0&0\\ 
0&0&0&0&1&1&0\\ 
0&0&0&0&0&0&0\\ 
0&0&0&0&0&0&0 
\end{pmatrix} = c^{(3)}\oplus c^{(1)} \oplus c^{(2)} \oplus \mathbf{0}_1
\end{eqnarray*}

\end{proof}

The proof of our next observation is now straightforward; it follows from the similar arguments as above.
\begin{Corollary}\label{C:decreasing form}

Let $e$ be an idempotent from $\mc{P}_n$. Then there exists a permutation $\sigma\in \mc{S}_n$ such that 
$\sigma e \sigma^{-1}$ is a block diagonal matrix of the form 
\begin{align}\label{A:std form}
\sigma e \sigma^{-1} = c^{(\beta_1)} \oplus \cdots \oplus c^{(\beta_r)} \oplus \mathbf{0}_m,
\end{align}
where $\mathbf{0}_m$ is the square 0 matrix of dimension $m:=n-\sum \beta_i$ and $\beta_1\geq 
\beta_2 \cdots \geq \beta_r$.
\end{Corollary}


We now generalize our previous definition of the ``standard form of an idempotent''. For an arbitrary 
idempotent $e\in \mc{P}_n$, we call its conjugate $\sigma e \sigma^{-1}=c^{(\beta_1)} \oplus \cdots \oplus 
c^{(\beta_r)} \oplus \mathbf{0}_m$ as in (\ref{A:std form}) the standard form of $e$ and denote it by $st(e)$. 
An idempotent $e$ is called {\em standard} if it satisfies the equation $x=st(x)$. Our next 
task is to compute the stabilizer of a standard idempotent.

\begin{Theorem}\label{T:St stabilizer}
Let $\beta_1 \geq \beta_2 \geq \cdots \geq \beta_r$ be a sequence of positive integers and let $e\in \mc{P}
_n$ be a standard idempotent in which $c^{(\beta_{1})}$ appears $m_1$ times, $c^{(\beta_{2})}$ appears 
$m_2$ times, and so on. If $\sum_{i=1}^s m_i \beta_i = n$, then the stabilizer subgroup in $\mc{S}_n$ of $e$ is 
$$
\mt{Stab}_{\mc{S}_n}(e)= (\mc{S}_{m_1} \wr \mc{S}_{\beta_1-1}) \times (\mc{S}_{m_2} \wr \mc{S}_{\beta_2-1}) \times 
\cdots \times (\mc{S}_{m_s}\wr \mc{S}_{\beta_s-1}).
$$
\end{Theorem}

\begin{proof}
We start with the simple case that $e= c^{(\beta_1)}\oplus \cdots \oplus c^{(\beta_1)} \oplus c^{(\beta_2)} 
\oplus \cdots \oplus c^{(\beta_2)}$, where $c^{(\beta_1)}$ appears $m_1$ times and $c^{(\beta_2)}$ 
appears $m_2$ times. If $\sigma \in \mc{S}_n$ stabilizes $e$, then we have $\sigma e = e \sigma$. We 
decompose $\sigma$ into 4 blocks 
$$
\sigma = 
\begin{pmatrix} 
\sigma_{11} & \sigma_{12} \\
\sigma_{21} & \sigma_{22}
\end{pmatrix}
$$
where $\sigma_{11}$ is of dimension $m_1\beta_1$ and $\sigma_{22}$ is of dimension $m_2\beta_2$. 
For simplicity of notation, let us denote $c^{(\beta_1)}\oplus \cdots \oplus c^{(\beta_1)}$ by $c_{11}$ and 
denote $c^{(\beta_2)}\oplus \cdots \oplus c^{(\beta_2)}$ by $c_{22}$.
By a simple block-matrix multiplication we observe that 
\begin{enumerate}
\item $\sigma_{11}$ commutes with $c_{11}$;
\item $\sigma_{22}$ commutes with $c_{22}$;
\item $\sigma_{12} c_{22} = c_{11} \sigma_{12}$ which implies that $\sigma_{12}$ is a $0-$matrix since 
$c^{(\beta_1)}$ and $c^{(\beta_2)}$ are of different dimensions;
\item $\sigma_{21} c_{11} = c_{22} \sigma_{21}$ which implies that $\sigma_{21}$ is a $0-$matrix since 
$c^{(\beta_1)}$ and $c^{(\beta_2)}$ are of different dimensions.  
\end{enumerate}
The last two items shows that $\mt{Stab}_{\mc{S}_n}(e) = \mt{Stab}_{\mc{S}_{m_1\beta_1}}(c_{11}) \oplus  
\mt{Stab}_{\mc{S}_{m_2\beta_2}}(c_{22})$. Thus, $\sigma_{ii}$ belongs to $\mc{S}_{m_i \beta_i}$ for $i=1,2$.

To show that $ \mt{Stab}_{\mc{S}_{m_i\beta_i}}(c_{ii}) \simeq \mc{S}_{m_i}\wr \mc{S}_{\beta_i-1} $ ($i=1,2$), 
we assume without loss of generality that $i=1$ and we decompose $\sigma_{11}$ into $m_1^2$ square 
blocks as in 
\begin{align}\label{A:row of blocks}
\sigma_{11} = 
\begin{pmatrix} 
\tau_{11} & \cdots &  \tau_{1m_1} \\
\vdots & \ddots  & \vdots \\
\tau_{m_11} & \cdots &  \tau_{m_1 m_1} 
\end{pmatrix}.
\end{align}
Note that each $\tau_{ij}$ is a rook matrix. Since $c_{11}$ has copies of $c^{(\beta_1)}$'s along its main 
diagonal and zero blocks elsewhere, it follows from the commutation relation $\sigma_{11} c_{11} = c_{11} 
\sigma_{11}$ that 
\begin{align}\label{A:necessary}
\tau_{ij} c^{(\beta_1)} = c^{(\beta_1)} \tau_{ij}
\end{align}
for all $i,j=1,\dots, m_1$. It is not difficult to verify that unless $\tau_{ij}$ is a permutation matrix with 1 at its 
$(1,1)$-th entry, in order for relation (\ref{A:necessary}) to hold, $\tau_{ij}$ must be a zero matrix. In other 
words, in each row (of blocks) in (\ref{A:row of blocks}) there exists exactly 1 block which is a permutation 
matrix of dimension $\beta_1$ with 1 at its $(1,1)$-th entry. It is not difficult to verify that a permutation matrix 
$\tau \in \mc{S}_{m_1 \beta_1}$ with $m_1^2$ square blocks just as described satisfies the commutation 
relation $\tau c_{11} = c_{11} \tau$, also. But such a permutation $\tau$ is uniquely determined by a 
permutation $\tau' \in \mc{S}_{m_1}$ describing the positions of nonzero blocks of $\tau$ and in addition 
$m_1^2$ permutation matrices $\tau_{ij}' \in \mc{S}_{\beta_1-1}$ which forms that lower-right sub-matrices 
of the nonzero blocks. This is precisely our definition of $\mc{S}_{m_1}\wr \mc{S}_{\beta_1-1}$. 
\end{proof}

\begin{Corollary}\label{C:Algebraic}
If $e\in \mc{P}_n$ is an idempotent, then there exists a multiset of integers $\{\beta_1^{m_1},\dots, 
\beta_s^{m_s}\}$ and $m\in \N$ such that 
\begin{enumerate}
\item $m+ \sum_{i=1}^s m_i \beta_i =n$;
\item $\mt{Stab}_{\mc{S}_n}(e) \simeq (\mc{S}_{m_1} \wr \mc{S}_{\beta_1}) \times (\mc{S}_{m_2} 
\wr \mc{S}_{\beta_2}) \times \cdots \times (\mc{S}_{m_s}\wr \mc{S}_{\beta_s}) \times \mc{S}_m$.
\end{enumerate}
\end{Corollary}
\begin{proof}
We know from Corollary~\ref{C:decreasing form} that by conjugating with an appropriate permutation $\sigma 
\in \mc{S}_n$ we get $ \sigma e \sigma^{-1} = \widetilde{e} \oplus \mathbf{0}_m$, where $\widetilde{e}$ is a 
standard idempotent in $\mc{P}_k$ for some $1\leq k\leq n$ and $m=n-k$. The rest of the proof follows from 
Theorem~\ref{T:St stabilizer}.
\end{proof}


\section{The conjecture}\label{S:Conjecture}

Throughout this section, when the parts of a partition $\lambda=(\lambda_1,\dots, \lambda_k)$ 
satisfy $\lambda_1=\dots = \lambda_k$, we will write $\lambda= \lambda_1^k$. 

Let $\tau$ be a rooted forest with the Frobenius character $F_\tau$. Recall from Introduction that 
adding a root to $\tau$ has the effect of multiplying $F_\tau$ with $h_1$. Recall also that, if $\tau$ 
is a rooted tree, then the Frobenius character of the rooted forest that is comprised of $m$ copies 
of $\tau$ is equal to $h_m [F_\tau]$. 
\begin{Example}\label{E:E1}
Since $h_1$ is equal to the power-sum symmetric function $p_1$, the Frobenius character of $n-1$ 
disjoint nodes (for $n\geq 2$) is $h_{n-1}[h_1]=h_{n-1}$.  
\end{Example}
\begin{Example}\label{E:E2}
The Frobenius character of the rooted tree that is obtained from $n-1$ disjoint nodes (for $n\geq 2$)
by attaching them to a root is $h_1h_{n-1}$.  
\end{Example}
\begin{Example}\label{E:E3}
The Frobenius character of the rooted forest that is obtained by taking $m$ copies of the rooted tree
of Example~\ref{E:E2} is $h_m[h_1h_{n-1}]$.  
\end{Example}
The rest of our paper is devoted to a proof of our conjectured 
extension of the Foulkes' conjecture in the special case where 
$m=2$. 
In the light of Examples~\ref{E:E1}--\ref{E:E3}, our general 
conjecture can be stated in the language 
of symmetric functions as follows.
\begin{Conjecture}\label{C:Conjecture}
If $n > m \geq 2$, then for any $\lambda$ which has more than two parts, 
$$\langle h_n[h_1h_{m-1}],s_{\lambda}\rangle \geq 
\langle h_m[h_1h_{n-1}],s_{\lambda}\rangle.$$
\end{Conjecture}

We will prove this conjecture in the case where 
$m =2$.  In particular, we shall show that 
for all $n > 2$, and $\lambda \neq (2,2n-2)$, 
\begin{equation}
\langle h_n[h_1h_1],s_{\lambda}\rangle \geq 
\langle h_2[h_1h_{n-1}],s_{\lambda}\rangle.
\end{equation}
and 
\begin{eqnarray*}
\langle h_n[h_1h_1],s_{(2,2n-2)}\rangle =2, \ \mathrm{and}\ 
\langle h_2[h_1h_{n-1}],s_{(2,2n-2)}\rangle =3. 
\end{eqnarray*}

Before we can carry out our computations, we need 
to recall Littlewood's formulas for the plethysms $h_n[h_2]$ and $h_n[e_2]$,
the Murnaghan-Nakayama rule to compute the 
product of power symmetric function $p_n$ times 
a Schur function $s_\mu$, and Chen's rule \cite{CGR} which 
allows us to compute the plethysm $p_k[h_n]$.

We say that a partition $\lambda = (\lambda_1, \ldots, \lambda_k)$ 
is even if $\lambda_i$ is even for all $i$. Then 
Littlewood \cite{L} proved the following two results. 
\begin{equation}\label{hne2}
h_n[h_2] = \sum_{\lambda \vdash 2n, \lambda \ \mathrm{even}} s_{\lambda}
\end{equation}
and 
\begin{equation}\label{hnh2}
h_n[e_2] = \sum_{\lambda \vdash 2n, \lambda'  \ \mathrm{even}} s_{\lambda}
\end{equation}

Next we introduce the notion of rim hook tableaux, special rim hook 
tabloids, and transposed special rim hook tabloids. 

Given a Ferrers diagram $F(\lambda)$ of a partition, a {\em rim hook}  $h$ of 
$\lambda$ is a consecutive sequence of cells along the northeast boundary of 
$\lambda$ such that any two consecutive cells of $h$ share an edge 
and the removal of the cells of $h$ from $\lambda$ results in a Ferrers 
diagram of another partition. We let $r(h)$ denote the number of 
rows of $h$ and $c(h)$ denote the number of columns of $h$. We say that 
$h$ is {\em special} if $h$ has at least one cell in the first column 
of $F(\lambda)$ and $h$ is {\em transposed special} if $h$ has at 
least one cell in the first row of $F(\lambda)$. 
We define the sign of $h$, $sgn(h)$, to be 
$$sgn(h) = (-1)^{r(h)-1}.$$

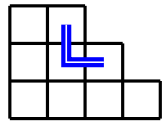
\begin{figure}[htp]
\centering
\begin{tikzpicture}[scale=.5]

 \draw[very thick] (0,0) -- (4,0);
 \draw[very thick] (0,1) -- (4,1);
 \draw[very thick] (0,2) -- (3,2);
 \draw[very thick] (0,3) -- (2,3);
 \draw[very thick] (0,3) -- (0,0);
 \draw[very thick] (1,3) -- (1,0);
  \draw[very thick] (2,3) -- (2,0);
  \draw[very thick] (3,2) -- (3,0);
    \draw[very thick] (4,1) -- (4,0);
  
  \draw[blue, double, ultra thick,-] (1.5,2.5) -- (1.5,1.5) --(2.5,1.5) ;

\end{tikzpicture}
\caption{A rim hook of $(2,3,4)$.}
\label{fig:rimhook}
\end{figure}


For example, Figure \ref{fig:rimhook} 
gives a rim hook $h$ of shape $(2,3,4)$ which is neither 
special or transposed special. Since $r(h) =2$, then 
$sgn(h) = (-1)^{2-1} = -1$. 
Figure \ref{fig:special} (a) pictures all the special rim hooks 
of $\lambda = (2,2,4)$ and Figure \ref{fig:special} (b) pictures all the 
transposed special rim hooks of $\lambda = (2,2,4)$.

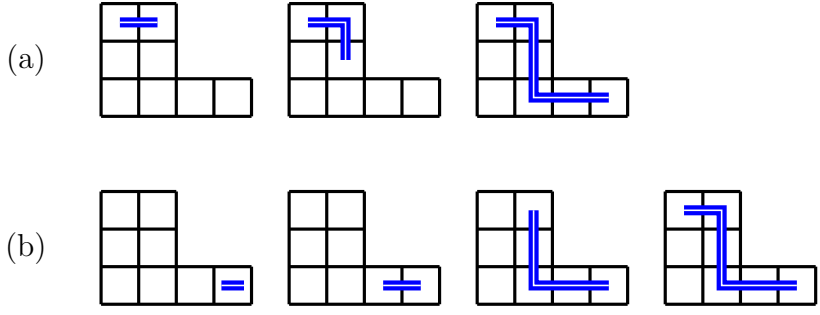
\begin{figure}[htp]
\centering
\begin{tikzpicture}[scale=.5]

\begin{scope}[yshift=3cm]
\begin{scope}[xshift=-7cm]
\node at (0,1.5) {(a)};
\end{scope}

\begin{scope}[xshift=-5cm]
 \draw[very thick] (0,0) -- (4,0);
 \draw[very thick] (0,1) -- (4,1);
 \draw[very thick] (0,2) -- (2,2);
 \draw[very thick] (0,3) -- (2,3);
 \draw[very thick] (0,3) -- (0,0);
 \draw[very thick] (1,3) -- (1,0);
  \draw[very thick] (2,3) -- (2,0);
  \draw[very thick] (3,1) -- (3,0);
  \draw[very thick] (4,1) -- (4,0);
  \draw[blue, double, ultra thick,-] (0.5,2.5) -- (1.5,2.5) ;
  \end{scope}
  
  \begin{scope}
 \draw[very thick] (0,0) -- (4,0);
 \draw[very thick] (0,1) -- (4,1);
 \draw[very thick] (0,2) -- (2,2);
 \draw[very thick] (0,3) -- (2,3);
 \draw[very thick] (0,3) -- (0,0);
 \draw[very thick] (1,3) -- (1,0);
  \draw[very thick] (2,3) -- (2,0);
  \draw[very thick] (3,1) -- (3,0);
  \draw[very thick] (4,1) -- (4,0);
  
  \draw[blue, double, ultra thick,-] (0.5,2.5) -- (1.5,2.5)  -- (1.5,1.5) ;
  \end{scope}
  
  \begin{scope}[xshift=5cm]
 \draw[very thick] (0,0) -- (4,0);
 \draw[very thick] (0,1) -- (4,1);
 \draw[very thick] (0,2) -- (2,2);
 \draw[very thick] (0,3) -- (2,3);
 \draw[very thick] (0,3) -- (0,0);
 \draw[very thick] (1,3) -- (1,0);
  \draw[very thick] (2,3) -- (2,0);
  \draw[very thick] (3,1) -- (3,0);
  \draw[very thick] (4,1) -- (4,0);
  
  \draw[blue, double, ultra thick,-] (0.5,2.5) -- (1.5,2.5)  -- (1.5,0.5) -- (3.5,.5) ;
  \end{scope}
  \end{scope}

\begin{scope}[yshift=-2cm]
\begin{scope}[xshift=-7cm]
\node at (0,1.5) {(b)};
\end{scope}

\begin{scope}[xshift=-5cm]
 \draw[very thick] (0,0) -- (4,0);
 \draw[very thick] (0,1) -- (4,1);
 \draw[very thick] (0,2) -- (2,2);
 \draw[very thick] (0,3) -- (2,3);
 \draw[very thick] (0,3) -- (0,0);
 \draw[very thick] (1,3) -- (1,0);
  \draw[very thick] (2,3) -- (2,0);
  \draw[very thick] (3,1) -- (3,0);
  \draw[very thick] (4,1) -- (4,0);
  \draw[blue, double, ultra thick,-] (3.2,0.5) -- (3.8,.5) ;
  \end{scope}
  
  \begin{scope}
 \draw[very thick] (0,0) -- (4,0);
 \draw[very thick] (0,1) -- (4,1);
 \draw[very thick] (0,2) -- (2,2);
 \draw[very thick] (0,3) -- (2,3);
 \draw[very thick] (0,3) -- (0,0);
 \draw[very thick] (1,3) -- (1,0);
  \draw[very thick] (2,3) -- (2,0);
  \draw[very thick] (3,1) -- (3,0);
  \draw[very thick] (4,1) -- (4,0);
  
  \draw[blue, double, ultra thick,-] (2.5,0.5) -- (3.5,.5);
  \end{scope}
  
  \begin{scope}[xshift=5cm]
 \draw[very thick] (0,0) -- (4,0);
 \draw[very thick] (0,1) -- (4,1);
 \draw[very thick] (0,2) -- (2,2);
 \draw[very thick] (0,3) -- (2,3);
 \draw[very thick] (0,3) -- (0,0);
 \draw[very thick] (1,3) -- (1,0);
  \draw[very thick] (2,3) -- (2,0);
  \draw[very thick] (3,1) -- (3,0);
  \draw[very thick] (4,1) -- (4,0);
  
  \draw[blue, double, ultra thick,-] (1.5,2.5) -- (1.5,0.5)  -- (3.5,.5) ;
  \end{scope}

  \begin{scope}[xshift=10cm]
 \draw[very thick] (0,0) -- (4,0);
 \draw[very thick] (0,1) -- (4,1);
 \draw[very thick] (0,2) -- (2,2);
 \draw[very thick] (0,3) -- (2,3);
 \draw[very thick] (0,3) -- (0,0);
 \draw[very thick] (1,3) -- (1,0);
  \draw[very thick] (2,3) -- (2,0);
  \draw[very thick] (3,1) -- (3,0);
  \draw[very thick] (4,1) -- (4,0);
  
  \draw[blue, double, ultra thick,-] (.5,2.5) -- (1.5,2.5) -- (1.5,0.5)  -- (3.5,.5)  ;
  \end{scope}
  \end{scope}

\end{tikzpicture}
\caption{The special and transposed special rim hooks of $(2,2,4)$.}
\label{fig:special}
\end{figure}


The Murnaghan-Nakayama rule gives us a rule for the Schur function 
expansion of a power symmetric function $p_n$ times a Schur 
function $s_{\mu}$. That is, 
\begin{equation}\label{MN}
p_n s_{\mu} = \sum_{\overset{\mu \subseteq \lambda}{ \lambda/\mu \ 
\mathrm{is \ a \ rim \ hook \ of} \ \lambda \ \mathrm{of\  size} \ n}} 
sgn(\lambda/\mu) s_\lambda. 
\end{equation}

A {\em rim hook tabloid} of shape $\lambda$ and type 
$\mu =(\mu_1, \ldots , \mu_k)$ is a filling of the Ferrers diagram 
of $\lambda$ with rim hooks $(h_1, \ldots, h_k)$ such 
that $(|h_1|, \ldots, |h_k|)$ is a rearrangement of 
$(\mu_1, \ldots , \mu_k)$ where $|h_i|$ denotes the number of 
cells of $h_i$. To be more precise, one can think of a rim hook 
tabloid $T$ as a sequences of shapes 
$\{\emptyset = \lambda^{(0)} \subset \lambda^{(1)} \subset \cdots 
\subset \lambda^{(k)} = \lambda\}$ such that for all $i \geq 1$, 
$\lambda^{(i)}/\lambda^{(i-1)}$ is a rim hook of $\lambda^{(i)}$ 
and $(|\lambda^{(1)}/\lambda^{(0)}|, |\lambda^{(2)}/\lambda^{(1)}|, \ldots, 
|\lambda^{(k)}/\lambda^{(k-1)}|)$ is a rearrangement of 
$(\mu_1, \ldots , \mu_k)$. $T$ is called a {\em special ($t$-special) rim 
hook tabloid} if for all $i \geq 1$,  
$\lambda^{(i)}/\lambda^{(i-1)}$ is a special (transposed special) rim hook. 
The sign of $T$, $sgn(T)$ is defined to be 
$$sgn(T) = \prod_{i=1}^k sgn(\lambda^{(i)}/\lambda^{(i-1)}).$$

We emphasize however that the rim hook tabloid $T$ of shape $\lambda$ is  
the filling of $F(\lambda)$ and is not the sequence of 
shapes $\{\emptyset = \lambda^{(0)} \subset \lambda^{(1)} \subset \cdots 
\subset \lambda^{(k)} = \lambda\}$. That is, in Figure 
\ref{fig:double} pictures a rim hook tabloid $T$ of shape 
$\lambda = (2,2,4)$ and type $(2,3,3)$ whose sign 
is $(-1)^{2-1}(-1)^{2-1}(-1)^{1-1} =1$. There are two 
sequences of shapes that can be associated to $T$, namely,
\begin{eqnarray*}
T &=& \{\emptyset \subset (1,2) \subset (2,2,2) \subset (2,2,4)\} \ 
\mathrm{and} \\
T &=&  \{\emptyset \subset (1,2) \subset (1,4) \subset (2,2,4)\}
\end{eqnarray*}
depending on whether one takes to top rim hook to be the 
second or third rim hook. Of course, if $T$ is 
a special rim hook tabloid or a $t$-special rim hook tabloid, then 
there is a unique sequence of shapes 
$\{\emptyset = \lambda^{(0)} \subset \lambda^{(1)} \subset \cdots 
\subset \lambda^{(k)} = \lambda\}$  that can be associated to 
$T$.

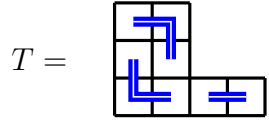
\begin{figure}[htp]
\centering
\begin{tikzpicture}[scale=.5]

\begin{scope}[xshift=-2cm]
\node at (0,1.5) {$T=$};
\end{scope}

\begin{scope}
 \draw[very thick] (0,0) -- (4,0);
 \draw[very thick] (0,1) -- (4,1);
 \draw[very thick] (0,2) -- (2,2);
 \draw[very thick] (0,3) -- (2,3);
 \draw[very thick] (0,3) -- (0,0);
 \draw[very thick] (1,3) -- (1,0);

\draw[very thick] (2,3) -- (2,0);
 \draw[very thick] (3,1) -- (3,0);
 \draw[very thick] (4,1) -- (4,0);
  
\draw[blue, double, ultra thick,-] (0.5,2.5) -- (1.5,2.5) -- (1.5,1.5) ;
\draw[blue, double, ultra thick,-] (0.5,1.5) -- (0.5,0.5) -- (1.5,0.5) ;
\draw[blue, double, ultra thick,-] (2.5,0.5) -- (3.5,0.5) ;
    
  \end{scope}

\end{tikzpicture}
\caption{A rim hook tabloid with associated sequences 
of shapes.}
\label{fig:double}
\end{figure}


Chen's rule \cite{CGR} states that 
\begin{equation}\label{Chen}
\langle p_k[h_n], s_{\lambda}  \rangle = 
\begin{cases} 
sgn(T) & \mbox{if $T$ is a special rim hook tabloid of 
shape $\lambda$ and type $k^n$;} \\
0 & \mbox{otherwise.} 
\end{cases}
\end{equation}

Thus to compute $p_k[s_n] = s_n[p_k]$, we need only generate 
all the $t$-special rim hook tabloids $T$ of type $k^n$ and 
replace $T$ by $sgn(T) s_{sh(T)}$ where $sh(T)$ denotes the shape of 
$T$. For example, Figure \ref{fig:Chen} pictures all the $t$-special 
rim hook tabloids of type $3^2$ where instead of drawing a Ferrers 
diagram, we have indicated the cells of the Ferrers diagram by 
dots. Thus 
$p_3[h_2] = s_2[p_3] = s_{(2^3)} -s_{(1,2,3)}+s_{(1^2,4)}+s_{(3,3)}
-s_{(1,5)}+s_{(6)}$.

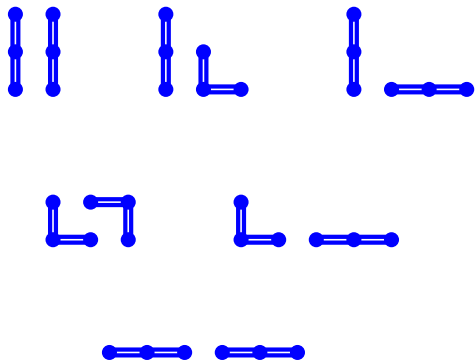
\begin{figure}[htp]
\centering
\begin{tikzpicture}[scale=.5]

\begin{scope}[yshift=4cm]
\begin{scope}[xshift=-4cm]
\draw[blue, double, ultra thick,-] (0,0) -- (0,2) ;
\fill[blue] (0,0) circle (.2cm);
\fill[blue] (0,1) circle (.2cm);
\fill[blue] (0,2) circle (.2cm);
\draw[blue, double, ultra thick,-] (1,0) -- (1,2) ;
\fill[blue] (1,0) circle (.2cm);
\fill[blue] (1,1) circle (.2cm);
\fill[blue] (1,2) circle (.2cm);
\end{scope}
\begin{scope}
\draw[blue, double, ultra thick,-] (0,0) -- (0,2) ;
\fill[blue] (0,0) circle (.2cm);
\fill[blue] (0,1) circle (.2cm);
\fill[blue] (0,2) circle (.2cm);
\draw[blue, double, ultra thick,-] (1,1) -- (1,0) -- (2,0) ;
\fill[blue] (1,1) circle (.2cm);
\fill[blue] (1,0) circle (.2cm);
\fill[blue] (2,0) circle (.2cm);
\end{scope}
\begin{scope}[xshift=5cm]
\draw[blue, double, ultra thick,-] (0,0) -- (0,2) ;
\fill[blue] (0,0) circle (.2cm);
\fill[blue] (0,1) circle (.2cm);
\fill[blue] (0,2) circle (.2cm);
\draw[blue, double, ultra thick,-] (1,0) -- (3,0) ;
\fill[blue] (1,0) circle (.2cm);
\fill[blue] (2,0) circle (.2cm);
\fill[blue] (3,0) circle (.2cm);
\end{scope}
\end{scope}

\begin{scope}
\begin{scope}[xshift=-3cm]
\draw[blue, double, ultra thick,-] (0,1) -- (0,0) -- (1,0);
\fill[blue] (0,1) circle (.2cm);
\fill[blue] (0,0) circle (.2cm);
\fill[blue] (1,0) circle (.2cm);
\draw[blue, double, ultra thick,-] (1,1) -- (2,1) -- (2,0) ;
\fill[blue] (1,1) circle (.2cm);
\fill[blue] (2,1) circle (.2cm);
\fill[blue] (2,0) circle (.2cm);
\end{scope}
\begin{scope}[xshift=2cm]
\draw[blue, double, ultra thick,-] (0,1) -- (0,0) -- (1,0);
\fill[blue] (0,1) circle (.2cm);
\fill[blue] (0,0) circle (.2cm);
\fill[blue] (1,0) circle (.2cm);
\end{scope}
\begin{scope}[xshift=4cm]
\draw[blue, double, ultra thick,-] (0,0) -- (2,0) ;
\fill[blue] (0,0) circle (.2cm);
\fill[blue] (1,0) circle (.2cm);
\fill[blue] (2,0) circle (.2cm);
\end{scope}
\end{scope}

\begin{scope}[yshift=-3cm]
\begin{scope}[xshift=-1.5cm]
\draw[blue, double, ultra thick,-] (0,0) -- (2,0) ;
\fill[blue] (0,0) circle (.2cm);
\fill[blue] (1,0) circle (.2cm);
\fill[blue] (2,0) circle (.2cm);
\end{scope}
\begin{scope}[xshift=1.5cm]
\draw[blue, double, ultra thick,-] (0,0) -- (2,0) ;
\fill[blue] (0,0) circle (.2cm);
\fill[blue] (1,0) circle (.2cm);
\fill[blue] (2,0) circle (.2cm);
\end{scope}
\end{scope}
 
\end{tikzpicture}
\caption{The $t$-special rim hook tabloids of type $(3^2)$.}
\label{fig:Chen}
\end{figure}


The computation of $p_2[h_n]$ is particularly straightforward 
via Chen's rule. That is, a $t$-special rim hook tabloid 
of type $2^n$ consists of bunch of vertical rim hooks 
followed by a bunch of horizontal rim hooks as pictured 
in Figure \ref{fig:Chen2}.

\begin{figure}[htp]
\centering
\begin{tikzpicture}[scale=.5]

\begin{scope}
\draw[very thick] (0,0) -- (11,0);
\draw[very thick] (0,1) -- (11,1);
\draw[very thick] (0,2) -- (5,2);
\draw[very thick] (0,2) -- (0,0);
\draw[very thick] (1,2) -- (1,0);
\draw[very thick] (2,2) -- (2,0);
\draw[very thick] (3,2) -- (3,0);
\draw[very thick] (4,2) -- (4,0);
\draw[very thick] (5,2) -- (5,0);

\draw[very thick] (6,1) -- (6,0);
\draw[very thick] (7,1) -- (7,0);
\draw[very thick] (8,1) -- (8,0);
\draw[very thick] (9,1) -- (9,0);
\draw[very thick] (10,1) -- (10,0);
\draw[very thick] (11,1) -- (11,0);

\draw[blue, double, ultra thick,-] (0.5,1.5) -- (0.5,0.5) ; 
\draw[blue, double, ultra thick,-] (1.5,1.5) -- (1.5,0.5) ; 
\draw[blue, double, ultra thick,-] (2.5,1.5) -- (2.5,0.5) ; 
\draw[blue, double, ultra thick,-] (3.5,1.5) -- (3.5,0.5) ; 
\draw[blue, double, ultra thick,-] (4.5,1.5) -- (4.5,0.5) ; 

\draw[blue, double, ultra thick,-] (5.5,0.5) -- (6.5,0.5) ; 
\draw[blue, double, ultra thick,-] (7.5,0.5) -- (8.5,0.5) ; 
\draw[blue, double, ultra thick,-] (9.5,0.5) -- (10.5,0.5) ; 

\end{scope}

\end{tikzpicture}
\caption{$t$-special rim hook tabloids of type $2^n$.}
\label{fig:Chen2}
\end{figure}
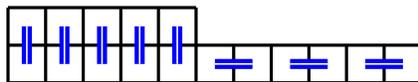


It easily follows that 
\begin{equation}\label{p2hn}
p_2[s_n]= \sum_{0 \leq a \leq n} (-1)^a s_{(a,2n-a)}.
\end{equation}

We want to compare $h_2[h_1h_{n-1}]$ to $h_n[h_1h_1]$. 
Obviously, the two expressions are equal when $n=2$. 
Our first goal is to prove the following theorem which 
gives an explicit formula for the Schur function expansion 
of  $h_2[h_1h_{n-1}]$.

{\em Notation:} For any 
logical statement $A$, we set 
$$\chi(A)=
\begin{cases}
1 & \text { if $A$ is true};\\
0 & \text { if $A$ is false}.
\end{cases}
$$

\begin{Theorem}\label{thm:1} For $n \geq 3$, 
\begin{eqnarray*}
h_2[h_1h_{n-1}] &=& 
\sum_{1 \leq a \leq n-1, a \ \mathrm{odd}} s_{(1^2,a,2n-2 -a)} + 
\sum_{2 \leq a \leq n-1,a \ \mathrm{even}} s_{(2,a,2n-2 -a)} + \\
&&s_{(1^2,2n-2)}+ \sum_{2 \leq a \leq n-1} 2 s_{(1,a,2n-1 -a)} + \\
&& s_{(1,2n-1)}+\sum_{2 \leq a \leq n-1, a \ \mathrm{even}} 3 s_{(a,2n-a)}+
\sum_{2 \leq a \leq n-1, a \ \mathrm{odd}} s_{(a,2n-a)} + \\
&& 2 \chi(n \ \mathrm{even}) s_{(n,n)} + s_{(2n)}.
\end{eqnarray*}
\end{Theorem}

\begin{proof}

Note that by (\ref{p2hn}) and the Pieri rules we have 
\begin{eqnarray*}
h_2[h_1h_{n-1}] &=& \left(\frac{1}{2}(p_1^2+p_2)\right)[h_1h_{n-1}] \\
&=& \frac{1}{2}h_1^2[h_1h_{n-1}]+\frac{1}{2}p_2[h_1h_{n-1}] \\
&=&\frac{1}{2} h_1^2h_{n-1}^2 + \frac{1}{2}p_2 p_2[h_{n-1}] \\
&=& \frac{1}{2}\left(h_1^2  (\sum_{0 \leq a \leq n-1}  s_{(a,2n-2-a)})+ 
p_2 (\sum_{0 \leq a \leq n-1} (-1)^a s_{(a,2n-2-a)})\right).
\end{eqnarray*}

It is easy to see from the Pieri rules and the Murnaghan-Nakayama 
rules that the only shapes $\lambda$ of 
Schur functions $s_{\lambda}$ that can arise 
in the Schur function expansions of 
products of the from $h_1^2s_{(a,2n-2-a)}$ and 
$p_2s_{(a,2n-2-a)}$ are $(1^2,a,b)$,  $(2,a,b)$, $(1,a,b)$, 
$(a,b)$, or $(2n)$.
We will consider each type of shape in turn. It follows 
from the Pieri rules that a product $h_1^2 s_{(a,2n-2-a)}$ 
is the sum over all $s_{\lambda}$ where we first add 
single cell on the outside of $(a,2n-2-a)$ to get 
a shape $\mu$, which we will indicate 
by a cell containing the number 1, and then adding a single 
cell on the outside of the $\mu$, which we will indicate by cell containing 2.

By the Murnaghan-Nakayama rule, a product of the form 
$p_2s_{(a,2n-2-a)}$ is the sum over all $\pm s_{\lambda}$ that arises 
by putting a rim hook of length 2 along the outside of 
$(a,2n-2-a)$, which we will indicate by a rim hook colored red,
 where the sign is plus if the rim hook is horizontal 
and $-1$ if the rim hook is vertical. \\
\ \\
{\bf Case 1.} $\lambda = (1^2,a,b)$ where $1 \leq a \leq n-1$.\\
\ \\
In this case, there is only one way to get a term 
of the form $s_{(1^2,a,2n-2-a)}$, where $1 \leq a \leq n-1$, from 
the products $h_1^2 h_{n-1}^2$ and $p_2[h_{n-1}]p_2$ which 
are pictured in Figure \ref{fig:Case1}.
That is, it is easy to see that $s_{(1^2,a,2n-2-a)}$ can only arise 
in one way from $h_1^2 h_{n-1}^2$ which comes the product 
$h_1^2 s_{(a,2n-2-a)}$ by placing the cells with 1 and 2 as 
pictured at the top of Figure \ref{fig:Case1}. Similarly, 
 $s_{(1^2,a,2n-2-a)}$ can only arise 
in one way from  which comes the product 
$p_2 (-1)^a s_{(a,2n-2-a)}$ by placing the rim hook as  
pictured in red at the bottom of Figure \ref{fig:Case1}. 
Hence we get a contribution of $(-1)^{a+1}s_{(1^2,a,2n-2-a)}$ 
from  $p_2 p_2[h_{n-1}]$.  These two cancel out if $1 \leq a \leq n-1$ 
and $a$ is even and they give 
a contribution of $\frac{1}{2} 2 s_{(1^2,a,2n-2-a)} 
=s_{(1^2,a,2n-2-a)}$ if $a$ is odd.

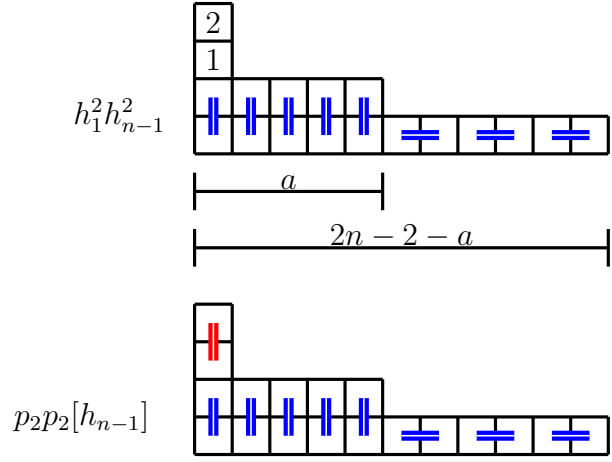
\begin{figure}[htp]
\centering
\begin{tikzpicture}[scale=.5]
\begin{scope}[yshift=4cm]
\begin{scope}[xshift=-2cm]
\node at (0,1) {$h_1^2 h_{n-1}^2$};
\end{scope}
\begin{scope}
\node at (.5,3.5) {$2$};
\node at (.5,2.5) {$1$};
\draw[very thick] (0,0) -- (11,0);
\draw[very thick] (0,1) -- (11,1);
\draw[very thick] (0,2) -- (5,2);
\draw[very thick] (0,4) -- (1,4);
\draw[very thick] (0,3) -- (1,3);
\draw[very thick] (0,4) -- (0,0);
\draw[very thick] (1,4) -- (1,0);
\draw[very thick] (2,2) -- (2,0);
\draw[very thick] (3,2) -- (3,0);
\draw[very thick] (4,2) -- (4,0);
\draw[very thick] (5,2) -- (5,0);
\draw[very thick] (6,1) -- (6,0);
\draw[very thick] (7,1) -- (7,0);
\draw[very thick] (8,1) -- (8,0);
\draw[very thick] (9,1) -- (9,0);
\draw[very thick] (10,1) -- (10,0);
\draw[very thick] (11,1) -- (11,0);
\draw[blue, double, ultra thick,-] (0.5,1.5) -- (0.5,0.5) ; 
\draw[blue, double, ultra thick,-] (1.5,1.5) -- (1.5,0.5) ; 
\draw[blue, double, ultra thick,-] (2.5,1.5) -- (2.5,0.5) ; 
\draw[blue, double, ultra thick,-] (3.5,1.5) -- (3.5,0.5) ; 
\draw[blue, double, ultra thick,-] (4.5,1.5) -- (4.5,0.5) ; 
\draw[blue, double, ultra thick,-] (5.5,0.5) -- (6.5,0.5) ; 
\draw[blue, double, ultra thick,-] (7.5,0.5) -- (8.5,0.5) ; 
\draw[blue, double, ultra thick,-] (9.5,0.5) -- (10.5,0.5) ; 
\draw[very thick] (0,-1) -- (5,-1);
\draw[very thick] (0,-.5) -- (0,-1.5);
\draw[very thick] (5,-0.5) -- (5,-1.5);
\node at (2.5,-.75) {$a$};
\draw[very thick] (0,-2.5) -- (11,-2.5);
\draw[very thick] (0,-3) -- (0,-2);
\draw[very thick] (11,-3) -- (11,-2);
\node at (5.5,-2.2) {$2n-2-a$};
\end{scope}
\end{scope}
\begin{scope}[yshift=-4cm]
\begin{scope}[xshift=-3cm]
\node at (0,1) {$p_2  p_2[h_{n-1}]$};
\end{scope}
\begin{scope}
\draw[very thick] (0,0) -- (11,0);
\draw[very thick] (0,1) -- (11,1);
\draw[very thick] (0,2) -- (5,2);
\draw[very thick] (0,4) -- (1,4);
\draw[very thick] (0,3) -- (1,3);
\draw[very thick] (0,4) -- (0,0);
\draw[very thick] (1,4) -- (1,0);
\draw[very thick] (2,2) -- (2,0);
\draw[very thick] (3,2) -- (3,0);
\draw[very thick] (4,2) -- (4,0);
\draw[very thick] (5,2) -- (5,0);
\draw[very thick] (6,1) -- (6,0);
\draw[very thick] (7,1) -- (7,0);
\draw[very thick] (8,1) -- (8,0);
\draw[very thick] (9,1) -- (9,0);
\draw[very thick] (10,1) -- (10,0);
\draw[very thick] (11,1) -- (11,0);
\draw[red, double, ultra thick,-] (0.5,3.5) -- (0.5,2.5) ; 
\draw[blue, double, ultra thick,-] (0.5,1.5) -- (0.5,0.5) ; 
\draw[blue, double, ultra thick,-] (1.5,1.5) -- (1.5,0.5) ; 
\draw[blue, double, ultra thick,-] (2.5,1.5) -- (2.5,0.5) ; 
\draw[blue, double, ultra thick,-] (3.5,1.5) -- (3.5,0.5) ; 
\draw[blue, double, ultra thick,-] (4.5,1.5) -- (4.5,0.5) ; 
\draw[blue, double, ultra thick,-] (5.5,0.5) -- (6.5,0.5) ; 
\draw[blue, double, ultra thick,-] (7.5,0.5) -- (8.5,0.5) ; 
\draw[blue, double, ultra thick,-] (9.5,0.5) -- (10.5,0.5) ; 
\end{scope}
\end{scope}
\end{tikzpicture}
\caption{The diagrams for Case 1.}
\label{fig:Case1}
\end{figure}


Thus the Schur functions in Case 1 contribute 
$$\sum_{1 \leq a \leq n-1, a \ \mathrm{odd} }s_{(1^2,a,2n-2-a)}$$ 
to $h_2[h_1 h_{n-1}]$. \\
\ \\
{\bf Case 2.} $\lambda = (2,a,b)$ where $2 \leq a \leq n-1$.\\
\ \\
In this case, there is only one way to get a term 
of the form $s_{(1^2,a,2n-2-a)}$, where $1 \leq a \leq n-1$, from 
the products $h_1^2 h_{n-1}^2$ and $p_2[h_{n-1}]p_2$ which 
are pictured in Figure \ref{fig:Case2}. That is, it is easy to see that $s_{(1^2,a,2n-2-a)}$ can only arise 
in one way from $h_1^2 h_{n-1}^2$ which comes the product 
$h_1^2 s_{(a,2n-2-a)}$ by placing the cells with 1 and 2 as 
pictured at the top of Figure \ref{fig:Case2}. Similarly, 
 $s_{(1^2,a,2n-2-a)}$ can only arise 
in one way from  which comes the product 
$p_2 (-1)^a s_{(a,2n-2-a)}$ by placing the rim hook as  
pictured in red at the bottom of Figure \ref{fig:Case2}. 
Hence we get a contribution of $(-1)^{a}s_{(1^2,a,2n-2-a)}$ 
from  $p_2 p_2[h_{n-1}]$.  This two cancel out if $1 \leq a \leq n-1$ 
and $a$ is odd and a contribution of $\frac{1}{2} 2 s_{(1^2,a,2n-2-a)} 
=s_{(1^2,a,2n-2-a)}$ if $a$ is even.

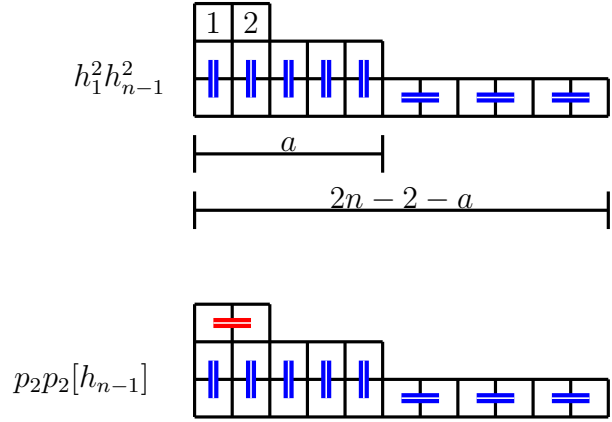
\begin{figure}[htp]
\centering
\begin{tikzpicture}[scale=.5]
\begin{scope}[yshift=4cm]
\begin{scope}[xshift=-2cm]
\node at (0,1) {$h_1^2 h_{n-1}^2$};
\end{scope}
\begin{scope}
\node at (1.5,2.5) {$2$};
\node at (.5,2.5) {$1$};
\draw[very thick] (0,0) -- (11,0);
\draw[very thick] (0,1) -- (11,1);
\draw[very thick] (0,2) -- (5,2);
\draw[very thick] (0,3) -- (2,3);
\draw[very thick] (0,3) -- (0,0);
\draw[very thick] (1,3) -- (1,0);
\draw[very thick] (2,3) -- (2,0);
\draw[very thick] (3,2) -- (3,0);
\draw[very thick] (4,2) -- (4,0);
\draw[very thick] (5,2) -- (5,0);
\draw[very thick] (6,1) -- (6,0);
\draw[very thick] (7,1) -- (7,0);
\draw[very thick] (8,1) -- (8,0);
\draw[very thick] (9,1) -- (9,0);
\draw[very thick] (10,1) -- (10,0);
\draw[very thick] (11,1) -- (11,0);
\draw[blue, double, ultra thick,-] (0.5,1.5) -- (0.5,0.5) ; 
\draw[blue, double, ultra thick,-] (1.5,1.5) -- (1.5,0.5) ; 
\draw[blue, double, ultra thick,-] (2.5,1.5) -- (2.5,0.5) ; 
\draw[blue, double, ultra thick,-] (3.5,1.5) -- (3.5,0.5) ; 
\draw[blue, double, ultra thick,-] (4.5,1.5) -- (4.5,0.5) ; 
\draw[blue, double, ultra thick,-] (5.5,0.5) -- (6.5,0.5) ; 
\draw[blue, double, ultra thick,-] (7.5,0.5) -- (8.5,0.5) ; 
\draw[blue, double, ultra thick,-] (9.5,0.5) -- (10.5,0.5) ; 
\draw[very thick] (0,-1) -- (5,-1);
\draw[very thick] (0,-.5) -- (0,-1.5);
\draw[very thick] (5,-0.5) -- (5,-1.5);
\node at (2.5,-.75) {$a$};
\draw[very thick] (0,-2.5) -- (11,-2.5);
\draw[very thick] (0,-3) -- (0,-2);
\draw[very thick] (11,-3) -- (11,-2);
\node at (5.5,-2.2) {$2n-2-a$};
\end{scope}
\end{scope}
\begin{scope}[yshift=-4cm]
\begin{scope}[xshift=-3cm]
\node at (0,1) {$p_2  p_2[h_{n-1}]$};
\end{scope}
\begin{scope}
\draw[very thick] (0,0) -- (11,0);
\draw[very thick] (0,1) -- (11,1);
\draw[very thick] (0,2) -- (5,2);
\draw[very thick] (0,3) -- (2,3);
\draw[very thick] (0,3) -- (0,0);
\draw[very thick] (1,3) -- (1,0);
\draw[very thick] (2,3) -- (2,0);
\draw[very thick] (3,2) -- (3,0);
\draw[very thick] (4,2) -- (4,0);
\draw[very thick] (5,2) -- (5,0);
\draw[very thick] (6,1) -- (6,0);
\draw[very thick] (7,1) -- (7,0);
\draw[very thick] (8,1) -- (8,0);
\draw[very thick] (9,1) -- (9,0);
\draw[very thick] (10,1) -- (10,0);
\draw[very thick] (11,1) -- (11,0);
\draw[red, double, ultra thick,-] (0.5,2.5) -- (1.5,2.5) ; 
\draw[blue, double, ultra thick,-] (0.5,1.5) -- (0.5,0.5) ; 
\draw[blue, double, ultra thick,-] (1.5,1.5) -- (1.5,0.5) ; 
\draw[blue, double, ultra thick,-] (2.5,1.5) -- (2.5,0.5) ; 
\draw[blue, double, ultra thick,-] (3.5,1.5) -- (3.5,0.5) ; 
\draw[blue, double, ultra thick,-] (4.5,1.5) -- (4.5,0.5) ; 
\draw[blue, double, ultra thick,-] (5.5,0.5) -- (6.5,0.5) ; 
\draw[blue, double, ultra thick,-] (7.5,0.5) -- (8.5,0.5) ; 
\draw[blue, double, ultra thick,-] (9.5,0.5) -- (10.5,0.5) ; 
\end{scope}
\end{scope}
\end{tikzpicture}
\caption{The diagrams for Case 2.}
\label{fig:Case2}
\end{figure}


Thus the Schur functions in Case 2 contribute 
$$\sum_{2 \leq a \leq n-1, a \ \mathrm{even} }s_{(2,a,2n-2-a)}$$ 
to $h_2[h_1 h_{n-1}]$. \\
\ \\
{\bf Case 3.} $\lambda = (1,a,2n-a -1)$ where $1 \leq a \leq n-1$. \\
\ \\
In this case, we must consider two subcases.  \\
\ \\
{\bf Subcase 3.1.} $a =1$.\\
\\
Then $s_{(1^2,2n-2-a)}$ can arise in exactly 
three ways from $h_1^2 h_{n-1}^2$, once from $h_1^2 s_{2n-2}$ 
and two ways from $h_1^2 s_{(1,2n-3)}$. These are pictured 
at the top of Figure \ref{fig:Case31}.  
$s_{(1^2,2n-2-a)}$ can only arise in one way from $p_2 p_2[h_{n-1}]$ which 
from the product $p_2s_{2n-2}$ and comes with a sign of $-1$. This 
possibility is pictured at the bottom Figure \ref{fig:Case31}. 
Thus $s_{(1^2,2n-2)}$ appears once in the expansion of $h_2[h_1 h_{n-1}]$.

\begin{figure}[htp]
\centering
\begin{tikzpicture}[scale=.5]
\begin{scope}[yshift=5cm]
\begin{scope}[xshift=-3cm]
\node at (0,1) {$h_1^2 h_{n-1}^2$};
\end{scope}
\begin{scope}
\node at (0.5,1.5) {$1$};
\node at (0.5,2.5) {$2$};
\draw[very thick] (0,0) -- (10,0);
\draw[very thick] (0,1) -- (10,1);
\draw[very thick] (0,2) -- (1,2);
\draw[very thick] (0,3) -- (1,3);
\draw[very thick] (0,3) -- (0,0);
\draw[very thick] (1,3) -- (1,0);
\draw[very thick] (2,1) -- (2,0);
\draw[very thick] (3,1) -- (3,0);
\draw[very thick] (4,1) -- (4,0);
\draw[very thick] (5,1) -- (5,0);
\draw[very thick] (6,1) -- (6,0);
\draw[very thick] (7,1) -- (7,0);
\draw[very thick] (8,1) -- (8,0);
\draw[very thick] (9,1) -- (9,0);
\draw[very thick] (10,1) -- (10,0);
\end{scope}
\begin{scope}[yshift=-3.75cm,xshift=-6cm]
\node at (9.5,0.5) {$1$};
\node at (0.5,2.5) {$2$};
\draw[very thick] (0,0) -- (10,0);
\draw[very thick] (0,1) -- (10,1);
\draw[very thick] (0,2) -- (1,2);
\draw[very thick] (0,3) -- (1,3);
\draw[very thick] (0,3) -- (0,0);
\draw[very thick] (1,3) -- (1,0);
\draw[very thick] (2,1) -- (2,0);
\draw[very thick] (3,1) -- (3,0);
\draw[very thick] (4,1) -- (4,0);
\draw[very thick] (5,1) -- (5,0);
\draw[very thick] (6,1) -- (6,0);
\draw[very thick] (7,1) -- (7,0);
\draw[very thick] (8,1) -- (8,0);
\draw[very thick] (9,1) -- (9,0);
\draw[very thick] (10,1) -- (10,0);
\end{scope}
\begin{scope}[yshift=-3.75cm,xshift=6cm]
\node at (9.5,0.5) {$2$};
\node at (0.5,2.5) {$1$};
\draw[very thick] (0,0) -- (10,0);
\draw[very thick] (0,1) -- (10,1);
\draw[very thick] (0,2) -- (1,2);
\draw[very thick] (0,3) -- (1,3);
\draw[very thick] (0,3) -- (0,0);
\draw[very thick] (1,3) -- (1,0);
\draw[very thick] (2,1) -- (2,0);
\draw[very thick] (3,1) -- (3,0);
\draw[very thick] (4,1) -- (4,0);
\draw[very thick] (5,1) -- (5,0);
\draw[very thick] (6,1) -- (6,0);
\draw[very thick] (7,1) -- (7,0);
\draw[very thick] (8,1) -- (8,0);
\draw[very thick] (9,1) -- (9,0);
\draw[very thick] (10,1) -- (10,0);
\end{scope}
\end{scope}
\begin{scope}[yshift=-5cm]
\begin{scope}[xshift=-3cm]
\node at (0,1) {$p_2  p_2[h_{n-1}]$};
\end{scope}
\begin{scope}
\draw[very thick] (0,0) -- (10,0);
\draw[very thick] (0,1) -- (10,1);
\draw[very thick] (0,2) -- (1,2);
\draw[very thick] (0,3) -- (1,3);
\draw[very thick] (0,3) -- (0,0);
\draw[very thick] (1,3) -- (1,0);
\draw[very thick] (2,1) -- (2,0);
\draw[very thick] (3,1) -- (3,0);
\draw[very thick] (4,1) -- (4,0);
\draw[very thick] (5,1) -- (5,0);
\draw[very thick] (6,1) -- (6,0);
\draw[very thick] (7,1) -- (7,0);
\draw[very thick] (8,1) -- (8,0);
\draw[very thick] (9,1) -- (9,0);
\draw[very thick] (10,1) -- (10,0);
\draw[red, double, ultra thick,-] (0.5,1.5) -- (0.5,2.5) ; 
\draw[blue, double, ultra thick,-] (0.5,0.5) -- (1.5,0.5) ; 
\draw[blue, double, ultra thick,-] (2.5,0.5) -- (3.5,0.5) ; 
\draw[blue, double, ultra thick,-] (4.5,0.5) -- (5.5,0.5) ; 
\draw[blue, double, ultra thick,-] (6.5,0.5) -- (7.5,0.5) ; 
\draw[blue, double, ultra thick,-] (8.5,0.5) -- (9.5,0.5) ; 
\end{scope}
\end{scope}
\end{tikzpicture}
\caption{The diagrams for Case 3.1.}
\label{fig:Case31}
\end{figure}
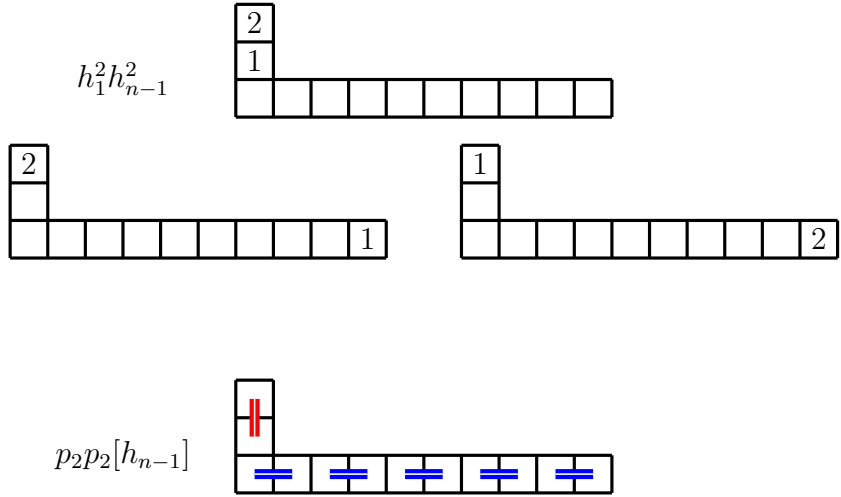


\ \\
\noindent {\bf Subcase 3.2.} $2 \leq a \leq n-1$. 
\\ \\
Then it is easy to see that we can not obtain a Schur function of shape 
$(1,a,2n-1-a)$ by adding a rim hook of size two on the outside 
of two row shape.  Thus $s_{(1,a,2n-1-a)}$ does not appear in the expansion 
of $p_2 p_2[h_{n-1}]$.  However it can appear in the expansion of 
$h_1^2 h_{n-1}$ in 4 different ways, namely, 
twice from the product  product $h_1^2 s_{(a-1,2n-a-1)}$ and twice 
from the product $h_1^2 s_{(a,2n-a-2)}$. These are pictured in 
Figure \ref{fig:Case32}.  Thus in Subcase 3.2, 
we get an additional contribution of 
$2 \sum_{2 \leq a \leq n-1} s_{(1,a,2n-1-a)}$.

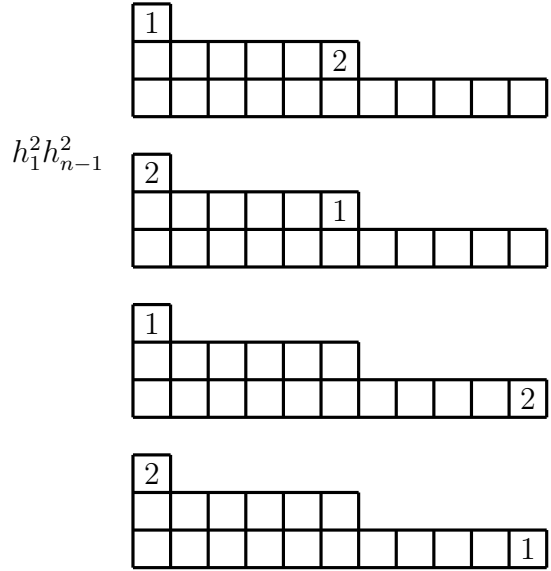
\begin{figure}[htp]
\centering
\begin{tikzpicture}[scale=.5]
\begin{scope}[xshift=-2cm,yshift=6cm]
\node at (0,1) {$h_1^2 h_{n-1}^2$};
\end{scope}
\begin{scope}[yshift=8cm]
\node at (5.5,1.5) {$2$};
\node at (0.5,2.5) {$1$};
\draw[very thick] (0,0) -- (11,0);
\draw[very thick] (0,1) -- (11,1);
\draw[very thick] (0,2) -- (6,2);
\draw[very thick] (0,3) -- (1,3);
\draw[very thick] (0,3) -- (0,0);
\draw[very thick] (1,3) -- (1,0);
\draw[very thick] (2,2) -- (2,0);
\draw[very thick] (3,2) -- (3,0);
\draw[very thick] (4,2) -- (4,0);
\draw[very thick] (5,2) -- (5,0);
\draw[very thick] (6,2) -- (6,0);
\draw[very thick] (7,1) -- (7,0);
\draw[very thick] (8,1) -- (8,0);
\draw[very thick] (9,1) -- (9,0);
\draw[very thick] (10,1) -- (10,0);
\draw[very thick] (11,1) -- (11,0);
\end{scope}
\begin{scope}[yshift=4cm]
\node at (5.5,1.5) {$1$};
\node at (0.5,2.5) {$2$};
\draw[very thick] (0,0) -- (11,0);
\draw[very thick] (0,1) -- (11,1);
\draw[very thick] (0,2) -- (6,2);
\draw[very thick] (0,3) -- (1,3);
\draw[very thick] (0,3) -- (0,0);
\draw[very thick] (1,3) -- (1,0);
\draw[very thick] (2,2) -- (2,0);
\draw[very thick] (3,2) -- (3,0);
\draw[very thick] (4,2) -- (4,0);
\draw[very thick] (5,2) -- (5,0);
\draw[very thick] (6,2) -- (6,0);
\draw[very thick] (7,1) -- (7,0);
\draw[very thick] (8,1) -- (8,0);
\draw[very thick] (9,1) -- (9,0);
\draw[very thick] (10,1) -- (10,0);
\draw[very thick] (11,1) -- (11,0);
\end{scope}
\begin{scope}
\node at (10.5,0.5) {$2$};
\node at (0.5,2.5) {$1$};
\draw[very thick] (0,0) -- (11,0);
\draw[very thick] (0,1) -- (11,1);
\draw[very thick] (0,2) -- (6,2);
\draw[very thick] (0,3) -- (1,3);
\draw[very thick] (0,3) -- (0,0);
\draw[very thick] (1,3) -- (1,0);
\draw[very thick] (2,2) -- (2,0);
\draw[very thick] (3,2) -- (3,0);
\draw[very thick] (4,2) -- (4,0);
\draw[very thick] (5,2) -- (5,0);
\draw[very thick] (6,2) -- (6,0);
\draw[very thick] (7,1) -- (7,0);
\draw[very thick] (8,1) -- (8,0);
\draw[very thick] (9,1) -- (9,0);
\draw[very thick] (10,1) -- (10,0);
\draw[very thick] (11,1) -- (11,0);
\end{scope}
\begin{scope}[yshift=-4cm]
\node at (10.5,0.5) {$1$};
\node at (0.5,2.5) {$2$};
\draw[very thick] (0,0) -- (11,0);
\draw[very thick] (0,1) -- (11,1);
\draw[very thick] (0,2) -- (6,2);
\draw[very thick] (0,3) -- (1,3);
\draw[very thick] (0,3) -- (0,0);
\draw[very thick] (1,3) -- (1,0);
\draw[very thick] (2,2) -- (2,0);
\draw[very thick] (3,2) -- (3,0);
\draw[very thick] (4,2) -- (4,0);
\draw[very thick] (5,2) -- (5,0);
\draw[very thick] (6,2) -- (6,0);
\draw[very thick] (7,1) -- (7,0);
\draw[very thick] (8,1) -- (8,0);
\draw[very thick] (9,1) -- (9,0);
\draw[very thick] (10,1) -- (10,0);
\draw[very thick] (11,1) -- (11,0);
\end{scope}
\end{tikzpicture}
\caption{The diagrams for Case 3.2.}
\label{fig:Case32}
\end{figure}


Thus the Schur functions in Case 3 contribute 
$$s_{(1^2,2n-2)}+ 2 \sum_{2 \leq a \leq n-1 }s_{(1,a,2n-1-a)}$$ 
to $h_2[h_1 h_{n-1}]$. \\
\ \\
{\bf Case 4.} $\lambda = (a,2n-a)$ where $1 \leq a \leq n$.\\
\ \\
In this case, we have three subcases. \\
\ \\
{\bf Subcase 4.1.} $a =1$. \\
\ \\
In this case we have three ways to obtain $s_{(1,2n-1)}$ from 
$h_1^2 h_{n-1}^2$, namely, two ways from the product 
$h_1^2 s_{(2n-1)}$ and one way from the product 
$h_1^2 s_{(1,2n-3)}$. These three possibilities are pictured at 
the top of Figure \ref{fig:Case41}. There is only one way 
to obtain  $s_{(1,2n-1)}$ from 
$p_2 p_2[ h_{n-1}]$, namely, from the product $p_2 s_{(1,2n-3)}$ which 
comes with a sign of $-1$. Thus Subcase 4.1 contributes a 
factor of $s_{(1,2n-1)}$.

\begin{figure}[htp]
\centering
\begin{tikzpicture}[scale=.5]
\begin{scope}[yshift=4.5cm]
\begin{scope}[xshift=-3cm]
\node at (0,1) {$h_1^2 h_{n-1}^2$};
\end{scope}
\begin{scope}
\node at (9.5,0.5) {$1$};
\node at (10.5,0.5) {$2$};
\draw[very thick] (0,0) -- (11,0);
\draw[very thick] (0,1) -- (11,1);
\draw[very thick] (0,2) -- (1,2);
\draw[very thick] (0,0) -- (0,2);
\draw[very thick] (1,0) -- (1,2);
\draw[very thick] (2,1) -- (2,0);
\draw[very thick] (3,1) -- (3,0);
\draw[very thick] (4,1) -- (4,0);
\draw[very thick] (5,1) -- (5,0);
\draw[very thick] (6,1) -- (6,0);
\draw[very thick] (7,1) -- (7,0);
\draw[very thick] (8,1) -- (8,0);
\draw[very thick] (9,1) -- (9,0);
\draw[very thick] (10,1) -- (10,0);
\draw[very thick] (11,1) -- (11,0);
\end{scope}
\begin{scope}[yshift=-3.2cm,xshift=-6cm]
\node at (10.5,0.5) {$1$};
\node at (0.5,1.5) {$2$};
\draw[very thick] (0,0) -- (11,0);
\draw[very thick] (0,1) -- (11,1);
\draw[very thick] (0,2) -- (1,2);
\draw[very thick] (0,0) -- (0,2);
\draw[very thick] (1,0) -- (1,2);
\draw[very thick] (2,1) -- (2,0);
\draw[very thick] (3,1) -- (3,0);
\draw[very thick] (4,1) -- (4,0);
\draw[very thick] (5,1) -- (5,0);
\draw[very thick] (6,1) -- (6,0);
\draw[very thick] (7,1) -- (7,0);
\draw[very thick] (8,1) -- (8,0);
\draw[very thick] (9,1) -- (9,0);
\draw[very thick] (10,1) -- (10,0);
\draw[very thick] (11,1) -- (11,0);
\end{scope}
\begin{scope}[yshift=-3.2cm,xshift=6cm]
\node at (10.5,0.5) {$2$};
\node at (0.5,1.5) {$1$};
\draw[very thick] (0,0) -- (11,0);
\draw[very thick] (0,1) -- (11,1);
\draw[very thick] (0,2) -- (1,2);
\draw[very thick] (0,0) -- (0,2);
\draw[very thick] (1,0) -- (1,2);
\draw[very thick] (2,1) -- (2,0);
\draw[very thick] (3,1) -- (3,0);
\draw[very thick] (4,1) -- (4,0);
\draw[very thick] (5,1) -- (5,0);
\draw[very thick] (6,1) -- (6,0);
\draw[very thick] (7,1) -- (7,0);
\draw[very thick] (8,1) -- (8,0);
\draw[very thick] (9,1) -- (9,0);
\draw[very thick] (10,1) -- (10,0);
\draw[very thick] (11,1) -- (11,0);
\end{scope}
\end{scope}

\begin{scope}[yshift=-4cm]
\begin{scope}[xshift=-3cm]
\node at (0,1) {$p_2  p_2[h_{n-1}]$};
\end{scope}
\begin{scope}
\draw[very thick] (0,0) -- (11,0);
\draw[very thick] (0,1) -- (11,1);
\draw[very thick] (0,2) -- (1,2);
\draw[very thick] (0,0) -- (0,2);
\draw[very thick] (1,0) -- (1,2);
\draw[very thick] (2,1) -- (2,0);
\draw[very thick] (3,1) -- (3,0);
\draw[very thick] (4,1) -- (4,0);
\draw[very thick] (5,1) -- (5,0);
\draw[very thick] (6,1) -- (6,0);
\draw[very thick] (7,1) -- (7,0);
\draw[very thick] (8,1) -- (8,0);
\draw[very thick] (9,1) -- (9,0);
\draw[very thick] (10,1) -- (10,0);
\draw[very thick] (11,1) -- (11,0);
\draw[blue, double, ultra thick,-] (0.5,1.5) -- (0.5,0.5) ; 
\draw[blue, double, ultra thick,-] (1.5,0.5) -- (2.5,0.5) ; 
\draw[blue, double, ultra thick,-] (3.5,0.5) -- (4.5,0.5) ; 
\draw[blue, double, ultra thick,-] (5.5,0.5) -- (6.5,0.5) ; 
\draw[blue, double, ultra thick,-] (7.5,0.5) -- (8.5,0.5) ; 
\draw[red, double, ultra thick,-] (9.5,0.5) -- (10.5,0.5) ; 
\end{scope}
\end{scope}
\end{tikzpicture}
\caption{The diagrams for Case 4.1.}
\label{fig:Case41}
\end{figure}
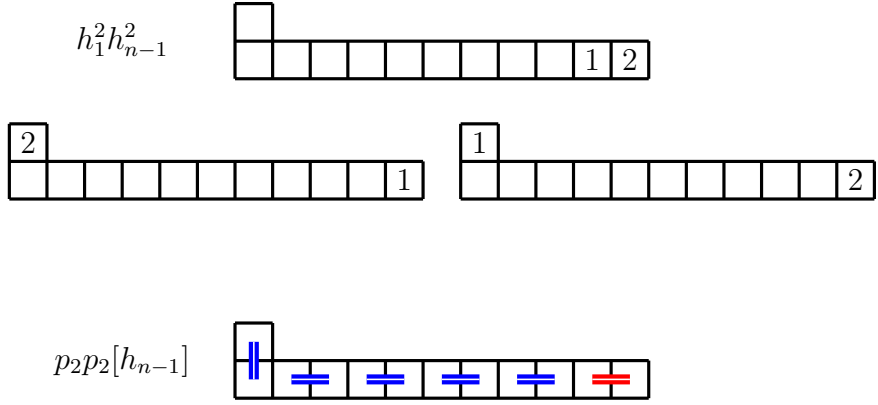

\ \\
\noindent
{\bf Subcase 4.2.} $a =n$. \\
\ \\
In this case we have two ways to obtain $s_{(n,n)}$ from 
$h_1^2 h_{n-1}^2$, namely, one way from the product 
$h_1^2 s_{(n-2,n)}$ and one way from the product 
$h_1^2 s_{(n-1,n-1)}$. These possibilities are pictured at 
the top of Figure \ref{fig:Case42}. There are two ways 
to obtain  $s_{(n,n)}$ from 
$p_2p_2[ h_{n-1}]$, namely, once from the product $p_2s_{(n-2,n)}$ which 
comes with a sign of $(-1)^{n-2}$ and once from the product 
$p_2 s_{(n-1,n-1)}$ which comes with a sign of $(-1)^n$. These possibilities are pictured at 
the bottom of Figure \ref{fig:Case42}.
Thus Subcase 4.2 contributes a 
factor of $2s_{(n,n)}$ if $n$ is even and 0 if $n$ is odd.

\begin{figure}[htp]
\centering
\begin{tikzpicture}[scale=.5]
\begin{scope}[yshift=2.5cm]
\begin{scope}[xshift=-2cm, yshift=1cm]
\node at (0,1) {$h_1^2 h_{n-1}^2$};
\end{scope}
\begin{scope}[yshift=1.5cm]
\node at (7.5,1.5) {$1$};
\node at (8.5,1.5) {$2$};
\draw[very thick] (0,0) -- (9,0);
\draw[very thick] (0,1) -- (9,1);
\draw[very thick] (0,2) -- (9,2);
\draw[very thick] (0,0) -- (0,2);
\draw[very thick] (1,0) -- (1,2);
\draw[very thick] (2,0) -- (2,2);
\draw[very thick] (3,0) -- (3,2);
\draw[very thick] (4,0) -- (4,2);
\draw[very thick] (5,0) -- (5,2);
\draw[very thick] (6,0) -- (6,2);
\draw[very thick] (7,0) -- (7,2);
\draw[very thick] (8,0) -- (8,2);
\draw[very thick] (9,0) -- (9,2);
\end{scope}
\begin{scope}[yshift=-1cm]
\node at (8.5,0.5) {$1$};
\node at (8.5,1.5) {$2$};
\draw[very thick] (0,0) -- (9,0);
\draw[very thick] (0,1) -- (9,1);
\draw[very thick] (0,2) -- (9,2);
\draw[very thick] (0,0) -- (0,2);
\draw[very thick] (1,0) -- (1,2);
\draw[very thick] (2,0) -- (2,2);
\draw[very thick] (3,0) -- (3,2);
\draw[very thick] (4,0) -- (4,2);
\draw[very thick] (5,0) -- (5,2);
\draw[very thick] (6,0) -- (6,2);
\draw[very thick] (7,0) -- (7,2);
\draw[very thick] (8,0) -- (8,2);
\draw[very thick] (9,0) -- (9,2);
\end{scope}
\end{scope}
\begin{scope}[yshift=-3cm]
\begin{scope}[xshift=-3cm, yshift=0cm]
\node at (0,1) {$p_2  p_2[h_{n-1}]$};
\end{scope}
\begin{scope}[yshift=1cm]
\draw[very thick] (0,0) -- (9,0);
\draw[very thick] (0,1) -- (9,1);
\draw[very thick] (0,2) -- (9,2);
\draw[very thick] (0,0) -- (0,2);
\draw[very thick] (1,0) -- (1,2);
\draw[very thick] (2,0) -- (2,2);
\draw[very thick] (3,0) -- (3,2);
\draw[very thick] (4,0) -- (4,2);
\draw[very thick] (5,0) -- (5,2);
\draw[very thick] (6,0) -- (6,2);
\draw[very thick] (7,0) -- (7,2);
\draw[very thick] (8,0) -- (8,2);
\draw[very thick] (9,0) -- (9,2);
\draw[blue, double, ultra thick,-] (0.5,1.5) -- (0.5,0.5) ; 
\draw[blue, double, ultra thick,-] (1.5,1.5) -- (1.5,0.5) ; 
\draw[blue, double, ultra thick,-] (2.5,1.5) -- (2.5,0.5) ; 
\draw[blue, double, ultra thick,-] (3.5,1.5) -- (3.5,0.5) ; 
\draw[blue, double, ultra thick,-] (4.5,1.5) -- (4.5,0.5) ; 
\draw[blue, double, ultra thick,-] (5.5,1.5) -- (5.5,0.5) ; 
\draw[blue, double, ultra thick,-] (6.5,1.5) -- (6.5,0.5) ; 
\draw[blue, double, ultra thick,-] (7.5,0.5) -- (8.5,0.5) ; 
\draw[red, double, ultra thick,-] (7.5,1.5) -- (8.5,1.5) ; 
\end{scope}
\begin{scope}[yshift=-1.5cm]
\draw[very thick] (0,0) -- (9,0);
\draw[very thick] (0,1) -- (9,1);
\draw[very thick] (0,2) -- (9,2);
\draw[very thick] (0,0) -- (0,2);
\draw[very thick] (1,0) -- (1,2);
\draw[very thick] (2,0) -- (2,2);
\draw[very thick] (3,0) -- (3,2);
\draw[very thick] (4,0) -- (4,2);
\draw[very thick] (5,0) -- (5,2);
\draw[very thick] (6,0) -- (6,2);
\draw[very thick] (7,0) -- (7,2);
\draw[very thick] (8,0) -- (8,2);
\draw[very thick] (9,0) -- (9,2);
\draw[blue, double, ultra thick,-] (0.5,1.5) -- (0.5,0.5) ; 
\draw[blue, double, ultra thick,-] (1.5,1.5) -- (1.5,0.5) ; 
\draw[blue, double, ultra thick,-] (2.5,1.5) -- (2.5,0.5) ; 
\draw[blue, double, ultra thick,-] (3.5,1.5) -- (3.5,0.5) ; 
\draw[blue, double, ultra thick,-] (4.5,1.5) -- (4.5,0.5) ; 
\draw[blue, double, ultra thick,-] (5.5,1.5) -- (5.5,0.5) ; 
\draw[blue, double, ultra thick,-] (6.5,1.5) -- (6.5,0.5) ; 
\draw[blue, double, ultra thick,-] (7.5,1.5) -- (7.5,0.5) ; 
\draw[red, double, ultra thick,-] (8.5,1.5) -- (8.5,0.5) ; 
\end{scope}
\end{scope}
\end{tikzpicture}
\caption{The diagrams for Case 4.2.}
\label{fig:Case42}
\end{figure}
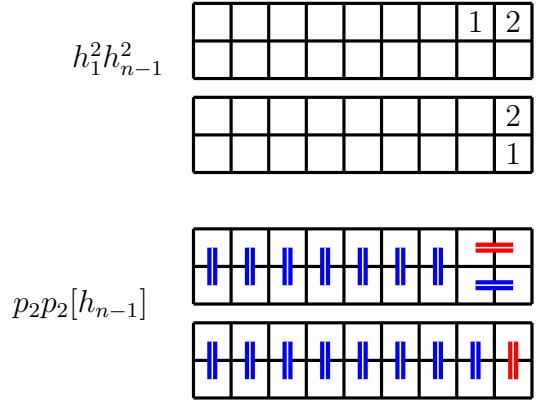

\ \\
\noindent {\bf Subcase 4.3.} $2 \leq a \leq n-1$. \\
\ \\
In this case we have four ways to obtain $s_{(a,2n-a)}$ from 
$h_1^2 h_{n-1}^2$, namely, once from the product 
$h_1^2 s_{(a-2,2n-a)}$, once from the product  $h_1^2 s_{(a,2n-a-2)}$, 
and twice from the product 
$h_1^2 s_{(a-1,2n-a-1)}$. These possibilities are pictured at 
the top of Figure \ref{fig:Case43}. There are two ways 
to obtain  $s_{(a,2n-1)}$ from 
$p_2p_2[ h_{n-1}]$, namely, once from the product $p_2s_{(a-2,2n-a)}$ which 
comes with a sign of $(-1)^{a-2}$ and once from the product 
$p_2 s_{(a,2n-a-2)}$ which comes with a sign of $(-1)^a$. These possibilities are pictured at 
the bottom of Figure \ref{fig:Case43}.
Thus Subcase 4.3 contributes a 
factor of 
$$\sum_{2 \leq a \leq n-1, a \ \mathrm{even}} 3 s_{(a,2n-a)} +
\sum_{2 \leq a \leq n-1, a \ \mathrm{odd}}  s_{(a,2n-a)}.$$

\begin{figure}[htp]
\centering
\begin{tikzpicture}[scale=.5]
\begin{scope}[yshift=3.5cm]
\begin{scope}[xshift=-2cm, yshift=2cm]
\node at (0,1) {$h_1^2 h_{n-1}^2$};
\end{scope}
\begin{scope}[yshift=2.5cm]
\node at (5.5,1.5) {$1$};
\node at (6.5,1.5) {$2$};
\draw[very thick] (0,0) -- (11,0);
\draw[very thick] (0,1) -- (11,1);
\draw[very thick] (0,2) -- (7,2);
\draw[very thick] (0,0) -- (0,2);
\draw[very thick] (1,0) -- (1,2);
\draw[very thick] (2,0) -- (2,2);
\draw[very thick] (3,0) -- (3,2);
\draw[very thick] (4,0) -- (4,2);
\draw[very thick] (5,0) -- (5,2);
\draw[very thick] (6,0) -- (6,2);
\draw[very thick] (7,0) -- (7,2);
\draw[very thick] (8,0) -- (8,1);
\draw[very thick] (9,0) -- (9,1);
\draw[very thick] (10,0) -- (10,1);
\draw[very thick] (11,0) -- (11,1);
\end{scope}
\begin{scope}[yshift=0cm]
\node at (10.5,0.5) {$2$};
\node at (9.5,0.5) {$1$};
\draw[very thick] (0,0) -- (11,0);
\draw[very thick] (0,1) -- (11,1);
\draw[very thick] (0,2) -- (7,2);
\draw[very thick] (0,0) -- (0,2);
\draw[very thick] (1,0) -- (1,2);
\draw[very thick] (2,0) -- (2,2);
\draw[very thick] (3,0) -- (3,2);
\draw[very thick] (4,0) -- (4,2);
\draw[very thick] (5,0) -- (5,2);
\draw[very thick] (6,0) -- (6,2);
\draw[very thick] (7,0) -- (7,2);
\draw[very thick] (8,0) -- (8,1);
\draw[very thick] (9,0) -- (9,1);
\draw[very thick] (10,0) -- (10,1);
\draw[very thick] (11,0) -- (11,1);
\end{scope}
\begin{scope}[yshift=-2.5cm]
\node at (10.5,0.5) {$2$};
\node at (6.5,1.5) {$1$};
\draw[very thick] (0,0) -- (11,0);
\draw[very thick] (0,1) -- (11,1);
\draw[very thick] (0,2) -- (7,2);
\draw[very thick] (0,0) -- (0,2);
\draw[very thick] (1,0) -- (1,2);
\draw[very thick] (2,0) -- (2,2);
\draw[very thick] (3,0) -- (3,2);
\draw[very thick] (4,0) -- (4,2);
\draw[very thick] (5,0) -- (5,2);
\draw[very thick] (6,0) -- (6,2);
\draw[very thick] (7,0) -- (7,2);
\draw[very thick] (8,0) -- (8,1);
\draw[very thick] (9,0) -- (9,1);
\draw[very thick] (10,0) -- (10,1);
\draw[very thick] (11,0) -- (11,1);
\end{scope}
\begin{scope}[yshift=-5cm]
\node at (10.5,0.5) {$1$};
\node at (6.5,1.5) {$2$};
\draw[very thick] (0,0) -- (11,0);
\draw[very thick] (0,1) -- (11,1);
\draw[very thick] (0,2) -- (7,2);
\draw[very thick] (0,0) -- (0,2);
\draw[very thick] (1,0) -- (1,2);
\draw[very thick] (2,0) -- (2,2);
\draw[very thick] (3,0) -- (3,2);
\draw[very thick] (4,0) -- (4,2);
\draw[very thick] (5,0) -- (5,2);
\draw[very thick] (6,0) -- (6,2);
\draw[very thick] (7,0) -- (7,2);
\draw[very thick] (8,0) -- (8,1);
\draw[very thick] (9,0) -- (9,1);
\draw[very thick] (10,0) -- (10,1);
\draw[very thick] (11,0) -- (11,1);
\end{scope}
\end{scope}
\begin{scope}[yshift=-6cm]
\begin{scope}[xshift=-2cm, yshift=0cm]
\node at (0,1) {$p_2  p_2[h_{n-1}]$};
\end{scope}
\begin{scope}[yshift=1cm]
\draw[very thick] (0,0) -- (11,0);
\draw[very thick] (0,1) -- (11,1);
\draw[very thick] (0,2) -- (7,2);
\draw[very thick] (0,0) -- (0,2);
\draw[very thick] (1,0) -- (1,2);
\draw[very thick] (2,0) -- (2,2);
\draw[very thick] (3,0) -- (3,2);
\draw[very thick] (4,0) -- (4,2);
\draw[very thick] (5,0) -- (5,2);
\draw[very thick] (6,0) -- (6,2);
\draw[very thick] (7,0) -- (7,2);
\draw[very thick] (8,0) -- (8,1);
\draw[very thick] (9,0) -- (9,1);
\draw[very thick] (10,0) -- (10,1);
\draw[very thick] (11,0) -- (11,1);
\draw[blue, double, ultra thick,-] (0.5,1.5) -- (0.5,0.5) ; 
\draw[blue, double, ultra thick,-] (1.5,1.5) -- (1.5,0.5) ; 
\draw[blue, double, ultra thick,-] (2.5,1.5) -- (2.5,0.5) ; 
\draw[blue, double, ultra thick,-] (3.5,1.5) -- (3.5,0.5) ; 
\draw[blue, double, ultra thick,-] (4.5,1.5) -- (4.5,0.5) ; 
\draw[blue, double, ultra thick,-] (5.5,0.5) -- (6.5,0.5) ; 
\draw[red, double, ultra thick,-] (5.5,1.5) -- (6.5,1.5) ; 
\draw[blue, double, ultra thick,-] (7.5,0.5) -- (8.5,0.5) ; 
\draw[blue, double, ultra thick,-] (9.5,0.5) -- (10.5,0.5) ; 
\end{scope}
\begin{scope}[yshift=-1.5cm]
\draw[very thick] (0,0) -- (11,0);
\draw[very thick] (0,1) -- (11,1);
\draw[very thick] (0,2) -- (7,2);
\draw[very thick] (0,0) -- (0,2);
\draw[very thick] (1,0) -- (1,2);
\draw[very thick] (2,0) -- (2,2);
\draw[very thick] (3,0) -- (3,2);
\draw[very thick] (4,0) -- (4,2);
\draw[very thick] (5,0) -- (5,2);
\draw[very thick] (6,0) -- (6,2);
\draw[very thick] (7,0) -- (7,2);
\draw[very thick] (8,0) -- (8,1);
\draw[very thick] (9,0) -- (9,1);
\draw[very thick] (10,0) -- (10,1);
\draw[very thick] (11,0) -- (11,1);
\draw[blue, double, ultra thick,-] (0.5,1.5) -- (0.5,0.5) ; 
\draw[blue, double, ultra thick,-] (1.5,1.5) -- (1.5,0.5) ; 
\draw[blue, double, ultra thick,-] (2.5,1.5) -- (2.5,0.5) ; 
\draw[blue, double, ultra thick,-] (3.5,1.5) -- (3.5,0.5) ; 
\draw[blue, double, ultra thick,-] (4.5,1.5) -- (4.5,0.5) ; 
\draw[blue, double, ultra thick,-] (5.5,1.5) -- (5.5,0.5) ; 
\draw[blue, double, ultra thick,-] (6.5,1.5) -- (6.5,0.5) ; 
\draw[blue, double, ultra thick,-] (7.5,0.5) -- (8.5,0.5) ; 
\draw[red, double, ultra thick,-] (9.5,0.5) -- (10.5,0.5) ; 
\end{scope}
\end{scope}
\end{tikzpicture}
\caption{The diagrams for Case 4.3.}
\label{fig:Case43}
\end{figure}
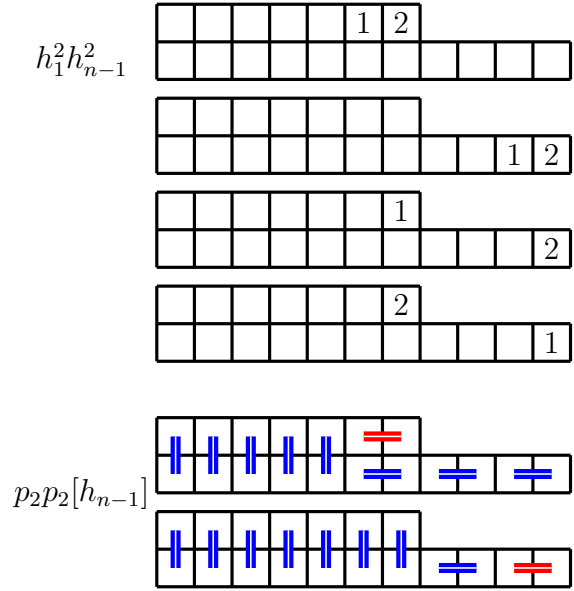


Thus the Schur functions in Case 4 contribute 
$$2 \chi(n \ \mathrm{is \ even}) s_{(n,n)} + 
3 \sum_{2 \leq a \leq n-1, a \ \mathrm{even} }s_{(a,2n-a)} 
+\sum_{1 \leq a \leq n-1, a \ \mathrm{odd} }s_{(a,2n-a)}$$ 
to $h_2[h_1 h_{n-1}]$. \\
\ \\
{\bf Case 5.} $\lambda = (2n)$. \\
\ \\
In this case, we have one way to obtain $s_{(2n)}$ from 
$h_1^2 h_{n-1}^2$, namely, it can arise in one way from the product 
$h_1^2 s_{(2n-2)}$.  There also is only one way 
to obtain  $s_{(2n)}$ from 
$p_2p_2[ h_{n-1}]$, namely, from the product $p_2s_{(2n-2)}$ which 
comes with a sign of $1$. Thus Subcase 5 contributes a 
factor of $s_{(2n)}$ to $h_2[h_1 h_{n-1}]$.

\end{proof}

Note that 
\begin{eqnarray*}
h_n[h_1h_1] &=& h_n[h_2+e_2] \\
&=& \sum_{k=0}^n h_k[h_2]h_{n-k}[e_2] \\
&=& \sum_{k=0}^n \left(\sum_{\lambda \ \mathrm{even}, \lambda \vdash 2k} 
s_\lambda \right) \left(\sum_{\lambda' \ \mathrm{even}, \lambda \vdash 2(n-k)} s_\lambda \right).
\end{eqnarray*}

 Our goal is to show that for all partitions of $2n$ except 
$\lambda = (2,2n-2)$, 
$$\langle h_n[h_1h_1], s_{\lambda} \rangle \geq  
\langle h_2[h_1h_{n-1}], s_{\lambda} \rangle.$$

It follows from our previous theorem that 
the only $s_\lambda$ such that $\langle h_2[h_1h_{n-1}], s_{\lambda} \rangle 
\neq 0$ is when $\lambda$ is either contained in 
$(2,2n,2n)$ or $(1^2,2n,2n)$. 
Now the only such $s_\lambda$ that appear in  
$\left(\sum_{\lambda \ \mathrm{even}, \lambda \vdash 2k} s_\lambda \right)$
are of the form 
\begin{enumerate}
\item $s_{(2,a,2k-a-2)}$ where $a$ is even and $2 \leq a \leq n-1$, 
\item $s_{(a,2k-a)}$ where $a$ is even and $2 \leq a \leq n$, or 
\item $s_{(2k)}$. 
\end{enumerate}
The only such $s_\lambda$ that appear in  
$\left(\sum_{\lambda' \ \mathrm{even}, \lambda \vdash 2(n-k)} s_\lambda \right)$
are of the form 
\begin{enumerate}
\item $s_{(1^2,n-k-1,n-k-1)}$ or  
\item $s_{(n-k,n-k)}$. 
\end{enumerate}

Our goal now is to go through the six cases of the proof of Theorem 
\ref{thm:1} to show that in all but one case, 
$\langle h_n[h_1h_1], s_{\lambda} \rangle \geq 
\langle h_2[h_1h_{n-1}], s_{\lambda} \rangle$. However, 
we will need to apply the skew Schur function expansion rule of Remmel and 
Whitney \cite{RW}. Given a skew shape, $\lambda/\mu$ written 
in French notation, the reverse lexicographic filling of 
$\lambda/\mu$, $RLF(\lambda/\mu)$, is the filling of the Ferrers 
diagram with the integers $1, \ldots, |\lambda/\mu|$, 
which is obtained by filling the cells in order from right to 
left in each row, starting at the bottom row and proceeding upwards. 
For example, the $RLF((2,4,4,4)/(1,2))$ is pictured in Figure 
\ref{fig:RLF}.

\begin{figure}[htp]
\centering
\begin{tikzpicture}[scale=.5]
\begin{scope}
\node at (0.5,3.5) {$11$};
\node at (1.5,3.5) {$10$};
\node at (0.5,2.5) {$9$};
\node at (1.5,2.5) {$8$};
\node at (2.5,2.5) {$7$};
\node at (3.5,2.5) {$6$};
\node at (1.5,1.5) {$5$};
\node at (2.5,1.5) {$4$};
\node at (3.5,1.5) {$3$};
\node at (2.5,0.5) {$2$};
\node at (3.5,0.5) {$1$};
\draw[very thick] (0,2) -- (0,4);
\draw[very thick] (1,1) -- (1,4);
\draw[very thick] (2,0) -- (2,4);
\draw[very thick] (3,0) -- (3,3);
\draw[very thick] (4,0) -- (4,3);
\draw[very thick] (0,4) -- (2,4);
\draw[very thick] (0,3) -- (4,3);
\draw[very thick] (0,2) -- (4,2);
\draw[very thick] (1,1) -- (4,1);
\draw[very thick] (2,0) -- (4,0);
\end{scope}
\end{tikzpicture}
\caption{The reverse lexicographic filling of $(2,4,4,4)/(1,2)$.} 
\label{fig:RLF}
\end{figure}


A standard tableau $T$ of shape $\nu$ is said to $\lambda/\mu$-compatible 
if and only if $T$ satisfies the following two conditions.
\begin{enumerate}
\item[{\bf R1.}] If $x$ and $x+1$ lie in the same row in  $RLF(\lambda/\mu)$,
then $x+1$ appear strictly to the right and weakly below 
$x$ in $T$. 

\item[{\bf R2.}] If $y$ is immediately above $x$ in  $RLF(\lambda/\mu)$, then 
$y$ must appear strictly above and weakly to the left of $x$ in $T$. 

\end{enumerate}

A pictorial representation of these two conditions are depicted in 
Figure \ref{fig:Tworules}. We let 
$\lambda/\mu$-$\mathcal{ST}$ denote the set of all 
$\lambda/\mu$-compatible standard tableaux.

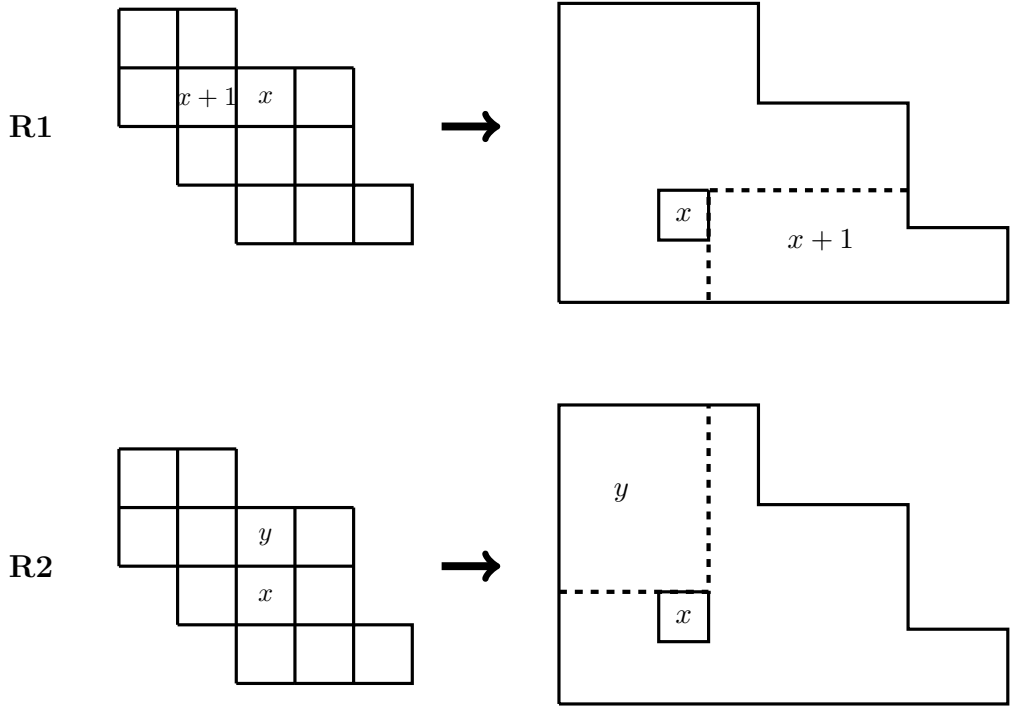
\begin{figure}[htp]
\centering
\begin{tikzpicture}[scale=.78]
\begin{scope}[yshift=3.75cm]
\begin{scope}[xshift=-5cm]
\node at (-1.5,2) {{\bf R1}}; 
\node at (1.5,2.5) {\footnotesize{$x+1$}};
\node at (2.5,2.5) {\footnotesize{$x$}};
\draw[very thick] (0,2) -- (0,4);
\draw[very thick] (1,1) -- (1,4);
\draw[very thick] (2,0) -- (2,4);
\draw[very thick] (3,0) -- (3,3);
\draw[very thick] (4,1) -- (4,3);
\draw[very thick] (0,4) -- (2,4);
\draw[very thick] (0,3) -- (4,3);
\draw[very thick] (0,2) -- (4,2);
\draw[very thick] (1,1) -- (4,1);
\draw[very thick] (2,0) -- (3,0);
\draw[very thick] (3,0) -- (5,0) -- (5,1) -- (4,1) -- (4,0);
\draw[->,line width=3pt] (5.5,2)  -- (6.5,2);
\end{scope}
\begin{scope}[xshift=2.5cm,yshift=-1cm,scale=.85]
\draw[very thick] (0,0) -- (9,0) -- (9,1.5) -- (7,1.5) -- (7,4) -- (4,4) -- (4,6) -- (0,6) -- (0,0);
\draw[ultra thick,dashed] (7,2.25) -- (3,2.25) -- (3, 0);
\draw[very thick] (3,2.25) -- (2,2.25) -- (2,1.25) -- (3,1.25) -- (3,2.25) ;
\node at (2.5,1.75) {\small{$x$}};
\node at (5.25,1.25) {\small{$x+1$}};
\end{scope}
\end{scope}
\begin{scope}[yshift=-3.75cm]
\begin{scope}[xshift=-5cm]
\node at (-1.5,2) {{\bf R2}}; 
\node at (2.5,1.5) {\footnotesize{$x$}};
\node at (2.5,2.5) {\footnotesize{$y$}};
\draw[very thick] (0,2) -- (0,4);
\draw[very thick] (1,1) -- (1,4);
\draw[very thick] (2,0) -- (2,4);
\draw[very thick] (3,0) -- (3,3);
\draw[very thick] (4,1) -- (4,3);
\draw[very thick] (0,4) -- (2,4);
\draw[very thick] (0,3) -- (4,3);
\draw[very thick] (0,2) -- (4,2);
\draw[very thick] (1,1) -- (4,1);
\draw[very thick] (2,0) -- (3,0);
\draw[very thick] (3,0) -- (5,0) -- (5,1) -- (4,1) -- (4,0);
\draw[->,line width=3pt] (5.5,2)  -- (6.5,2);
\end{scope}
\begin{scope}[xshift=2.5cm,yshift=-.35cm,scale=.85]
\draw[very thick] (0,0) -- (9,0) -- (9,1.5) -- (7,1.5) -- (7,4) -- (4,4) -- (4,6) -- (0,6) -- (0,0);
\draw[ultra thick,dashed] (0,2.25) -- (3,2.25) -- (3,6);
\draw[very thick] (3,2.25) -- (2,2.25) -- (2,1.25) -- (3,1.25) -- (3,2.25) ;
\node at (2.5,1.75) {\small{$x$}};
\node at (1.25,4.25) {\small{$y$}};
\end{scope}
\end{scope}
\end{tikzpicture}
\caption{The two rules for $\lambda/\mu$-compatible standard tableaux.}
\label{fig:Tworules}
\end{figure}


Then the {\em skew Schur function expansion rule} is that 
$$s_{\lambda/\mu} = \sum_{T \in \lambda/\mu\mbox{-}\mathcal{ST}} 
s_{sh(T)}$$
where $sh(T)$ is the shape of $T$. 
For example, if $\lambda = (2,2,3,4)$ and $\mu = (1,2,2)$, 
then in Figure \ref{fig:LRrule}, we have pictured 
the $RLF(\lambda/\mu)$ at the top right and we have constructed 
a tree which generates all $T$ in $\lambda/\mu$-$\mathcal{ST}$ 
by adding a new element at each stage in such a way 
that we  obey rules R1 and R2. 
It follows that in this case 
\begin{eqnarray*}
s_{(2,2,3,4)/(1,2,2)}&=& s_{(1,1,2,2)}+s_{(1,1,1,3)} + s_{(2,2,2)} + 
2s_{(1,2,3)}+ \\
&&s_{(1,1,4)}+s_{(3,3)}+s_{(2,4)}.
\end{eqnarray*}

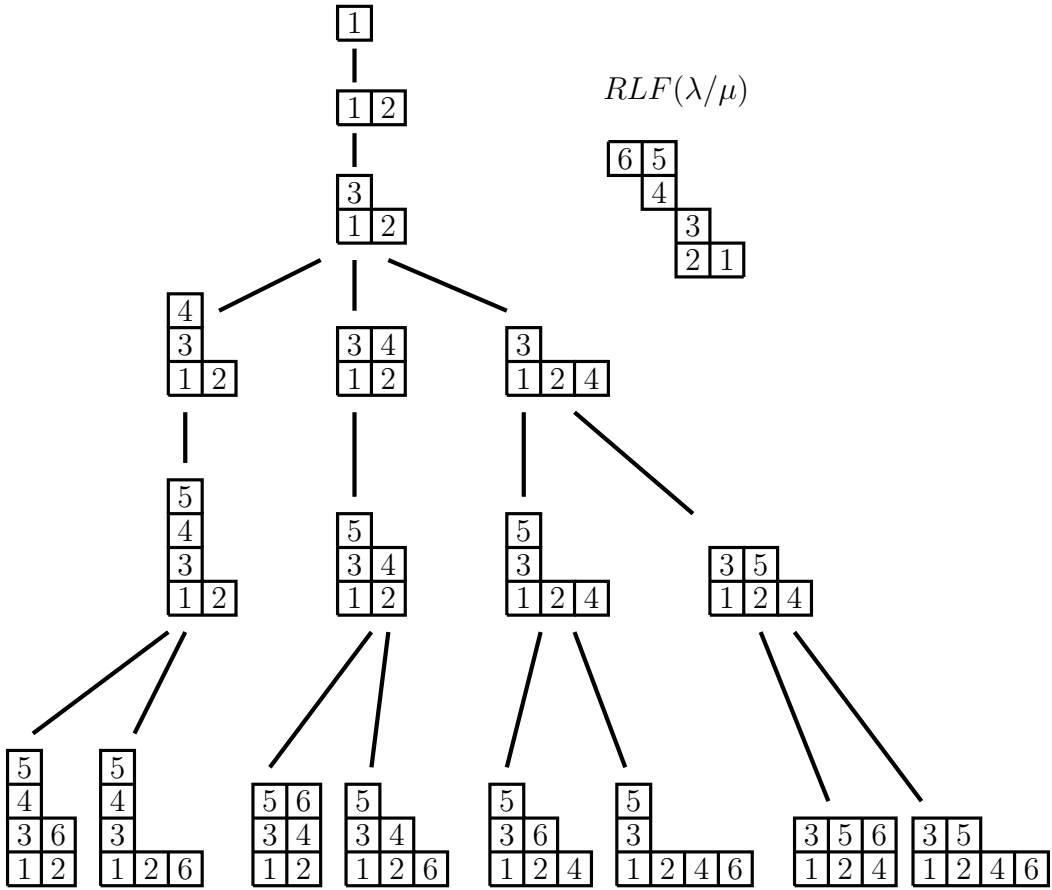
\begin{figure}[htp]
\centering
\begin{tikzpicture}[scale=.45]
\begin{scope}[yshift=8cm,xshift=8cm]
\node at (2,5.5) {$RLF(\lambda/\mu)$};
\draw[very thick] (2,1)--(4,1)--(4,0);
\draw[very thick] (0,4)--(0,3)--(2,3);
\draw[very thick] (0,4)--(2,4)--(2,0)--(4,0);
\draw[very thick] (1,4)--(1,2)--(3,2)--(3,0);
\node at (.5,3.5) {6};
\node at (1.5,3.5) {5};
\node at (1.5,2.5) {4};
\node at (2.5,1.5) {3};
\node at (2.5,0.5) {2};
\node at (3.5,0.5) {1};
\end{scope}
\begin{scope}[yshift=15cm]
\draw[very thick] (0,0)--(1,0)--(1,1)--(0,1)--(0,0);
\node at (.5,.5) {1};
\end{scope}
\begin{scope}[yshift=13.75cm]
\draw[ultra thick] (0.5,0)--(0.5,1);
\end{scope}
\begin{scope}[yshift=12.5cm]
\draw[very thick] (0,0)--(2,0)--(2,1)--(0,1)--(0,0);
\draw[very thick] (1,0)--(1,1);
\node at (.5,.5) {1};
\node at (1.5,.5) {2};
\end{scope}
\begin{scope}[yshift=11.25cm]
\draw[ultra thick] (0.5,0)--(0.5,1);
\end{scope}
\begin{scope}[yshift=9cm]
\draw[very thick] (0,0)--(2,0)--(2,1)--(1,1)--(1,2)--(0,2)--(0,0);
\draw[very thick] (1,0)--(1,1)--(0,1);
\node at (.5,.5) {1};
\node at (1.5,.5) {2};
\node at (.5,1.5) {3};
\end{scope}
\begin{scope}[yshift=7cm]
\draw[ultra thick] (0.5,0)--(0.5,1.5);
\end{scope}
\begin{scope}[yshift=7cm,xshift=3.5cm]
\draw[ultra thick] (1.5,0)--(-2,1.5);
\end{scope}
\begin{scope}[yshift=7cm,xshift=-2.5cm]
\draw[ultra thick] (-1,0)--(2,1.5);
\end{scope}
\begin{scope}[yshift=4.5cm,xshift=-5cm]
\draw[very thick] (0,0)--(2,0)--(2,1)--(1,1)--(1,2)--(0,2)--(0,0);
\draw[very thick] (1,0)--(1,1)--(0,1);
\draw[very thick] (0,2)--(0,3)--(1,3)--(1,2);
\node at (.5,.5) {1};
\node at (1.5,.5) {2};
\node at (.5,1.5) {3};
\node at (.5,2.5) {4};
\end{scope}
\begin{scope}[yshift=3cm,xshift=-5cm]
\draw[ultra thick] (0.5,-.5)--(0.5,1);
\end{scope}
\begin{scope}[yshift=4.5cm]
\draw[very thick] (0,0)--(2,0)--(2,1)--(1,1)--(1,2)--(0,2)--(0,0);
\draw[very thick] (1,0)--(1,1)--(0,1);
\draw[very thick] (2,1)--(2,2)--(1,2);
\node at (.5,.5) {1};
\node at (1.5,.5) {2};
\node at (.5,1.5) {3};
\node at (1.5,1.5) {4};
\end{scope}
\begin{scope}[yshift=3cm,xshift=-0cm]
\draw[ultra thick] (0.5,-1.5)--(0.5,1);
\end{scope}
\begin{scope}[yshift=4.5cm,xshift=5cm]
\draw[very thick] (0,0)--(2,0)--(2,1)--(1,1)--(1,2)--(0,2)--(0,0);
\draw[very thick] (1,0)--(1,1)--(0,1);
\draw[very thick] (2,0)--(3,0)--(3,1)--(2,1);
\node at (.5,.5) {1};
\node at (1.5,.5) {2};
\node at (.5,1.5) {3};
\node at (2.5,.5) {4};
\end{scope}
\begin{scope}[yshift=3cm,xshift=5cm]
\draw[ultra thick] (0.5,-1.5)--(0.5,1);
\end{scope}
\begin{scope}[yshift=3cm,xshift=6.5cm]
\draw[ultra thick] (0.5,1)--(4,-2);
\end{scope}
\begin{scope}[yshift=-2cm,xshift=-5cm]
\draw[very thick] (0,0)--(2,0)--(2,1)--(1,1)--(1,2)--(0,2)--(0,0);
\draw[very thick] (1,0)--(1,1)--(0,1);
\draw[very thick] (0,2)--(0,3)--(1,3)--(1,2);
\draw[very thick] (0,3)--(0,4)--(1,4)--(1,3);
\node at (.5,.5) {1};
\node at (1.5,.5) {2};
\node at (.5,1.5) {3};
\node at (.5,2.5) {4};
\node at (.5,3.5) {5};
\end{scope}
\begin{scope}[yshift=-2cm,xshift=0cm]
\draw[very thick] (0,0)--(2,0)--(2,1)--(1,1)--(1,2)--(0,2)--(0,0);
\draw[very thick] (1,0)--(1,1)--(0,1);
\draw[very thick] (2,1)--(2,2)--(1,2);
\draw[very thick] (0,2)--(0,3)--(1,3)--(1,2);
\node at (.5,.5) {1};
\node at (1.5,.5) {2};
\node at (.5,1.5) {3};
\node at (.5,2.5) {5};
\node at (1.5,1.5) {4};
\end{scope}
\begin{scope}[yshift=-2cm,xshift=5cm]
\draw[very thick] (0,0)--(2,0)--(2,1)--(1,1)--(1,2)--(0,2)--(0,0);
\draw[very thick] (1,0)--(1,1)--(0,1);
\draw[very thick] (2,0)--(3,0)--(3,1)--(2,1);
\draw[very thick] (0,2)--(0,3)--(1,3)--(1,2);
\node at (.5,.5) {1};
\node at (1.5,.5) {2};
\node at (.5,1.5) {3};
\node at (.5,2.5) {5};
\node at (2.5,.5) {4};
\end{scope}
\begin{scope}[yshift=-2cm,xshift=11cm]
\draw[very thick] (0,0)--(2,0)--(2,1)--(1,1)--(1,2)--(0,2)--(0,0);
\draw[very thick] (1,0)--(1,1)--(0,1);
\draw[very thick] (2,0)--(3,0)--(3,1)--(2,1);
\draw[very thick] (1,2)--(2,2)--(2,1);
\node at (.5,.5) {1};
\node at (1.5,.5) {2};
\node at (.5,1.5) {3};
\node at (1.5,1.5) {5};
\node at (2.5,.5) {4};
\end{scope}
\begin{scope}[yshift=-3cm,xshift=-5.5cm]
\draw[ultra thick] (0.5,0.5)--(-3.5,-2.5);
\draw[ultra thick] (1,0.5)--(-.5,-2.5);
\end{scope}
\begin{scope}[yshift=-3cm,xshift=.5cm]
\draw[ultra thick] (0.5,0.5)--(-2.5,-3.5);
\draw[ultra thick] (1,0.5)--(.5,-3.5);
\end{scope}
\begin{scope}[yshift=-3cm,xshift=5.5cm]
\draw[ultra thick] (0.5,0.5)--(-.5,-3.5);
\draw[ultra thick] (1.5,0.5)--(3,-3.5);
\end{scope}
\begin{scope}[yshift=-3cm,xshift=12cm]
\draw[ultra thick] (0.5,0.5)--(2.5,-4.5);
\draw[ultra thick] (1.5,0.5)--(5.25,-4.5);
\end{scope}
\begin{scope}[yshift=-10cm,xshift=-9.75cm]
\draw[very thick] (0,0)--(2,0)--(2,1)--(1,1)--(1,2)--(0,2)--(0,0);
\draw[very thick] (1,0)--(1,1)--(0,1);
\draw[very thick] (0,2)--(0,3)--(1,3)--(1,2);
\draw[very thick] (0,3)--(0,4)--(1,4)--(1,3);
\draw[very thick] (1,2)--(2,2)--(2,1);
\node at (.5,.5) {1};
\node at (1.5,.5) {2};
\node at (.5,1.5) {3};
\node at (.5,2.5) {4};
\node at (.5,3.5) {5};
\node at (1.5,1.5) {6};
\end{scope}
\begin{scope}[yshift=-10cm,xshift=-7cm]
\draw[very thick] (0,0)--(3,0)--(3,1)--(1,1)--(1,4)--(0,4)--(0,0);
\draw[very thick] (0,3)--(1,3);
\draw[very thick] (0,2)--(1,2);
\draw[very thick] (0,1)--(1,1)--(1,0);
\draw[very thick] (2,0)--(2,1);
\node at (.5,.5) {1};
\node at (1.5,.5) {2};
\node at (.5,1.5) {3};
\node at (.5,2.5) {4};
\node at (.5,3.5) {5};
\node at (2.5,0.5) {6};
\end{scope}
\begin{scope}[yshift=-10cm,xshift=-2.5cm]
\draw[very thick] (0,0)--(0,3)--(2,3)--(2,0)--(0,0);
\draw[very thick] (1,0)--(1,3);
\draw[very thick] (0,1)--(2,1);
\draw[very thick] (0,2)--(2,2);
\node at (.5,.5) {1};
\node at (1.5,.5) {2};
\node at (.5,1.5) {3};
\node at (.5,2.5) {5};
\node at (1.5,1.5) {4};
\node at (1.5,2.5) {6};
\end{scope}
\begin{scope}[yshift=-10cm,xshift=.25cm]
\draw[very thick] (0,0)--(3,0)--(3,1)--(2,1)--(2,2) --(1,2) --(1,3)--(0,3)--(0,0);
\draw[very thick] (0,1)--(2,1)--(2,0);
\draw[very thick] (0,2)--(1,2)--(1,0);
\node at (.5,.5) {1};
\node at (1.5,.5) {2};
\node at (.5,1.5) {3};
\node at (.5,2.5) {5};
\node at (1.5,1.5) {4};
\node at (2.5,0.5) {6};
\end{scope}
\begin{scope}[yshift=-10cm,xshift=4.5cm]
\draw[very thick] (0,0)--(3,0)--(3,1)--(2,1)--(2,2) --(1,2) --(1,3)--(0,3)--(0,0);
\draw[very thick] (0,1)--(2,1)--(2,0);
\draw[very thick] (0,2)--(1,2)--(1,0);
\node at (.5,.5) {1};
\node at (1.5,.5) {2};
\node at (.5,1.5) {3};
\node at (.5,2.5) {5};
\node at (1.5,1.5) {6};
\node at (2.5,0.5) {4};
\end{scope}
\begin{scope}[yshift=-10cm,xshift=8.25cm]
\draw[very thick] (0,0)--(4,0)--(4,1)--(1,1)--(1,3) --(0,3)--(0,0);
\draw[very thick] (0,2)--(1,2);
\draw[very thick] (0,1)--(1,1)--(1,0);
\draw[very thick] (3,0)--(3,1);
\draw[very thick] (2,0)--(2,1);
\node at (.5,.5) {1};
\node at (1.5,.5) {2};
\node at (.5,1.5) {3};
\node at (.5,2.5) {5};
\node at (3.5,0.5) {6};
\node at (2.5,0.5) {4};
\end{scope}
\begin{scope}[yshift=-10cm,xshift=13.5cm]
\draw[very thick] (0,0)--(3,0)--(3,2)--(0,2)--(0,0);
\draw[very thick] (0,1)--(3,1);
\draw[very thick] (1,0)--(1,2);
\draw[very thick] (2,0)--(2,2);
\node at (.5,.5) {1};
\node at (1.5,.5) {2};
\node at (.5,1.5) {3};
\node at (1.5,1.5) {5};
\node at (2.5,.5) {4};
\node at (2.5,1.5) {6};
\end{scope}
\begin{scope}[yshift=-10cm,xshift=17cm]
\draw[very thick] (0,0)--(4,0)--(4,1)--(2,1)--(2,2)--(0,2)--(0,0);
\draw[very thick] (1,0)--(1,2);
\draw[very thick] (2,0)--(2,1)--(0,1);
\draw[very thick] (3,0)--(3,1);
\node at (.5,.5) {1};
\node at (1.5,.5) {2};
\node at (.5,1.5) {3};
\node at (1.5,1.5) {5};
\node at (2.5,.5) {4};
\node at (3.5,0.5) {6};
\end{scope}
\end{tikzpicture}
\caption{The tree for the expansion of $s_{(2,2,3,4)/(1,2,2)}$.}
\label{fig:LRrule}
\end{figure}


Note that expanding a product of Schur functions is a special case of 
expanding a skew Schur function.  That is, it easy to see that 
$s_{\lambda}s_{\mu}$ is equal to the skew Schur function 
$s_{\lambda * \mu}$ where $\lambda * \mu$ is the skew shape 
that arises by placing $\lambda$ above $\mu$ so that the 
bottom row of $\lambda$ starts just to the left of top row 
of $\mu$.  For example, the if $\lambda = (2,2)$ and $\mu =(1,2)$, 
then $\lambda * \mu$ is pictured at the top-right of Figure \ref{fig:LRrule2}.
The other thing to observe is that if $\mu = (\mu_1, \ldots, \mu_k)$ 
where $0 < \mu_1 \leq \cdots \leq \mu_k$, 
then in the expansion of 
$s_{\lambda * \mu}$, the numbers in the part of the reverse lexicographic 
filling of $\mu$ are completely determined. That is, 
$1, \ldots, \mu_k$ must appear in row 1, reading from left to right, 
the numbers $\mu_k +1, \ldots, \mu_k+ \mu_{k-1}$ must appear in 
row 2, reading from left to right, the numbers 
$\mu_k + \mu_{k-1} +1, \ldots, \mu_k+ \mu_{k-1}+\mu_{k-2}$ must appear 
in row three, etc.. This means that in constructing our tree, we will eventually 
reach a single filling of shape $\mu$ and then we will continue 
to build our tree using the numbers of the reverse lexicographic filling 
that corresponding to $\lambda$.  One can see an example of this phenomenon 
in Figure \ref{fig:LRrule} in processing the numbers 1, 2, and 3 corresponding 
to the shape $(1,2)$. This means that we might as will ignore 
the numbers in $\mu$ and simply build our tree starting with a the 
Ferrers diagram of $\mu$ and using the numbers $1,2, \ldots, |\lambda|$ 
to fill the part corresponding to 
$\lambda$ in reverse lexicographic filling of $\lambda * \mu$. An example of this 
process is pictured in Figure \ref{fig:LRrule2} for the 
Schur function expansion of $s_{(2,2)}s_{(1,2)}$. 
In this case, we see that 

$$s_{(2,2)}s_{(1,2)} = s_{(1,2,2,2)}+s_{(1,1,2,3)}+s_{(2,2,3)}+ s_{(1,3,3)}+ 
s_{(1,2,4)}+s_{(3,4)}.$$

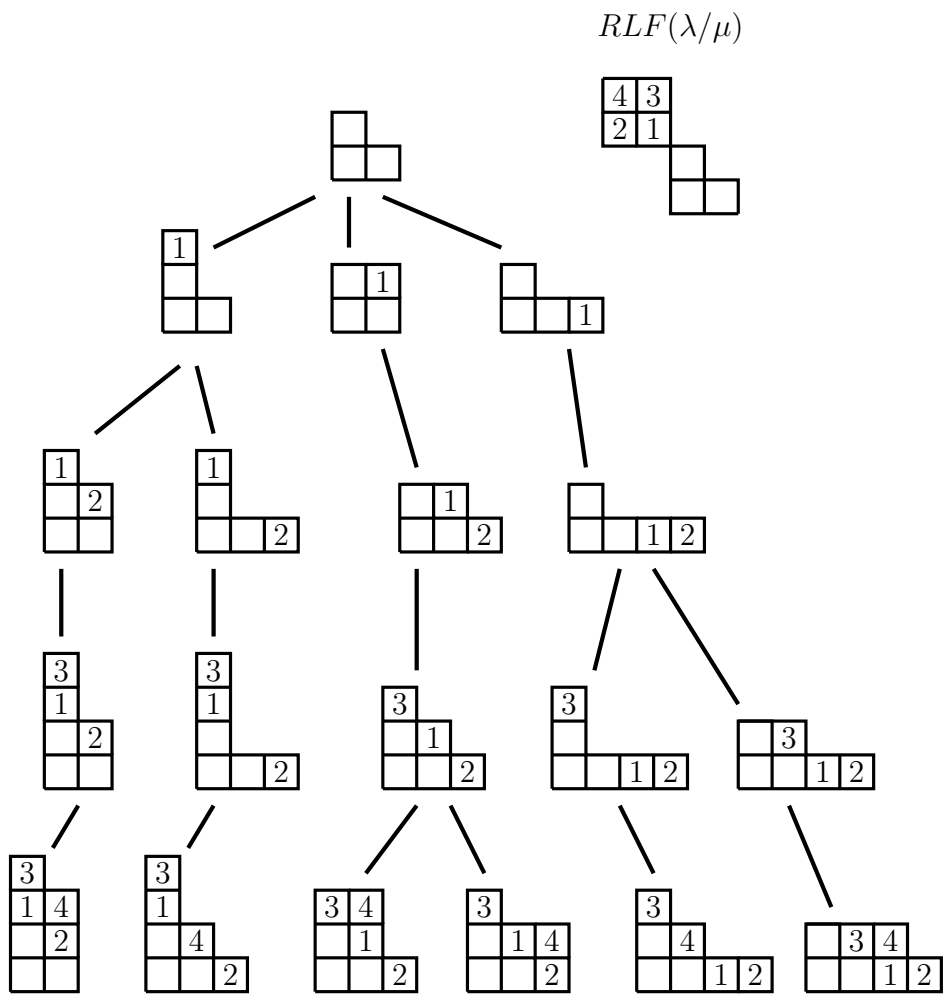
\begin{figure}[htp]
\centering
\begin{tikzpicture}[scale=.45]
\begin{scope}[yshift=8cm,xshift=8cm]
\node at (2,5.5) {$RLF(\lambda/\mu)$};
\draw[very thick] (2,1)--(4,1)--(4,0);
\draw[very thick] (0,4)--(0,3)--(2,3);
\draw[very thick] (0,3)--(0,2)--(1,2);
\draw[very thick] (0,4)--(2,4)--(2,0)--(4,0);
\draw[very thick] (1,4)--(1,2)--(3,2)--(3,0);
\node at (.5,3.5) {4};
\node at (1.5,3.5) {3};
\node at (.5,2.5) {2};
\node at (1.5,2.5) {1};
\end{scope}
\begin{scope}[yshift=9cm]
\draw[very thick] (0,0)--(2,0)--(2,1)--(1,1)--(1,2)--(0,2)--(0,0);
\draw[very thick] (1,0)--(1,1)--(0,1);
\end{scope}
\begin{scope}[yshift=7cm]
\draw[ultra thick] (0.5,0)--(0.5,1.5);
\end{scope}
\begin{scope}[yshift=7cm,xshift=3.5cm]
\draw[ultra thick] (1.5,0)--(-2,1.5);
\end{scope}
\begin{scope}[yshift=7cm,xshift=-2.5cm]
\draw[ultra thick] (-1,0)--(2,1.5);
\end{scope}
\begin{scope}[yshift=4.5cm,xshift=-5cm]
\draw[very thick] (0,0)--(2,0)--(2,1)--(1,1)--(1,2)--(0,2)--(0,0);
\draw[very thick] (1,0)--(1,1)--(0,1);
\draw[very thick] (0,2)--(0,3)--(1,3)--(1,2);
\node at (.5,2.5) {1};
\end{scope}
\begin{scope}[yshift=3cm,xshift=-4cm]
\draw[ultra thick] (-.5,0.5)--(-3,-1.5);
\end{scope}
\begin{scope}[yshift=3cm,xshift=-4cm]
\draw[ultra thick] (0,0.5)--(.5,-1.5);
\end{scope}
\begin{scope}[yshift=4.5cm]
\draw[very thick] (0,0)--(2,0)--(2,1)--(1,1)--(1,2)--(0,2)--(0,0);
\draw[very thick] (1,0)--(1,1)--(0,1);
\draw[very thick] (2,1)--(2,2)--(1,2);
\node at (1.5,1.5) {1};
\end{scope}
\begin{scope}[yshift=3cm,xshift=1.5cm]
\draw[ultra thick] (0,1)--(1,-2.5);
\end{scope}
\begin{scope}[yshift=4.5cm,xshift=5cm]
\draw[very thick] (0,0)--(2,0)--(2,1)--(1,1)--(1,2)--(0,2)--(0,0);
\draw[very thick] (1,0)--(1,1)--(0,1);
\draw[very thick] (2,0)--(3,0)--(3,1)--(2,1);
\node at (2.5,.5) {1};
\end{scope}
\begin{scope}[yshift=3cm,xshift=6.5cm]
\draw[ultra thick] (0.5,1)--(1,-2.5);
\end{scope}
\begin{scope}[yshift=-2cm,xshift=-8.5cm]
\draw[very thick] (0,0)--(2,0)--(2,1)--(1,1)--(1,2)--(0,2)--(0,0);
\draw[very thick] (1,0)--(1,1)--(0,1);
\draw[very thick] (2,1)--(2,2)--(1,2);
\draw[very thick] (0,2)--(0,3)--(1,3)--(1,2);
\node at (.5,2.5) {1};
\node at (1.5,1.5) {2};
\end{scope}
\begin{scope}[yshift=-2cm,xshift=-4cm]
\draw[very thick] (0,0)--(3,0)--(3,1)--(1,1)--(1,3)--(0,3)--(0,0);
\draw[very thick] (1,0)--(1,1)--(0,1);
\draw[very thick] (2,0)--(2,1);
\draw[very thick] (0,2)--(1,2);
\node at (.5,2.5) {1};
\node at (2.5,0.5) {2};
\end{scope}
\begin{scope}[yshift=-2cm,xshift=2cm]
\draw[very thick] (0,0)--(2,0)--(2,1)--(1,1)--(1,2)--(0,2)--(0,0);
\draw[very thick] (1,0)--(1,1)--(0,1);
\draw[very thick] (2,0)--(3,0)--(3,1)--(2,1);
\draw[very thick] (1,2)--(2,2)--(2,1);
\node at (1.5,1.5) {1};
\node at (2.5,.5) {2};
\end{scope}
\begin{scope}[yshift=-2cm,xshift=7cm]
\draw[very thick] (0,0)--(2,0)--(2,1)--(1,1)--(1,2)--(0,2)--(0,0);
\draw[very thick] (1,0)--(1,1)--(0,1);
\draw[very thick] (2,0)--(3,0)--(3,1)--(2,1);
\draw[very thick] (3,0)--(4,0)--(4,1)--(3,1);
\node at (2.5,.5) {1};
\node at (3.5,.5) {2};
\end{scope}
\begin{scope}[yshift=-3cm,xshift=-8cm]
\draw[ultra thick] (0,0.5)--(0,-1.5);
\end{scope}
\begin{scope}[yshift=-3cm,xshift=-4cm]
\draw[ultra thick] (0.5,0.5)--(0.5,-1.5);
\end{scope}
\begin{scope}[yshift=-3cm,xshift=2.5cm]
\draw[ultra thick] (0,0.5)--(0,-2.5);
\end{scope}
\begin{scope}[yshift=-3cm,xshift=7cm]
\draw[ultra thick] (1.5,0.5)--(.75,-2.5);
\draw[ultra thick] (2.5,0.5)--(5,-3.5);
\end{scope}
\begin{scope}[yshift=-9cm,xshift=-8.5cm]
\draw[very thick] (0,0)--(2,0)--(2,1)--(1,1)--(1,2)--(0,2)--(0,0);
\draw[very thick] (1,0)--(1,1)--(0,1);
\draw[very thick] (0,2)--(0,3)--(1,3)--(1,2);
\draw[very thick] (0,3)--(0,4)--(1,4)--(1,3);
\draw[very thick] (1,2)--(2,2)--(2,1);
\node at (.5,2.5) {1};
\node at (1.5,1.5) {2};
\node at (0.5,3.5) {3};
\end{scope}
\begin{scope}[yshift=-9cm,xshift=-4cm]
\draw[very thick] (0,0)--(3,0)--(3,1)--(1,1)--(1,4)--(0,4)--(0,0);
\draw[very thick] (0,3)--(1,3);
\draw[very thick] (0,2)--(1,2);
\draw[very thick] (0,1)--(1,1)--(1,0);
\draw[very thick] (2,0)--(2,1);
\node at (.5,2.5) {1};
\node at (2.5,0.5) {2};
\node at (0.5,3.5) {3};
\end{scope}
\begin{scope}[yshift=-9cm,xshift=1.5cm]
\draw[very thick] (0,0)--(3,0)--(3,1)--(2,1)--(2,2) --(1,2) --(1,3)--(0,3)--(0,0);
\draw[very thick] (0,1)--(2,1)--(2,0);
\draw[very thick] (0,2)--(1,2)--(1,0);
\node at (1.5,1.5) {1};
\node at (2.5,.5) {2};
\node at (0.5,2.5) {3};
\end{scope}
\begin{scope}[yshift=-9cm,xshift=6.5cm]
\draw[very thick] (0,0)--(4,0)--(4,1)--(1,1)--(1,3) --(0,3)--(0,0);
\draw[very thick] (0,2)--(1,2);
\draw[very thick] (0,1)--(1,1)--(1,0);
\draw[very thick] (3,0)--(3,1);
\draw[very thick] (2,0)--(2,1);
\node at (2.5,.5) {1};
\node at (3.5,.5) {2};
\node at (0.5,2.5) {3};
\end{scope}
\begin{scope}[yshift=-9cm,xshift=12cm]
\draw[very thick] (0,0)--(4,0)--(4,1)--(2,1)--(2,2) --(0,2) --(0,0);
\draw[very thick] (0,1)--(2,1)--(2,0);
\draw[very thick] (0,2)--(1,2)--(1,0);
\draw[very thick] (3,0)--(3,1);
\node at (3.5,.5) {2};
\node at (2.5,.5) {1};
\node at (1.5,1.5) {3};
\end{scope}
\begin{scope}[yshift=-15cm,xshift=-9.5cm]
\draw[very thick] (0,0)--(2,0)--(2,1)--(1,1)--(1,2)--(0,2)--(0,0);
\draw[very thick] (1,0)--(1,1)--(0,1);
\draw[very thick] (0,2)--(0,3)--(1,3)--(1,2);
\draw[very thick] (0,3)--(0,4)--(1,4)--(1,3);
\draw[very thick] (1,2)--(2,2)--(2,1);
\draw[very thick] (1,3)--(2,3)--(2,2);
\node at (.5,2.5) {1};
\node at (1.5,1.5) {2};
\node at (0.5,3.5) {3};
\node at (1.5,2.5) {4};
\end{scope}
\begin{scope}[yshift=-15cm,xshift=-5.5cm]
\draw[very thick] (0,0)--(3,0)--(3,1)--(1,1)--(1,4)--(0,4)--(0,0);
\draw[very thick] (0,3)--(1,3);
\draw[very thick] (0,2)--(1,2);
\draw[very thick] (0,1)--(1,1)--(1,0);
\draw[very thick] (2,0)--(2,1);
\draw[very thick] (1,2)--(2,2)--(2,1);
\node at (.5,2.5) {1};
\node at (2.5,0.5) {2};
\node at (0.5,3.5) {3};
\node at (1.5,1.5) {4};
\end{scope}
\begin{scope}[yshift=-15cm,xshift=-.5cm]
\draw[very thick] (0,0)--(3,0)--(3,1)--(2,1)--(2,2) --(1,2) --(1,3)--(0,3)--(0,0);
\draw[very thick] (0,1)--(2,1)--(2,0);
\draw[very thick] (0,2)--(1,2)--(1,0);
\draw[very thick] (1,3)--(2,3)--(2,2);
\node at (1.5,1.5) {1};
\node at (2.5,.5) {2};
\node at (0.5,2.5) {3};
\node at (1.5,2.5) {4};
\end{scope}
\begin{scope}[yshift=-15cm,xshift=4cm]
\draw[very thick] (0,0)--(3,0)--(3,1)--(2,1)--(2,2) --(1,2) --(1,3)--(0,3)--(0,0);
\draw[very thick] (0,1)--(2,1)--(2,0);
\draw[very thick] (0,2)--(1,2)--(1,0);
\draw[very thick] (2,2)--(3,2)--(3,1);
\node at (1.5,1.5) {1};
\node at (2.5,1.5) {4};
\node at (2.5,.5) {2};
\node at (0.5,2.5) {3};
\end{scope}
\begin{scope}[yshift=-15cm,xshift=9cm]
\draw[very thick] (0,0)--(3,0)--(3,1)--(2,1)--(2,2) --(1,2) --(1,3)--(0,3)--(0,0);
\draw[very thick] (0,1)--(2,1)--(2,0);
\draw[very thick] (0,2)--(1,2)--(1,0);
\draw[very thick] (3,1)--(4,1)--(4,0)--(3,0);
\node at (1.5,1.5) {4};
\node at (2.5,.5) {1};
\node at (3.5,.5) {2};
\node at (0.5,2.5) {3};
\end{scope}
\begin{scope}[yshift=-15cm,xshift=14cm]
\draw[very thick] (0,0)--(4,0)--(4,1)--(2,1)--(2,2) --(0,2) --(0,0);
\draw[very thick] (0,1)--(2,1)--(2,0);
\draw[very thick] (0,2)--(1,2)--(1,0);
\draw[very thick] (3,0)--(3,1);
\draw[very thick] (2,2)--(3,2)--(3,1);
\node at (3.5,.5) {2};
\node at (2.5,.5) {1};
\node at (1.5,1.5) {3};
\node at (2.5,1.5) {4};
\end{scope}
\begin{scope}[yshift=-10cm,xshift=-8cm]
\draw[ultra thick] (.5,0.5)--(-.25,-.75);
\end{scope}
\begin{scope}[yshift=-10cm,xshift=-4cm]
\draw[ultra thick] (.5,0.5)--(-.25,-.75);
\end{scope}
\begin{scope}[yshift=-10cm,xshift=2cm]
\draw[ultra thick] (.5,0.5)--(-1,-1.5);
\draw[ultra thick] (1.5,0.5)--(2.5,-1.5);
\end{scope}
\begin{scope}[yshift=-10cm,xshift=8cm]
\draw[ultra thick] (.5,0.5)--(1.5,-1.5);
\end{scope}
\begin{scope}[yshift=-10cm,xshift=12cm]
\draw[ultra thick] (1.5,0.5)--(2.75,-2.5);
\end{scope}
\end{tikzpicture}
\caption{The tree for the expansion of $s_{(2,2)}s_{(1,2)}$.}
\label{fig:LRrule2}
\end{figure}

\ \\
\ \\
{\bf Case 1.}  $\lambda = (1^2,a,2n-2-a)$ where $a$ is odd and 
$1 \leq a \leq n-1$.\\
\ \\
In this case, we have shown that 
$\langle h_2[h_1h_{n-1}],s_{(1^2,a,2n-a -2)}\rangle =1$. 
It is easy see from the Pieri rule that  
$s_{(1^2,a,2n-a -2)}$ appears in the expansion of 
$s_{1^4}s_{(a-1,2n-2-a -1)}$ with coefficient 1. 
But the term $s_{(1^4)}$ appears in the expansion of 
$h_2[e_2]$ and the term $s_{(a-1,2n-2-a -1)}$ 
appears in the expansion of $h_{n-2}[h_2]$ when $a$ is odd. 
Thus when $a$ is odd, 
$$\langle h_n[h_1h_1],s_{(1^2,a,2n-a -2)}\rangle \geq 1 
\geq \langle h_2[h_1h_{n-1}],s_{(1^2,a,2n-a -2)}\rangle.$$\\
\ \\
{\bf Case 2.}  $\lambda = (2,a,2n-2-a)$ where $a$ is even and 
$2 \leq a \leq n-1$.\\
\ \\
In this case, we have shown that 
$\langle h_2[h_1h_{n-1}],s_{(2,a,2n-a -2)}\rangle =1$. 
Since $(2,a,2n-a -2)$ has even parts, $s_{(2,a,2n-a -2)}$
appears in the expansion of $h_n[h_2]$ so that 
$$\langle h_n[h_1h_1],s_{(2,a,2n-a -2)}\rangle \geq 1 
\geq \langle h_2[h_1h_{n-1}],s_{(2,a,2n-a -2)}\rangle.$$\\
\ \\
{\bf Case 3.}  $\lambda = (1,a,2n-1-a)$ where 
$1 \leq a \leq n-1$.\\
\ \\
First we consider the case where $a =1$. In this case, we have shown that 
$\langle h_2[h_1h_{n-1}],s_{(1^2,2n-2)}\rangle =1$. 
It is easy see from the Pieri rule that  
$s_{(1^2,2n-2)}$ appears in the expansion of 
$s_{1^2}s_{(2n-2)}$ with coefficient 1. 
But the term $s_{1^2}$ appears in the expansion of 
$h_1[e_2]$ and the term $s_{(2n-2)}$ 
appears in the expansion of $h_{n-1}[h_2]$.  
Thus  
$$\langle h_n[h_1h_1],s_{(1^2,2n-2)}\rangle \geq 1 
\geq \langle h_2[h_1h_{n-1}],s_{(1^2,2n-2)}\rangle.$$

When $2 \leq a \leq n-1$, we have shown that 
$\langle h_2[h_1h_{n-1}],s_{(1,a,2n-1-a)}\rangle =2$.
The first thing to observe is that when 
$a$ is even, then $s_{(a,2n-2-a)}$ appears in the expansion 
of $h_{n-1}[h_2]$ from which it follows that 
$s_{(1,a,2n-1-a)}$ will occur with multiplicity 1 in 
the Schur function expansion of $h_1[e_2]h_{n-1}[h_2]= s_{(1^2)}h_{n-1}[h_2]$.
Similarly, when 
$a$ is odd, then $s_{(a-1,2n-1-a)}$ appears in the expansion 
of $h_{n-1}[h_2]$ from which it follows that 
$s_{(1,a,2n-1-a)}$ will occur with multiplicity 1 in 
the Schur function expansion of $h_1[e_2]h_{n-1}[h_2] = s_{(1^2)}h_{n-1}[h_2]$. 
This situation is pictured at the top of Figure 
\ref{fig:case3l}.  If $2 \leq a \leq n-1$, 
then $s_{(1,a,2n-1-a)}$ will occur in 
the expansion of $s_{(2,2)}h_{n-2}[h_2]$ as is evidenced 
by the pictures in the middle row of Figure \ref{fig:case3l}. 
Note that in the case where $a=n-1$ and $a$ is odd, the 
expansion still works. It will just be the case that the 
4 is on top of the 1.  
Hence in this case we get another occurrence of 
$s_{(1,a,2n-1-a)}$ from the term 
$h_2[e_2]h_{n-2}[h_2]$.  This will show 
that for $2 \leq a \leq n-1$, 
$$\langle h_n[h_1h_1],s_{(1,a,2n-1-a)}\rangle \geq 2 
\geq \langle h_2[h_1h_{n-1}],s_{(1^2,2n-2)}\rangle.$$
In fact, in the case where 
$a =n-1$, we note that $s_{(n-1,n-1)}$ appears in the 
expansion of $h_{n-1}[e_2]$ which implies 
that we would get another copy of 
$s_{(1,n-1,n)}$ in the expansion of 
$h_1[h_2]h_{n-1}[e_2] = h_2 h_{n-1}[e_2]$ as pictured 
at the bottom of Figure \ref{fig:case3l}.

\begin{figure}[htp]
\centering
\begin{tikzpicture}[scale=.45]
\begin{scope}[yshift=15cm,xshift=-1.5cm]
\node at (0,0) {$a$ odd};
\end{scope}
\begin{scope}[yshift=15cm,xshift=8.5cm]
\node at (0,0) {$a$ even};
\end{scope}
\begin{scope}[yshift=11cm,xshift=-9cm]
\node at (0,0) {$s_{(1^2)} h_{n-1}[s_2]$};
\end{scope}
\begin{scope}[yshift=10cm,xshift=-5cm]
\node at (.5,2.5) {$2$};
\node at (2.5,1.5) {$1$};
\draw[very thick] (0,0)--(1,0)--(1,3)--(0,3)--(0,0);
\draw[very thick] (0,0)--(3,0)--(3,2)--(0,2);
\draw[very thick] (0,0)--(8,0)--(8,1)--(0,1);
\draw[very thick] (2,0)--(2,2);
\draw[very thick] (4,0)--(4,1);
\draw[very thick] (5,0)--(5,1);
\draw[very thick] (6,0)--(6,1);
\draw[very thick] (7,0)--(7,1);
\end{scope}
\begin{scope}[yshift=10cm,xshift=5cm]
\node at (.5,2.5) {$2$};
\node at (7.5,.5) {$1$};
\draw[very thick] (0,0)--(1,0)--(1,3)--(0,3)--(0,0);
\draw[very thick] (0,0)--(3,0)--(3,2)--(0,2);
\draw[very thick] (0,0)--(8,0)--(8,1)--(0,1);
\draw[very thick] (2,0)--(2,2);
\draw[very thick] (4,0)--(4,1);
\draw[very thick] (5,0)--(5,1);
\draw[very thick] (6,0)--(6,1);
\draw[very thick] (7,0)--(7,1);
\draw[very thick] (3,2)--(4,2)--(4,1);
\end{scope}
\begin{scope}[yshift=6cm,xshift=-9cm]
\node at (0,0) {$s_{(2^2)} h_{n-1}[s_2]$};
\end{scope}
\begin{scope}[yshift=5cm,xshift=-5cm]
\node at (6.5,.5) {$1$};
\node at (7.5,.5) {$2$};
\node at (.5,2.5) {$3$};
\node at (2.5,1.5) {$4$};
\draw[very thick] (0,0)--(1,0)--(1,3)--(0,3)--(0,0);
\draw[very thick] (0,0)--(3,0)--(3,2)--(0,2);
\draw[very thick] (0,0)--(8,0)--(8,1)--(0,1);
\draw[very thick] (2,0)--(2,2);
\draw[very thick] (4,0)--(4,1);
\draw[very thick] (5,0)--(5,1);
\draw[very thick] (6,0)--(6,1);
\draw[very thick] (7,0)--(7,1);
\end{scope}
\begin{scope}[yshift=5cm,xshift=5cm]
\node at (2.5,1.5) {$1$};
\node at (3.5,1.5) {$4$};
\node at (.5,2.5) {$3$};
\node at (7.5,.5) {$2$};
\draw[very thick] (0,0)--(1,0)--(1,3)--(0,3)--(0,0);
\draw[very thick] (0,0)--(3,0)--(3,2)--(0,2);
\draw[very thick] (0,0)--(8,0)--(8,1)--(0,1);
\draw[very thick] (2,0)--(2,2);
\draw[very thick] (4,0)--(4,1);
\draw[very thick] (5,0)--(5,1);
\draw[very thick] (6,0)--(6,1);
\draw[very thick] (7,0)--(7,1);
\draw[very thick] (3,2)--(4,2)--(4,1);
\end{scope}
\begin{scope}[yshift=1cm,xshift=-9cm]
\node at (0,0) {$s_{(2)} h_{n-1}[s_{(1^2)}]$};
\end{scope}
\begin{scope}[yshift=0cm,xshift=-5cm]
\node at (.5,2.5) {$1$};
\node at (5.5,.5) {$2$};
\draw[very thick] (0,0)--(1,0)--(1,3)--(0,3)--(0,0);
\draw[very thick] (0,0)--(5,0)--(5,2)--(0,2);
\draw[very thick] (5,0)--(6,0)--(6,1)--(5,1);
\draw[very thick] (0,1)--(6,1);
\draw[very thick] (2,0)--(2,2);
\draw[very thick] (3,0)--(3,2);
\draw[very thick] (4,0)--(4,2);
\end{scope}
\begin{scope}[yshift=0cm,xshift=5cm]
\node at (.5,2.5) {$1$};
\node at (6.5,.5) {$2$};
\draw[very thick] (0,0)--(1,0)--(1,3)--(0,3)--(0,0);
\draw[very thick] (0,0)--(6,0)--(6,2)--(0,2);
\draw[very thick] (6,0)--(7,0)--(7,1)--(6,1);
\draw[very thick] (0,1)--(7,1);
\draw[very thick] (2,0)--(2,2);
\draw[very thick] (3,0)--(3,2);
\draw[very thick] (4,0)--(4,2);
\draw[very thick] (5,0)--(5,2);
\end{scope}
\end{tikzpicture}
\caption{Possibilities for $s_{(1,a,2n-1-a)}$ in 
the expansion of $h_n[h_1h_1]$.}
\label{fig:case3l}
\end{figure}


\ \\
{{\bf Case 4.} $\lambda = (a,2n-a)$.\\
\ \\
In this case, we can give an explicit formula for 
$\langle h_n[h_1h_1],s_{(a,2n-a)}\rangle$. 
The key observation to make is that the 
only Schur function $s_\lambda$ which appears 
in the expansion of $h_{n-k}[e_2]$ where $\lambda$ is a 2-part 
partition is $s_{(n-k,n-k)}$. 
In addition, it is easy to see from our Skew Schur function 
expansion rule that the only $s_\lambda$ which appears 
in the expansion of a product 
$s_{(n-k,n-k)}s_{(a,b)}$ where $\lambda$ is a 2-part 
partition is $s_{(n-k+a,n-k+b)}$. That is, one can see 
that when multiplying a Schur function $s_{(a,b)}$ by 
any Schur function $s_{(c,c)}$, the only way to get a 
a shape with two parts is to put all the elements in 
first row of the reverse lexicographic filling corresponding to $(c,c)$ in the 
first row and all the elements in the second row of the reverse lexicographic filling 
corresponding to $(c,c)$ in the 
second row. This process is pictured in Figure \ref{fig:twoparts}.

\begin{figure}[htp]
\centering
\begin{tikzpicture}[scale=.45]
\begin{scope}[yshift=.5cm,xshift=-4cm]
\node at (.5,3.5) {Reverse lexicographic filling};
\end{scope}
\begin{scope}[yshift=2.5cm,xshift=4cm]
\node at (.5,3.5) {$6$};
\node at (1.5,3.5) {$5$};
\node at (2.5,3.5) {$4$};
\node at (.5,2.5) {$3$};
\node at (1.5,2.5) {$2$};
\node at (2.5,2.5) {$1$};
\draw[very thick] (0,2)--(0,4)--(3,4)--(3,0)--(10,0)--(10,1)--(7,1)--(7,2)--(0,2);
\draw[very thick] (1,2) -- (1,4);
\draw[very thick] (2,2) -- (2,4);
\draw[very thick] (4,0) -- (4,2);
\draw[very thick] (5,0) -- (5,2);
\draw[very thick] (6,0) -- (6,2);
\draw[very thick] (8,0) -- (8,1);
\draw[very thick] (9,0) -- (9,1);
\draw[very thick] (0,3) -- (3,3);
\draw[very thick] (3,1) -- (7,1)--(7,0);
\end{scope}
\begin{scope}[yshift=-2.5cm]
\node at (4.5,1.5) {$4$};
\node at (5.5,1.5) {$5$};
\node at (6.5,1.5) {$6$};
\node at (7.5,.5) {$1$};
\node at (8.5,.5) {$2$};
\node at (9.5,.5) {$3$};
\draw[very thick] (0,0)--(10,0)--(10,1)--(0,1)--(0,0);
\draw[very thick] (0,1)--(0,2)--(7,2)--(7,0);
\draw[very thick] (1,0)--(1,2);
\draw[very thick] (2,0)--(2,2);
\draw[very thick] (3,0)--(3,2);
\draw[very thick] (4,0)--(4,2);
\draw[very thick] (5,0)--(5,2);
\draw[very thick] (6,0)--(6,2);
\draw[very thick] (8,0)--(8,1);
\draw[very thick] (9,0)--(9,1);
\end{scope}
\end{tikzpicture}
\caption{Mutliplying $s_{(a,b)}$ by $s_{(c,c)}$.}
\label{fig:twoparts}
\end{figure}


It follows that if $\lambda =(a,2n-a)$, then we have  
\begin{eqnarray*}
\langle \sum_{k=0}^n h_k[h_2] h_{n-k}[e_2],s_{(a,2n-a)}\rangle 
&=&  \langle \sum_{k=0}^n h_k[h_2]s_{(n-k,n-k)},s_{(a,2n-a)}\rangle \\
&=&\sum_{k=0}^n \chi(n-k < a \ \& \ (a-(n-k),2n-a -(n-k)) \ 
\mathrm{is \ even}).
\end{eqnarray*}
Thus  
\begin{equation}
\langle\sum_{k=0}^n h_k[h_2]h_{n-k}[e_2],s_{(a,2n-a)}\rangle  = 
\begin{cases} \frac{a}{2}+1 & \mbox{if $a$ is even;} \\
\frac{a+1}{2} & \mbox{if $a$ is odd.}
\end{cases}
\end{equation}

It follows that if $a$ is odd and $1 \leq a \leq n-1$, 
then 
$$\langle h_n[h_1h_1], s_{(a,2n-a)}\rangle \geq 1 =  
\langle h_2[h_1h_{n-1}], s_{(a,2n-a)}\rangle.$$
If $a$ is even and $4 \leq a \leq n-1$, then 
$$\langle h_n[h_1h_1], s_{(a,2n-a)}\rangle \geq 3 =  
\langle h_2[h_1h_{n-1}], s_{(a,2n-a)}\rangle.
$$
However, if 
$a =2$, then $\langle h_n[h_1h_1], s_{(a,2n-a)}\rangle = 2$ 
while $\langle h_2[h_1h_{n-1}], s_{(a,2n-a)}\rangle =3$ so 
we do not have 
$$\langle h_n[h_1h_1], s_{(2,2n-2)}\rangle \geq 
\langle h_2[h_1h_{n-1}], s_{(2,2n-2)}\rangle.$$

The only other Schur function corresponding to 
a two row shape that we have not considered is 
$s_{(n,n)}$. In this case,
\begin{equation}
\langle h_2[h_1h_{n-1}], s_{(n,n)}\rangle =
\begin{cases} 0 & \mbox{if $a$ is odd;} \\
2 & \mbox{if $a$ is even}.
\end{cases}
\end{equation}

Note however, if $n$ is even, then $s_{(n,n)}$ occurs 
in the expansion of $h_n[h_2]$ and the expansion of 
$h_n[e_2]$ so that 
$$
\langle h_n[h_1h_1],s_{(n,n)}\rangle \geq 2= 
\langle h_2[h_1h_{n-1}], s_{(n,n)}\rangle.$$
\ \\
{\bf Case 5.} $\lambda = (2n)$. \\
\ \\
In this case, we proved that 
$\langle h_2[h_1h_{n-1}], s_{(2n)}\rangle=1$. But 
$s_{(2n)}$ clearly appears in the expansion of 
$h_n[h_2]$ so that 
$$
\langle h_n[h_1h_1],s_{(2n)}\rangle \geq 
\langle h_2[h_1h_{n-1}], s_{(2n)}\rangle.$$

\bibliography{References.bib}
\bibliographystyle{plain}
\end{document}